\documentclass[a4paper,11pt,reqno]{article}
\usepackage{geometry}
\usepackage[english]{babel}
\usepackage[utf8]{inputenc}
\usepackage{amsmath,amsthm}
\usepackage{amsfonts,amssymb}

\usepackage[colorlinks=true,linktocpage=true,linkcolor=black,citecolor=black]{hyperref}
\usepackage[capitalise]{cleveref} 

\usepackage{enumitem}
\setenumerate[1]{label=(\arabic*)} 
\setlist[itemize]{noitemsep}

\usepackage{footmisc} 

\usepackage{mathtools}
\RequirePackage{tikz-cd}
\RequirePackage{amssymb}
\usetikzlibrary{calc}
\usetikzlibrary{decorations.pathmorphing}

\tikzset{curve/.style={settings={#1},to path={(\tikztostart)
    .. controls ($(\tikztostart)!\pv{pos}!(\tikztotarget)!\pv{height}!270:(\tikztotarget)$)
    and ($(\tikztostart)!1-\pv{pos}!(\tikztotarget)!\pv{height}!270:(\tikztotarget)$)
    .. (\tikztotarget)\tikztonodes}},
    settings/.code={\tikzset{quiver/.cd,#1}
        \def\pv##1{\pgfkeysvalueof{/tikz/quiver/##1}}},
    quiver/.cd,pos/.initial=0.35,height/.initial=0}

\tikzset{tail reversed/.code={\pgfsetarrowsstart{tikzcd to}}}
\tikzset{2tail/.code={\pgfsetarrowsstart{Implies[reversed]}}}
\tikzset{2tail reversed/.code={\pgfsetarrowsstart{Implies}}}
\tikzset{no body/.style={/tikz/dash pattern=on 0 off 1mm}}

\usepackage{tikz,tikz-cd} 
\usetikzlibrary{calc}
\tikzset{curve/.style={settings={#1},to path={(\tikztostart)
    .. controls ($(\tikztostart)!\pv{pos}!(\tikztotarget)!\pv{height}!270:(\tikztotarget)$)
    and ($(\tikztostart)!1-\pv{pos}!(\tikztotarget)!\pv{height}!270:(\tikztotarget)$)
    .. (\tikztotarget)\tikztonodes}},
    settings/.code={\tikzset{quiver/.cd,#1}
        \def\pv##1{\pgfkeysvalueof{/tikz/quiver/##1}}},
    quiver/.cd,pos/.initial=0.35,height/.initial=0}

  \DeclareMathSymbol{:}{\mathpunct}{operators}{"3A}

  \newcommand{\uvar}{\mathord{\relbar}}

   \tikzset{cto/.style={>->}}
  
   \tikzset{fto/.style={->>}}

\newcommand{\Db}{\mathbb{D}} 
 
\newcommand{\Gb}{\mathbb{G}}

\swapnumbers

\theoremstyle{plain}
\newtheorem{theorem}{Theorem}[subsection]
\newtheorem{lemma}[theorem]{Lemma}
\newtheorem{prop}[theorem]{Proposition}
\newtheorem*{prop*}{Proposition}
\newtheorem{cor}[theorem]{Corollary}

\theoremstyle{definition}
\newtheorem{definition}[theorem]{Definition}

\newtheorem{remark}[theorem]{Remark}
\newtheorem{example}[theorem]{Example}

\newtheorem{construction}[theorem]{Construction}
\newtheorem{notation}[theorem]{Notation}

\newcommand\Psh[1]{\widehat{#1}}
\newcommand\tPsh[1]{\widehat{#1}^+}
\newcommand\relativtPsh[2]{\widehat{#1}^{+_{#2}}}
\newcommand\ctPsh[1]{\widehat{#1}^{++}}
\newcommand\relativctPsh[3]{\widehat{#1}^{+_{#2}+_{#3}}}

\DeclareMathOperator*{\Colim}{Colim}

\DeclareMathOperator*{\Lim}{Lim}

\DeclareMathOperator{\Hom}{Hom}

\DeclareMathOperator{\mSset}{m\Psh{\Delta}}
\DeclareMathOperator{\stratSset}{\tPsh{\Delta}}
\DeclareMathOperator{\bmSset}{bm\Psh{\Delta}}
\DeclareMathOperator{\bstratSset}{\ctPsh{\Delta}}

\DeclareMathOperator{\infcat}{\omega -Cat}
\DeclareMathOperator{\satinfcat}{\overline{\omega -Cat}}

\newcommand{\invamalg}{\mathbin{\rotatebox[origin=c]{180}{$\amalg$}}}
\newcommand{\invoslash}{\mathbin{\rotatebox[origin=c]{90}{$\oslash$}}}
\newcommand{\invperp}{\mathbin{\rotatebox[origin=c]{180}{$\perp$}}}
\newcommand{\costar }{\mathbin{\overset{co}{\star}}}
\newcommand{\fwedge}{\mathbin{\rotatebox[origin=c]{270}{$\gtrdot$}}}

\author{Félix Loubaton}
\title{Dualities in the complicial model of $\infty$-categories}
\begin{document}
\maketitle
\begin{abstract}
In this note, we study the connection between Gray  tensor product and suspension. We derive a characterization of weak equivalences as fully faithful and essentially surjective functors. We construct the $co$ duality, a weak involution that reverses the direction of even dimensional cells. We conclude by studying Grothendieck fibrations of complicial sets.
\end{abstract}
\section*{Introduction}
The complicial sets, defined by Verity, are one of the models of $\infty$-categories\footnote{$\infty$-category stands here for what is sometimes called a $(\infty,\infty)$-category in the literature.}. One of their advantages is to admit a simple definition of the  Gray tensor product, an important ingredient of the theory of higher categories. Being strongly linked to (strict) $\omega$-categories by the Street nerve, they are also a privileged framework for stating and proving results of strictness, as done in \cite{ozornova2020fundamental},\cite{ozornova2020gray} and \cite{ozornova2020suspension}. 
However, they do not interact \textit{a priori} well with the globular language. One of the striking examples is the difficulty of constructing a functor of the Joyal's $\Theta$ category  into complicial sets, a work that should however be done in \cite{ozornova2020suspension}.
The goal of this note is to show that with some computation, it is possible to have a globular point of view in this model. 

The technical ingredient on which this work is based, is a study of the interaction between the Gray tensor product with the directed interval and the suspension. It will be  easier to give an intuition in the world of (strict) $\omega$-categories.  We denote $\Db_1$ the directed interval, i.e the $\omega$-category generated by the $1$-graph $0\to 1$. If $X$ is any $\omega$-category, the suspension of $A$, noted $\Sigma A$, is the $\omega$-category having two $0$-objects, $\perp $ and $\invperp$, and such that 
$$\Hom_{\Sigma A}(\perp,\invperp) := A,~~~\Hom_{\Sigma A}(\invperp,\perp) := \emptyset,~~~\Hom_{\Sigma A}(\perp,\perp)=\Hom_{\Sigma A}(\invperp,\invperp):=\{id\}.$$
For example, if we denote $\Db_0$ the terminal $\omega$-category, we have $\Db_1 = \Sigma\Db_0$.
The Gray tensor product $\Db_1\otimes \Db_1$ can be pictured by:
\[\begin{tikzcd}
	00 & 01 \\
	10 & 11.
	\arrow[from=1-1, to=2-1]
	\arrow[from=2-1, to=2-2]
	\arrow[from=1-1, to=1-2]
	\arrow[from=1-2, to=2-2]
	\arrow["{f\otimes f}"{description}, Rightarrow, from=1-2, to=2-1]
\end{tikzcd}\]
The $\omega$-category $\Db_1\otimes\Db_1$ is then equal to the colimit of the following diagram: 
$$\Db_1\vee \Db_1\xleftarrow{(f\otimes 1)\circ_0(0\otimes f)} \Db_1\hookrightarrow \Db_2\hookleftarrow\Db_1\xrightarrow{(1\otimes f) \circ_0 (f\otimes 0) } \Db_1\vee \Db_1.$$
Now, let $\Db_2$ be the suspension of $\Db_1$. This is the $\omega$-category generated by the following $2$-graph: 
\[\begin{tikzcd}
	0 & 1.
	\arrow[""{name=0, anchor=center, inner sep=0}, "{f_0}", curve={height=-12pt}, from=1-1, to=1-2]
	\arrow[""{name=1, anchor=center, inner sep=0}, "{f_1}"', curve={height=12pt}, from=1-1, to=1-2]
	\arrow["\alpha", shorten <=3pt, shorten >=3pt, Rightarrow, from=0, to=1]
\end{tikzcd}\]
The Gray tensor product $\Db_1\otimes\Db_2$ can be pictured by:
\[\begin{tikzcd}
	00 & 01 & 00 & 01 \\
	10 & 11 & 10 & 11.
	\arrow[""{name=0, anchor=center, inner sep=0}, curve={height=-30pt}, from=1-1, to=1-2]
	\arrow[""{name=1, anchor=center, inner sep=0}, curve={height=6pt}, from=1-1, to=1-2]
	\arrow[from=1-1, to=2-1]
	\arrow[from=2-1, to=2-2]
	\arrow[""{name=2, anchor=center, inner sep=0}, from=1-2, to=2-2]
	\arrow["{f\otimes f_1}"{description}, Rightarrow, from=1-2, to=2-1]
	\arrow[""{name=3, anchor=center, inner sep=0}, from=1-3, to=2-3]
	\arrow[""{name=4, anchor=center, inner sep=0}, curve={height=-6pt}, from=2-3, to=2-4]
	\arrow[""{name=5, anchor=center, inner sep=0}, curve={height=30pt}, from=2-3, to=2-4]
	\arrow[from=1-3, to=1-4]
	\arrow[from=1-4, to=2-4]
	\arrow["{f\otimes f_0}"{description}, Rightarrow, from=1-4, to=2-3]
	\arrow["0\otimes\alpha"{description}, shorten <=5pt, shorten >=5pt, Rightarrow, from=0, to=1]
	\arrow["{1\otimes \alpha}"{description}, shorten <=5pt, shorten >=5pt, Rightarrow, from=4, to=5]
	\arrow["f\otimes\alpha", shorten <=6pt, shorten >=6pt, Rightarrow, from=2, to=3]
\end{tikzcd}\]
If we denote $\alpha_0 := (f\otimes 1)\circ_0 (0\otimes\alpha)$ and $\alpha_1 := (1\otimes \alpha)\circ_0 (f\otimes 0)$, the
 source of $f\otimes \alpha$ is  $(f\otimes f_1)\circ_1\alpha_0$ and it's target is $\alpha_1\circ_1 (f\otimes f_1)$. The $\omega$-category $\Db_1\otimes\Db_2$ is then equal to the colimit of the following diagram: 
 $$\Db_2\vee \Db_1\xleftarrow{\alpha_0} \Db_2\hookrightarrow \Sigma(\Db_1\otimes\Db_1)\hookleftarrow\Db_2\xrightarrow{\alpha_1} \Db_1\vee \Db_2.$$
We now want to find a formula that combines these two examples. If $A$ is a category, we denote $\triangledown$ the two whiskerings :
$$\triangledown:\Sigma A\to \Db_1\vee \Sigma A ~~~~\mbox{and}~~~~ \triangledown:\Sigma A \to \Sigma A\vee \Db_1.$$

\begin{theorem}
In the category of $\omega$-categories, there exists an isomorphism, natural in $A$, between $\Db_1\otimes \Sigma A$ and the colimit of the following diagram:
 $$\Sigma A\vee \Db_1\xleftarrow{\triangledown} \Sigma A \hookrightarrow \Sigma(\Db_1\otimes A)\hookleftarrow \Sigma A\xrightarrow{\triangledown} \Db_1\vee \Sigma A.$$
\end{theorem} 
It is natural to ask whether an instance of this formula exists in complicial sets. We also have a suspension is this category. Objects $\Delta[1]\fwedge \Sigma A$ and $\Sigma A\fwedge \Delta[1]$ are defined in \ref{subsection:wedge}, but for now, we can suppose that they are fibrant replacements of respectively $\Delta[1]\coprod_{\Delta[0]}\Sigma A$ and $\Sigma A\coprod_{\Delta[0]}\Delta[1]$.
They come along with morphisms that are analogue to whiskerings, and that we also note $\triangledown$: 
$$\triangledown:\Sigma A\to \Delta[1]\fwedge \Sigma A ~~~~\mbox{and}~~~~ \triangledown:\Sigma A \to \Sigma A\fwedge \Delta[1].$$ 
In this paper, we show the following theorem:
\begin{theorem}[\ref{subsubsection:First_formula}]
In the category of complicial sets, there exists a zigzag of weak equivalences, natural in $A$, between $\Delta[1]\otimes \Sigma A$ and the colimit of the following diagram:
 $$\Sigma A\fwedge \Delta[1]\xleftarrow{\triangledown} \Sigma A \hookrightarrow \Sigma(\Delta[1]\otimes A)\hookleftarrow \Sigma A\xrightarrow{\triangledown} \Delta[1]\fwedge \Sigma A.$$
\end{theorem}
We also have a similar formula for $\Sigma A\otimes \Delta[1]$.

\subsubsection*{Organisation of the paper}

The first section is a recollection of usual results and definitions on complicial sets.

In the second section, we prove the formulas mentioned above. 
We deduce other formulas, in particular some establishing an interaction between the join and the suspension.

In the third section, we use these formulas to characterize weak equivalences between $\infty$-categories as essentially surjective and fully faithful functors. We also give a criterium to verify that a natural transformation between two left Quillen adjoints is a pointwise  weak equivalence.

Complicial sets inherit from simplicial sets the duality "op". Morally, it is the endofunctor that reverses the direction of odd dimensional cells. In the fourth part, we then define two other dualities, the even duality, and the full duality. The first one reverses the direction of the even dimension cells, and the second one reverses the directions of all the cells.

In the fifth section, we give two definitions of Grothendieck fibrations, one more simplicial and the other more globular, and we show that they coincide. A more detailed study of these fibrations is the subject of a future work.

\subsubsection*{Acknowledgment}

The author would like to thank Viktoriya Ozornova and Martina Rovelli for sharing their unpublished results, which are essential for sections \ref{section:A criterium to be a weakly invertible transformation} and \ref{section:The even duality}. The author also thanks Denis-Charles Cisinski and Carlos Simpson for helpful discussions during the development of this project. Finally, we thank Marnie Valentini for her help with English.
 
\subsubsection*{Warning} 

This text is still in draft form. Although the author is convinced of the veracity of the results, the text probably still contains many typos, and some proof are not written in the clearest way, especially the more technical ones.  A new version should quickly replace this one.
\tableofcontents
\section{The complicial model}

\subsection{Model structure on marked simplicial sets}
\begin{definition}
A \textit{stratified simplicial set} is a pair $(X,tX)$ where  $X$ is a simplicial set and  $tX := \cup_{n>0}tX_n$
 a graded set such that for all $n\geq 1$, $tX_n$ is a subset of $X_n$ that includes all degenerate simplices. A simplex in $tX$ is called \textit{thin}. 
A \textit{stratified morphism} $f:(X,tX)\to (Y,tY)$ is the data of a morphism on the underlying simplicial set  such that $f(tX_n)\subset tY_n$.
The category of stratified simplicial sets is denoted $\stratSset$.  
\end{definition}

We can extend the join to stratified simplicial sets as follow: 

\begin{definition}
If $(X,tX)$ and $(Y,tY)$ are two stratified simplicial sets, we define $tX\star tY$ as the set of simplices of  $X\star Y$ of shape $x\star y$ where either $x$  or $y$ are thin. We then define 
$$(X,tX)\star (Y,tY) := (X\star Y, tX\star tY).$$
\end{definition}

\begin{definition}
A stratified monomorphism $f:X\to Y$ is 
\begin{enumerate}
\item  \textit{entire} if it is an identity on  underlying simplicial sets.
\item \textit{regular} if for every $n\geq 1$ the following diagram is a pullback:
\[\begin{tikzcd}
	{tX_n} & {X_n} \\
	{tY_n} & {Y_n}.
	\arrow[from=2-1, to=2-2]
	\arrow[from=1-2, to=2-2]
	\arrow[from=1-1, to=1-2]
	\arrow[from=1-1, to=2-1]
	\arrow["\lrcorner"{anchor=center, pos=0.125}, draw=none, from=1-1, to=2-2]
\end{tikzcd}\]
\end{enumerate}
\end{definition}

\begin{definition}
We define several stratified structures on $\Delta[n]$:
\begin{enumerate}
\item $\Delta[n]_t$. The top $n$-simplex is thin.
\item $\Delta^k[n]$. All simplices that include $\{k-1,k,k+1\}\cap[n]$ are thin. 
\item $(\Delta^k[n])'$. All simplices that include $\{k-1,k,k+1\}\cap[n]$, together with the $(k-1)$-face and the $(k+1)$ face are thin.
\item $(\Delta^k[n])''$. All simplices that include $\{k-1,k,k+1\}\cap[n]$, together with the $(k-1)$-face, the $k$-face and the $(k+1)$ face are thin.
\item $\Delta[3]^{eq}$. All simplices of dimension strictly higher than $2$, together with $[0,2]$ and $[1,3]$ are thin.
\item $\Delta[n]^\sharp$. All simplices are thin.
\end{enumerate}
\end{definition}

\begin{definition}[{\cite[Definition 1.19]{ozornova2020model}}]
An \textit{elementary anodyne extension} is one of the following:
\begin{enumerate}
\item The \textit{complicial horn inclusions} are the regular extensions:
$$\Lambda^k[n]\to \Delta^k[n],~n\geq 1,~ n\geq k\geq 0.$$
\item The \textit{complicial thinness extensions}:
$$(\Delta^k[n])'\to (\Delta^k[n])'',~n\geq 2,~ n\geq k\geq 0.$$
\item The \textit{saturation extensions}:
$$\Delta[n]\star\Delta[3]^{eq}\star\Delta[m]\to \Delta[n]\star\Delta[3]^{\sharp}\star\Delta[m],~ n,m\geq -1.$$
\end{enumerate}
The set of complicial horn inclusions is $\Lambda$ and the reunion of \textit{complicial thinness extensions} and of \textit{saturation extensions} is $S$.
\end{definition}

\begin{definition}	
An \textit{$\infty$-category} is a stratified set having the right lifting property against all elementary anodyne extensions. 
\end{definition}

\begin{definition}
A \textit{marked simplicial set} is a stratified simplicial set having the right lifting property against morphisms in $S$. In particular, all $\infty$-categories are marked. The category of marked simplicial sets is denoted $\mSset$. There is an adjunction:
\[\begin{tikzcd}
	\stratSset && \mSset.
	\arrow[""{name=0, anchor=center, inner sep=0}, "R", curve={height=-12pt}, from=1-1, to=1-3]
	\arrow[""{name=1, anchor=center, inner sep=0}, "i", curve={height=-12pt}, from=1-3, to=1-1]
	\arrow["\dashv"{anchor=center, rotate=-90}, draw=none, from=0, to=1]
\end{tikzcd}\]
The left adjoint $R$ sends a stratified simplicial set $(X,tX)$ on the marked simplicial set $(X,\overline{tX})$ where $\overline{tX}$ is the smaller stratification including $tX$ and making $(X,tX)$ a marked simplicial set.
\end{definition}

\begin{prop}[{\cite[Proposition 1.26]{ozornova2020model}}]
\label{prop:martina}
For any elementary anodyne extension $K\to L$ and any cofibration $M\to N$ the induced morphism:
$$K\times N\cup L\times M\to L\times N.$$
is  a sequence of pushouts along anodyne extensions.
\end{prop}

\begin{prop}
Using notations of appendix \ref{section:appendix1}, we have an equivalence of categories:
$$\mSset \cong \tPsh{\Delta}_{/\Gamma S}.$$
\end{prop}
\begin{proof}
We have to show that a stratified simplicial set is $S$-saturated if and only if it is $\Gamma S$-saturated.
Let $i:K\to L$ be a morphism in $S$ and $j:M\to N$ a cofibration. 
The proposition \ref{prop:martina} implies that $i \hat{\times} j$ is a sequence of pushouts along morphisms in $\Lambda\cup S$. This cofibration is the identity on underlying simplicial sets. This implies that this morphism is a sequence of pushouts along morphisms in $S$.
\end{proof}

The interval $I$ will be the marked simplicial set $\Delta[1]_t$. 
We use the notion of local model structure as defined in \ref{defi:local_model_structure}.

\begin{theorem}[Ozornova, Rovelli, Verity]
There exists a $(\Lambda\cup S,I)$-local model structure on $\stratSset$, and a $(\Lambda,I)$-local model structure on $\mSset$, such that the adjoint pair $(R,i)$ is a Quillen equivalence. Fibrant objects of these model structures are $\infty$-categories.
\end{theorem}
\begin{proof}
Theorem \ref{theo:local_model_structure_on_stratified_presheave} and \ref{theo:local_model_structure_on_saturated_stratified_presheave} induces the existence of a $(\Gamma (\Lambda\cup S),I)$-local model structure on $\stratSset$, and of a $(\Gamma T,I)$-local model structure on $\mSset$ such that $(R,i)$ is a Quillen equivalence. Eventually, proposition \ref{prop:martina} implies that the structure on $\stratSset$ is $(\Lambda\cup S,I)$-local, and the one on $\mSset$ is $(\Lambda,I)$-local.
\end{proof}

\begin{construction}
\label{cons:truncation functor}
Let $n$ be an integer, and $(X,tX)$ a marked simplicial set. We define $\tau_n(tX)$ as the reunion of $tX$ and all simplices of dimension superior or equal to $n$. This induces a functor, called the \textit{$n$-truncation}:
$$\begin{array}{rcll}
\tau_n :& \mSset&\mapsto &\mSset\\
 &(X,tX)&\mapsto &(X, \overline{\tau_n(tX)}).
\end{array}$$
For every anodyne extension $i:K\to L$, we have a pushout 
\[\begin{tikzcd}
	K & L \\
	{\tau_n(K)} & {\tau_n(L).}
	\arrow[from=1-1, to=2-1]
	\arrow[from=2-1, to=2-2]
	\arrow[from=1-2, to=2-2]
	\arrow[from=1-1, to=1-2]
	\arrow["\lrcorner"{anchor=center, pos=0.125, rotate=180}, draw=none, from=2-2, to=1-1]
\end{tikzcd}\]
The $n$-truncation is then a left Quillen functor.
\end{construction}

\subsection{Gray tensor product}

\begin{construction}[{\cite[Notation 5]{verity2008weak}}] For any $n,p,q\geq 0$ such that $n=p+q$, we define:
\begin{itemize}
\item the \textit{degeneration partition operator:}
$$
\begin{array}{rclllrrclll}
\invamalg^1_{p,q}:&[n]&\to&[p]&&~~~~~~&\invamalg^2_{p,q}:&[n]&\to&[q]&\\
&k&\mapsto &k &\mbox{if $k\leq p$} &&&k&\mapsto &0& \mbox{if $k\leq p$}\\
&k&\mapsto &p 	&\mbox{if $k>p$} &&&k&\mapsto &k-p& \mbox{if $k> p$}.
\end{array}
$$
\item the \textit{face partition operator:}
$$
\begin{array}{rcllrrcll}
\amalg^1_{p,q}:&[p]&\to&[n]&~~~~~~&\amalg^2_{p,q}:&[q]&\to&[n]\\
&k&\mapsto &k &&&k&\mapsto &k+p.
\end{array}
$$
\end{itemize}
\end{construction}

\begin{definition}[{\cite[Definition 125]{verity2008weak}}]
Let $(X,tX)$ and $(Y,tY)$ be two stratified simplicial sets. 
We define the \textit{Gray tensor product} of $(X,tX)$ and $(Y,tY)$ as the stratified simplicial set 
$(X\times Y,tX\otimes tY)$ where $tX\otimes tY$ is the set of pairs $(x,y)$ such that for all partitions $(p,q)$ of $n$ either  $\amalg^1_{p,q}x$ or $\amalg^2_{p,q}y$ is thin. 
\end{definition}

\begin{remark}
Let $X,Y$ be two stratified simplicial sets such that all simplices of $X$ are thin. The functor 
$$X\otimes Y\to X\times Y$$ is then an isomorphism.
\end{remark}

In \cite{verity2008weak}, it is shown that the Gray tensor is associative. The problem of this operation comes from the fact that it doesn't commute with colimits. Verity then defines an other binary operation,  which is cocontinuous, \textit{the Gray pretensor} $(X,M)\boxtimes(Y,N):=(X\times Y, M\boxtimes N)$, together with a natural transformation: 
$$\uvar\boxtimes\uvar\to \uvar\otimes\uvar$$
that is a pointwise sequence of pushouts along morphisms in $S$. Moreover, in \cite{ozornova2020gray}, it is shown that this pretensor is a Quillen bifunctor for the model category on $\stratSset$. 

\begin{definition}[Gray tensor product for marked simplicial sets]
If $(X,tX)$ and $(Y,tY)$ be two marked simplicial sets.  We define the \textit{Gray tensor product} of $(X,tX)$ and $(Y,tY)$ as the marked simplicial set 
$(X\times Y,\overline{tX\otimes tY})$.
Remark that we have an equality: 
$(X\times Y,\overline{tX\otimes tY})=(X\times Y,\overline{tX\boxtimes tY})$.
\end{definition}

\begin{prop}
\label{prop:R_commutes_with_gray_tensor}
The functor $R$ commutes with Gray pretensor product. 
\end{prop}
\begin{proof}
One have to show that $(X\times Y, \overline{M\boxtimes N}) = (X\times Y, \overline{\overline{M}\boxtimes \overline{N}})$.
First of all we have a morphism $h:(X\times Y, \overline{M\boxtimes N})\to (X\times Y, \overline{\overline{M}\boxtimes \overline{N}})$  as a lift in the following diagram: 
\[\begin{tikzcd}
	{(X\times Y, M\boxtimes  N)} & {(X\times Y, \overline{\overline{M}\boxtimes  \overline{N}})} \\
	{(X\times Y,\overline{M\boxtimes  N}).}
	\arrow[from=1-1, to=2-1]
	\arrow[from=1-1, to=1-2]
	\arrow["h"', dotted, from=2-1, to=1-2]
\end{tikzcd}\]
Secondly,  \cite[Theorem 2.1]{ozornova2020gray} implies that the morphism $(X\times Y, M\boxtimes N)\to (X\times Y, \overline{M}\boxtimes \overline{N})$ is a $S$-anodyne extension. We then have liftings in the following diagram:
\[\begin{tikzcd}
	{(X\times Y, M\boxtimes  N)} & {(X\times Y,\overline{M\boxtimes  N})} \\
	{(X\times Y,\overline{M}\boxtimes\overline{N})} \\
	{(X\times Y, \overline{\overline{M}\boxtimes  \overline{N}}).}
	\arrow[from=1-1, to=1-2]
	\arrow["k"', dotted, from=3-1, to=1-2]
	\arrow[from=1-1, to=2-1]
	\arrow[from=2-1, to=3-1]
	\arrow[dotted, from=2-1, to=1-2]
\end{tikzcd}\]
Both $kh$ and $hk$ are identities on  underlying simplicial sets, and so are equal to identity. 
\end{proof}

We can then deduce the following proposition:
\begin{prop}
\label{prop:gray_product_is_a_left_Quillen_bifunctor}
The Gray tensor product is a left Quillen bifunctor in $\mSset$.
\end{prop}
\begin{proof}
The Gray tensor product of marked simplicial sets is cocontinuous and preserves cofibrations. Furthermore, \cite[Corollary 2.3]{ozornova2020gray} and proposition \ref{prop:R_commutes_with_gray_tensor} implies that it preserves acyclic cofibrations in both variables. 
\end{proof}

\begin{construction}
Let $X$ be a simplicial set. We define \textit{the suspension of $X$} to be the colimit of the following diagram:
\[\begin{tikzcd}
	{X\otimes\partial\Delta[1]} & {X\otimes \Delta[1]} \\
	{\partial\Delta[1]} & {\Sigma X.}
	\arrow[from=1-2, to=2-2]
	\arrow["\lrcorner"{anchor=center, pos=0.125, rotate=180}, draw=none, from=2-2, to=1-1]
	\arrow[from=1-1, to=2-1]
	\arrow[from=2-1, to=2-2]
	\arrow[from=1-1, to=1-2]
\end{tikzcd}\]
This assignation defines a cocontinuous functor $\Sigma:\mSset\to \mSset_{\partial\Delta[1]/}.$ For every acyclic cofibration $K\to L$
, we have a colimit diagram: 
\[\begin{tikzcd}
	{L\otimes\partial\Delta[1]} & {K\otimes\Delta[1]\cup L\otimes\partial\Delta[1]} & {L\otimes\Delta[1]} \\
	{\partial\Delta[1]} & {\Sigma K} & {\Sigma L.}
	\arrow[from=1-1, to=2-1]
	\arrow[""{name=0, anchor=center, inner sep=0}, from=1-1, to=1-2]
	\arrow[from=2-1, to=2-2]
	\arrow[from=1-2, to=2-2]
	\arrow[from=1-3, to=2-3]
	\arrow[from=2-2, to=2-3]
	\arrow[""{name=1, anchor=center, inner sep=0}, from=1-2, to=1-3]
	\arrow["\lrcorner"{anchor=center, pos=0.125, rotate=180}, draw=none, from=2-2, to=0]
	\arrow["\lrcorner"{anchor=center, pos=0.125, rotate=180}, draw=none, from=2-3, to=1]
\end{tikzcd}\]
This shows that the suspension is a left Quillen functor.

Furthermore, this functor admits a right adjoint, that sends a pair $(a,b,C)$ to $C(a,b)$ where $a,b$ are two $0$-simplices of $C$. If $p:C\to D$ is a morphism between $\infty$-categories, and $a,b$ two $0$-simplices of $C$, we denote 
$$p(a,b):C(a,b)\to D(pa,pb)$$
the induced morphism.
\end{construction}

\subsection{Diamond operation}
We introduce an other operation, the diamond product, that makes the link between the Gray tensor product and the join.

\begin{construction}
Let $X$ and $Y$ be two marked simplicial sets. We define $X\diamond Y$ as the colimit in the following diagram:
\[\begin{tikzcd}
	{X\otimes \{0\}\otimes Y} & {X\otimes\Delta[1]\otimes Y} & {X\otimes \{1\}\otimes Y} \\
	X & {X\diamond Y} & Y.
	\arrow[from=1-3, to=1-2]
	\arrow[from=1-2, to=2-2]
	\arrow[from=2-3, to=2-2]
	\arrow[from=1-3, to=2-3]
	\arrow[from=1-1, to=2-1]
	\arrow[from=2-1, to=2-2]
	\arrow[from=1-1, to=1-2]
\end{tikzcd}\]
Functors 
$$\uvar\diamond X:\mSset\to \mSset_{/X} ~~~~ X\diamond \uvar:\mSset\to \mSset_{/X}.$$
are colimit preserving. Furthermore, for every acyclic cofibration $K\to L$, we have a diagram: 
\[\begin{tikzcd}
	{K\amalg X} & {K\otimes \partial\Delta[1]\otimes X} & {K\otimes \Delta[1]\otimes X} \\
	{L\amalg X} & {L\otimes \partial\Delta[1]\otimes X} & {L\otimes \Delta[1] \otimes X.}
	\arrow[from=1-2, to=1-1]
	\arrow[from=1-2, to=1-3]
	\arrow[from=2-2, to=2-3]
	\arrow[from=2-2, to=2-1]
	\arrow[from=1-2, to=2-2]
	\arrow[from=1-3, to=2-3]
	\arrow[from=1-1, to=2-1]
\end{tikzcd}\]
Taking colimit of lines induces a morphism
$$K\diamond X\to L\diamond X.$$
However, these two colimits are homotopy colimits, and all the horizontal maps of the previous diagram are weak equivalences. This morphism is then an acyclic cofibration. This shows that 
 $\uvar\diamond X$ is   a left Quillen functor. We show analogously that $X\diamond \uvar$ is a left Quillen functor.
These functors then admit right Quillen adjoints, that  are denoted:
$$(\uvar)_{//X}:\mSset_{/X}\to  \mSset~~~~ (\uvar)_{X//}:\mSset_{/X}\to  \mSset.$$
\end{construction}

\begin{lemma}
There exists a unique natural transformation $\gamma_{X,Y}:X\diamond Y\to X\star  Y$ that fits in the following diagram: 
\[\begin{tikzcd}
	{X\coprod Y} & {X\star  Y} \\
	{X\diamond Y} & {\Delta[1].}
	\arrow[from=1-1, to=2-1]
	\arrow[from=1-1, to=1-2]
	\arrow[from=1-2, to=2-2]
	\arrow[from=2-1, to=2-2]
	\arrow["{\gamma_{X,Y}}", from=2-1, to=1-2]
\end{tikzcd}\]
\end{lemma}
\begin{proof}
We begin by defining this morphism on simplicial sets, and for this we can suppose that both $X$ and $Y$ are representables, ie $X:=\Delta[n]$, $Y:=\Delta[m]$.
On object, this morphism is induced by the assignation:
$$p(k,0,l) := k~~~p(k,1,l) := l.$$ 

We need to verify that  this morphism preserves thin cells. Suppose now that $(x,v,y)$ is a thin $n$-simplex of $X\diamond Y$. There are several cases to consider. \textbf{Case $v_n=0$.} The simplex $x$ is then thin, and is sent to $x\star \emptyset$ which is also thin. \textbf{Case $v_0=1$.} Similar. \textbf{Case $v_0=0$ and $v_n=1$.} Let $p$ be the smaller integer such that $v_p=1$. Either $\amalg_{p-1,n-p+1}^1(x)$ or $\amalg_{p,n-p}^2(y)$ is thin. This implies that $\phi_{X,Y}(x,v,y)= \amalg_{p-1,n-p+1}^1(x)\star \amalg_{p,n-p}^2(y)$ is  thin. 
\end{proof}

\begin{prop}
\label{prop:equivalence between diamond and join product}
For all $X,Y$, the morphism $\gamma_{X,Y}$ is a weak equivalence. 
\end{prop}
\begin{proof}
The set of couples $(X,Y)$ such that $\gamma_{X,Y}$ is a weak equivalence is saturated by monomorphisms. It is then enough to show it for all the couples of representables.  

Let's start by the case $(X,Y)=(\Delta[n],\Delta[m])$. Let $s:X\star Y\to X\diamond Y$ be the morphism defined on objects by the formula: 
$$s(k\star \emptyset) = (k,0,0)~~~s(\emptyset \star l) = (n,1,l)$$
The  marking on $\Delta[n]\star\Delta[m]$ is trivial, $s$ is then a morphism of marked simplicial sets , we have
$$\gamma_{X,Y}s = id ~~~~s\gamma_{X,Y} (k,\epsilon,l) =(k + \epsilon (n-k), \epsilon,\epsilon l).$$

Let $\eta:\Delta[n]\diamond \Delta[m]\to \Delta[n]\diamond \Delta[m]$ be induced by the application
$$(k,\epsilon,l)\mapsto (k,\epsilon,\epsilon l).$$
We are now going to construct two transformations $\eta\Rightarrow s\gamma_{X,Y}$ and $\eta \Rightarrow id$.
The first is induced on the level of simplicial sets by
$$(k,\epsilon,l,\alpha)\mapsto (k + \alpha\epsilon (n-k),\epsilon,\epsilon l ),$$
and the second one by
$$(k,\epsilon,l,\alpha)\mapsto (k,\epsilon,(\epsilon\vee\alpha)l),$$
where $\epsilon\vee\alpha := \epsilon+\alpha - \epsilon\alpha.$
These two natural transformations extend to marked simplicial sets. 

To deal with cases $(X,Y) = (\Delta[n]_t,\Delta[m]), (\Delta[n],\Delta[m]_t)$ or $(\Delta[n]_t,\Delta[m]_t)$, we have to check that morphisms $s$, $\eta$, and the other natural transformations preserve thin cells. We left this computation to the reader. 
\end{proof}

In particular, that shows that for all marked simplicial sets $X$, functors $\uvar\star X$ and $X\star \uvar$ are left Quillen functors. Furthermore, for all $X$, $\gamma$ induces  weakly invertible comparison morphisms:
$$(\uvar)_{/X}\to (\uvar)_{//X},~~~~(\uvar)_{X/}\to (\uvar)_{X//}.$$

\subsection{Co-joint}

\begin{definition}
We define the \textit{co-joint} of $X$ and $Y$, denoted $X\costar  Y$  as the colimit in the following diagram:
\[\begin{tikzcd}
	{Y\otimes \{1\}\otimes X} & {Y\otimes \Delta[1]\otimes X} & {Y\otimes \{0\}\otimes X} \\
	Y & {X\costar  Y} & X.
	\arrow[from=1-1, to=2-1]
	\arrow[from=1-2, to=2-2]
	\arrow[from=1-1, to=1-2]
	\arrow[from=2-1, to=2-2]
	\arrow[from=1-3, to=2-3]
	\arrow[from=1-3, to=1-2]
	\arrow[from=2-3, to=2-2]
\end{tikzcd}\]
Functors 
$$\costar X:\mSset\to \mSset_{/X} ~~~~ X\costar \uvar:\mSset\to \mSset_{/X}.$$
are colimit preserving. Furthermore, for every acyclic cofibration $K\to L$, we have a diagram: 
\[\begin{tikzcd}
	{K\amalg X} & {X\otimes \partial\Delta[1]\otimes K} & {X\otimes \Delta[1]\otimes K} \\
	{L\amalg X} & {X\otimes \partial\Delta[1]\otimes L} & {X\otimes \Delta[1] \otimes K.}
	\arrow[from=1-2, to=1-1]
	\arrow[from=1-2, to=1-3]
	\arrow[from=2-2, to=2-3]
	\arrow[from=2-2, to=2-1]
	\arrow[from=1-2, to=2-2]
	\arrow[from=1-3, to=2-3]
	\arrow[from=1-1, to=2-1]
\end{tikzcd}\]
Taking colimit of lines induces a morphism
$$K\costar X\to L\costar X.$$
However, these two colimits are homotopy colimits, and all the horizontal maps of the previous diagram are weak equivalences. This morphism is then an acyclic cofibration.
This shows that $\uvar\costar X$ is  a left Quillen functor. We show analogously that $X\costar \uvar$ is a left Quillen functor.
\end{definition}

We have  $\Delta[0]\costar  \Delta[0] = \Delta[1]$.
However, this operation is not associative. For example, the two marked simplicial sets $K := \Delta[1]\costar \Delta[0]$ and $L:= \Delta[0]\costar \Delta[1]$ can be pictured by:	
\[\begin{tikzcd}
	00 & 01 && 00 & 01 \\
	10 & 11 && 10 & 11
	\arrow[from=1-1, to=1-2]
	\arrow["{L=}"', shorten <=2pt, shorten >=2pt, Rightarrow, no head, from=1-4, to=2-4]
	\arrow[""{name=0, anchor=center, inner sep=0}, from=1-1, to=2-2]
	\arrow[""{name=1, anchor=center, inner sep=0}, from=1-4, to=2-5]
	\arrow[from=1-5, to=2-5]
	\arrow[from=1-4, to=1-5]
	\arrow[shorten <=3pt, shorten >=3pt, from=2-4, to=2-5]
	\arrow["{K=}"', from=1-1, to=2-1]
	\arrow[from=1-2, to=2-2]
	\arrow[shorten <=3pt, shorten >=3pt, Rightarrow, no head, from=2-1, to=2-2]
	\arrow["\sim"{description}, Rightarrow, draw=none, from=0, to=1-2]
	\arrow["\sim"{description}, Rightarrow, draw=none, from=1, to=1-5]
	\arrow[shorten <=2pt, Rightarrow, from=0, to=2-1]
	\arrow[shorten <=2pt, shorten >=2pt, Rightarrow, from=1, to=2-4]
\end{tikzcd}\]
where simplices labeled by $=$ are degenerate, and simplices labeled by $\sim$ are thin.

The rest of the section is dedicated to the proof of the following proposition:
\begin{prop}
\label{prop:almost_associativity_of_star_co}
There is a zigzag of acyclic cofibrations functorial in $X,Y,Z$: 
$$X \costar  (Y \costar  Z) \leftrightsquigarrow (X \costar  Y) \costar  Z.$$
\end{prop}

Until the end of this section, we will use notations and results of appendix $B$.
\begin{construction}
We define three marked simplicial sets:
\[\begin{tikzcd}
	00 & 01 && {\underline{10}} & {\overline{10}} & {\underline{10}} & {\overline{10}} \\
	& 11 && 00 & 11 & 00 & 11
	\arrow[""{name=0, anchor=center, inner sep=0}, from=1-1, to=2-2]
	\arrow["{P := ~~~~~~~~~~~~~~~~}"', from=1-2, to=2-2]
	\arrow[from=1-1, to=1-2]
	\arrow[curve={height=12pt}, from=1-1, to=2-2]
	\arrow["{Q:=~~}", shorten <=2pt, shorten >=2pt, Rightarrow, no head, from=2-4, to=1-4]
	\arrow[from=1-4, to=1-5]
	\arrow[""{name=1, anchor=center, inner sep=0}, from=2-4, to=2-5]
	\arrow[""{name=2, anchor=center, inner sep=0}, from=2-4, to=1-5]
	\arrow[from=1-6, to=1-7]
	\arrow[""{name=3, anchor=center, inner sep=0}, from=1-6, to=2-7]
	\arrow[shorten <=2pt, shorten >=2pt, Rightarrow, no head, from=1-7, to=2-7]
	\arrow[""{name=4, anchor=center, inner sep=0}, shorten <=2pt, shorten >=2pt, Rightarrow, no head, from=1-5, to=2-5]
	\arrow[""{name=5, anchor=center, inner sep=0}, shorten <=2pt, shorten >=2pt, Rightarrow, no head, from=2-6, to=1-6]
	\arrow[""{name=6, anchor=center, inner sep=0}, from=2-6, to=2-7]
	\arrow["\sim"{description}, Rightarrow, draw=none, from=0, to=1-2]
	\arrow["{=}"{description}, Rightarrow, draw=none, from=2, to=1-4]
	\arrow["{=}"{description}, Rightarrow, draw=none, from=3, to=1-7]
	\arrow[shorten <=6pt, shorten >=6pt, Rightarrow, from=4, to=5]
	\arrow[shorten <=2pt, shorten >=2pt, Rightarrow, from=1, to=2]
	\arrow[shorten <=2pt, shorten >=2pt, Rightarrow, from=6, to=3]
\end{tikzcd}\]
where arrows labeled by $=$ are degenerate, and arrows labeled by $\sim$ are thin. We also set:
\[\begin{tikzcd}
	{\Delta[1]\underset{[00]\amalg[11]}{\coprod}\Delta[1]} & P \\
	Q & R.
	\arrow[from=1-1, to=1-2]
	\arrow[from=1-1, to=2-1]
	\arrow[from=2-1, to=2-2]
	\arrow[from=1-2, to=2-2]
	\arrow["\lrcorner"{anchor=center, pos=0.125, rotate=180}, draw=none, from=2-2, to=1-1]
\end{tikzcd}\]
We then have two inclusions
$i:K\to R$ and $j:L\to R$ where $i$ sends $10$ on $\overline{10}$ and $j$ sends $10$ on $\underline{10}$. Let $X$ be a marked simplicial set. We then define several marked simplicial sets: \newline
$\mathbf{(K\times X,M_{K\times X})}$. The marking $M_{K\times X}$ is obtained as the reunion of $M'_{K\times X}$ and all $n$-simplices $(v,x)$ such that $v$ is a thin simplex of $[00,10]$ and $x$ is thin, where  $M'_{K\times X}$ is the marking of the colimit:
\[\begin{tikzcd}
	{[1]\otimes X\otimes[0,1]} & {[0,1]\otimes X\otimes[0,1]} \\
	{[1]\otimes X\otimes \Delta[0]} & {(K\times M,M'_{K\times M}).}
	\arrow[from=1-1, to=2-1]
	\arrow[from=2-1, to=2-2]
	\arrow[from=1-1, to=1-2]
	\arrow[from=1-2, to=2-2]
	\arrow["\lrcorner"{anchor=center, pos=0.125, rotate=180}, draw=none, from=2-2, to=1-1]
\end{tikzcd}\]

$\mathbf{(L\times X,M_{L\times X})}$. The marking $M_{L\times X}$ is obtained as the reunion of $M'_{L\times X}$ and all $n$-simplices $(v,x)$ such that $v$ is a thin simplex of $[10,11]$ and $x$ is thin, where  $M'_{L\times X}$ is the marking of the colimit:
\[\begin{tikzcd}
	{[0,1]\otimes X\otimes[0]} & {[0,1]\otimes X\otimes[0,1]} \\
	{\Delta[0]\otimes X\otimes [0]} & {(L\times M,M'_{L\times M}).}
	\arrow[from=1-1, to=2-1]
	\arrow[from=2-1, to=2-2]
	\arrow[from=1-1, to=1-2]
	\arrow[from=1-2, to=2-2]
	\arrow["\lrcorner"{anchor=center, pos=0.125, rotate=180}, draw=none, from=2-2, to=1-1]
\end{tikzcd}\]

$\mathbf{(R\times X,M_{R\times X})}$. The marking $M_{R\times X}$ is the reunion of $M_{K\times X}$, $M_{L\times X}$, and $M_{Q\times X}$, where  $M_{Q\times X}$ is the marking of the colimit:
\[\begin{tikzcd}
	{\Delta[3]\times X} & {Q\times X} & {\Delta[3]\times X} \\
	{\Delta^2[3]\oslash X} & {(Q\times X,M_{Q\times X})} & {\Delta^1[3]\invoslash X.}
	\arrow[from=1-1, to=1-2]
	\arrow[from=1-1, to=2-1]
	\arrow[from=2-1, to=2-2]
	\arrow[from=1-3, to=2-3]
	\arrow[from=2-3, to=2-2]
	\arrow[from=1-3, to=1-2]
	\arrow[from=1-2, to=2-2]
\end{tikzcd}\]

$\mathbf{(P\times X,M_{P\times X})}$. The marking $M_{P\times X}$ is the restriction of $M_{R\times X}$ to $P\times X$.
\end{construction}

\begin{lemma}
The cofibration $(K\times X,M_{K\times X})\to (R\times X,M_{R\times X})$ is an acyclic cofibration.
\end{lemma}
\begin{proof}
We define $M$ to be the marking on $R\times X$ obtained as the colimit:
\[\begin{tikzcd}
	{\Lambda^2[3]\oslash X} & {(K\times X,M_{K\times X})} \\
	{\Delta^2[3]\oslash X} & {(R\times X,M).}
	\arrow[from=1-1, to=1-2]
	\arrow[from=1-1, to=2-1]
	\arrow[from=2-1, to=2-2]
	\arrow[from=1-2, to=2-2]
\end{tikzcd}\]
Lemma \ref{lemma:oslash} states that the right hand morphism is an acyclic cofibration.  Furthermore, we have an obvious inclusion $M\subset M_{R\times X}$, and we then have to show $M_{R\times X}\subset M$ to conclude.

First, we remark that $M$ includes $M_{P\times X}$. Let
$(v,s^{p-1}x)$ be a $n$-simplex in $\Delta^1[3]\invoslash X$. \textbf{Case $v_0=00$ and $v_1=11$.} The simplex $(v,s^{p-1}x)$ is in $\Delta^2[3]\oslash X$, and so is $M$. \textbf{Case $v_0\neq 00$ or $v_1\neq 11$.}  The simplex $(v,s^{p-1}x)$ is in $M_{P\times X}$ and so in $M$. The set $M$ then includes  $M_{Q\times X}$.

Let $(v,x)$ be a $n$-simplex in $M_{L\times X}$. \textbf{Case $(v,x)\in P\times X$.} The simplex $(v,x)$ is in $M_{P\times X}$ 
and so in $M$. \textbf{Case $(v,x)\notin P\times X$.} The simplex $v$ is then a degeneracy of $[00,\underline{10},11]$. First we remark that $(v[\underline{10}/\overline{10}],x)$ is in $M$.
As this set includes both $\Delta^1[3]\invoslash X$ and $\Delta^2[3]\oslash X$, we can apply lemma \ref{lemma:oslash_saturation_two way} that states that $(v,x)$ is in $M$.

The set $M$ then includes $M_{R\times X}$.
\end{proof}

Similarly, one can show: 
\begin{lemma}
The cofibration $(L\times X,M_{L\times X})\to (R\times X,M_{R\times X})$ is an acyclic cofibration.
\end{lemma}

\begin{construction}
We define $C_{\Delta[0],X,\Delta[0]}$ as the pushout: 
\[\begin{tikzcd}
	{[\underline{10},\overline{10}]\times X} & {(R\times X,M_{R\times X})} \\
	{[\underline{10},\overline{10}]} & {C_{\Delta[0],X,\Delta[0]}.}
	\arrow[from=1-1, to=1-2]
	\arrow[from=1-2, to=2-2]
	\arrow[from=1-1, to=2-1]
	\arrow[from=2-1, to=2-2]
	\arrow["\lrcorner"{anchor=center, pos=0.125, rotate=180}, draw=none, from=2-2, to=1-1]
\end{tikzcd}\]
\end{construction}

\begin{lemma}
There is a zigzag of acyclic cofibrations, natural in $X$:
$$(\Delta[0] \costar  X )\costar  \Delta[0] \leftrightsquigarrow \Delta[0] \costar  (X \costar  \Delta[0]) .$$
\end{lemma}
\begin{proof}
We have two pushouts:
\[\begin{tikzcd}
	{[00,10]\times X} & {(K\times X,M_{K\times X})} & {[10,11]\times X} & {(L\times X,M_{L\times X})} \\
	{[00,10]} & {(\Delta[0] \costar  X )\costar  \Delta[0],} & {[10,11]} & {\Delta[0] \costar  (X \costar  \Delta[0]).}
	\arrow[from=1-1, to=1-2]
	\arrow[from=1-2, to=2-2]
	\arrow[from=1-1, to=2-1]
	\arrow[from=2-1, to=2-2]
	\arrow["\lrcorner"{anchor=center, pos=0.125, rotate=180}, draw=none, from=2-2, to=1-1]
	\arrow[from=1-4, to=2-4]
	\arrow["\lrcorner"{anchor=center, pos=0.125, rotate=180}, draw=none, from=2-4, to=1-3]
	\arrow[from=1-3, to=2-3]
	\arrow[from=1-3, to=1-4]
	\arrow[from=2-3, to=2-4]
\end{tikzcd}\]
These two colimits are homotopy colimits, and so we have a zigzag of acyclic cofibrations:
$$ (\Delta[0] \costar  X )\costar  \Delta[0]\to C_{\Delta[0],X,\Delta[0]} \leftarrow \Delta[0] \costar  (X \costar  \Delta[0]).$$
\end{proof}

\begin{construction}
Let $X,Y,Z$ be three marked simplicial sets. We define $C_{X,Y,Z}$ as the following pushout: 
\[\begin{tikzcd}
	{X\otimes(\Delta[0]\costar   Y \amalg_Y Y\costar \Delta[0])\otimes Z} & {X\otimes C_{\Delta[0],Y,\Delta[0]}\otimes Z} \\
	{X\costar Y \coprod_Y Y\costar Z} & {C_{X,Y,Z}.}
	\arrow[from=1-1, to=1-2]
	\arrow[from=1-1, to=2-1]
	\arrow[from=2-1, to=2-2]
	\arrow[from=1-2, to=2-2]
\end{tikzcd}\]
\end{construction}

\begin{proof}[Proof of the proposition \ref{prop:almost_associativity_of_star_co}]
We have two pushouts: 
\[\begin{tikzcd}
	{X\otimes(\Delta[0]\costar   Y \amalg_Y Y\costar \Delta[0])\otimes Z} & {X\otimes (\Delta[0] \costar  (Y \costar  \Delta[0]))\otimes Z} \\
	{X\costar Y \coprod_Y Y\costar Z} & {X \costar  (Y \costar  Z),} \\
	{X\otimes(\Delta[0]\costar   Y \amalg_Y Y\costar \Delta[0])\otimes Z} & {X\otimes (\Delta[0] \costar  Y) \overset{co}{\star} \Delta[0])\otimes Z} \\
	{X\costar Y \coprod_Y Y\costar Z} & {(X \costar  Y )\costar  Z.}
	\arrow[""{name=0, anchor=center, inner sep=0}, from=1-1, to=1-2]
	\arrow[from=1-2, to=2-2]
	\arrow[from=1-1, to=2-1]
	\arrow[from=2-1, to=2-2]
	\arrow[from=3-2, to=4-2]
	\arrow[""{name=1, anchor=center, inner sep=0}, from=3-1, to=3-2]
	\arrow[from=3-1, to=4-1]
	\arrow[from=4-1, to=4-2]
	\arrow["\lrcorner"{anchor=center, pos=0.125, rotate=180}, draw=none, from=2-2, to=0]
	\arrow["\lrcorner"{anchor=center, pos=0.125, rotate=180}, draw=none, from=4-2, to=1]
\end{tikzcd}\]
These two colimits are homotopy colimits, and so we have a zigzag of acyclic cofibrations:
$$ (X \costar  Y )\costar Z\to C_{X,Y,Z} \leftarrow X\costar  (Y \costar  Z).$$
\end{proof}

\subsection{Wedge}
\label{subsection:wedge}

\begin{construction}

Let $X$ be a simplicial set. We define \textit{the wedge} of $\Sigma X$ and $\Delta[1]$, noted $\Sigma X\fwedge \Delta[1]$, as the colimit of the following diagram:
\[\begin{tikzcd}
	{X\otimes[0,1]} & {X\otimes\Delta[2]_t} & {X\otimes[1,2]} \\
	{\Sigma X} & {X\fwedge\Delta[1]} & {[1,2].}
	\arrow[from=1-1, to=2-1]
	\arrow[from=1-3, to=2-3]
	\arrow[from=1-1, to=1-2]
	\arrow[from=1-3, to=1-2]
	\arrow[from=1-2, to=2-2]
	\arrow[from=2-3, to=2-2]
	\arrow[from=2-1, to=2-2]
\end{tikzcd}\]
This assignation defines a cocontinuous functor $\uvar\fwedge \Delta[1]:\mSset\to \mSset_{\Delta[0]\amalg\Delta[1]/}.$ For every acyclic cofibration $K\to L$, we have a diagram: 
\[\begin{tikzcd}
	{\Delta[0]\coprod\Delta[1]} & {K\otimes([0]\coprod[1,2])} & {K\otimes\Delta[2]_t} \\
	{K\otimes\Delta[2]_t} & {L\otimes\Delta[2]_t} & {L\otimes\Delta[2]_t.}
	\arrow[from=1-2, to=1-1]
	\arrow[from=1-2, to=1-3]
	\arrow[from=1-1, to=2-1]
	\arrow[from=2-2, to=2-1]
	\arrow[from=2-2, to=2-3]
	\arrow[from=1-3, to=2-3]
	\arrow[from=1-2, to=2-2]
\end{tikzcd}\]
Taking colimit of lines induces a morphism
$$K\fwedge \Delta[1]\to L\fwedge \Delta[1].$$
However, these two colimits are homotopy colimits, and all the horizontal maps of the previous diagram are weak equivalences. This morphism is then an acyclic cofibration.
This shows that this functor is a left Quillen functor. We denote $$\triangledown:\Sigma X\to \Sigma X\triangledown \Delta[1]$$ the morphism induced by the inclusion $X\otimes [0,2]\subset X\otimes \Delta[2]_t$. We define similarly the left Quillen functor $$\Delta[1]\fwedge\uvar:\mSset\to \mSset_{\Delta[1]\amalg \Delta[0]}$$ and the morphism
$$\triangledown:\Sigma X\to \Delta[1]\fwedge\Sigma X.$$
\end{construction}

\begin{prop}
Morphisms 
$$ \Sigma X\coprod_{\Delta[0]}\Delta[1]\to \Sigma X\fwedge \Delta[1]~~~~\mbox{and}~~~~ \Delta[1]\coprod_{\Delta[0]}\Sigma X\to \Delta[1]\fwedge \Sigma X$$
are acyclic cofibrations. 
\end{prop}
\begin{proof}
We have a colimit diagram
\[\begin{tikzcd}
	{X\otimes([0]\coprod[1,2])} & {X\otimes \Lambda^{1}[2]} & {X\otimes\Delta[2]_t} \\
	{\Delta[0]\coprod\Delta[1]} & {\Sigma X\coprod_{\Delta[0]} \Delta[1]} & {\Sigma X\triangledown \Delta[1].}
	\arrow[from=1-1, to=2-1]
	\arrow[from=2-1, to=2-2]
	\arrow[from=1-1, to=1-2]
	\arrow[from=1-2, to=1-3]
	\arrow[from=1-3, to=2-3]
	\arrow[from=1-2, to=2-2]
	\arrow["\lrcorner"{anchor=center, pos=0.125, rotate=180}, draw=none, from=2-3, to=1-2]
	\arrow["\lrcorner"{anchor=center, pos=0.125, rotate=180}, draw=none, from=2-2, to=1-1]
	\arrow[from=2-2, to=2-3]
\end{tikzcd}\]
The upper right horizontal morphism is an acyclic cofibration, and so is the downer right horizontal one. We proceed similarly for the morphism $$\Delta[1]\coprod_{\Delta[0]}\Sigma X\to \Delta[1]\fwedge \Sigma X.$$
\end{proof}

\subsection{Street nerve}

In \cite{street1987algebra}, Street defines a cosimplicial object in $\infcat$, that associates to $n$, the \textit{$n^{th}$ oriental} $O_n$. The $\omega$-category $O_n$ is a $n$-category admitting a unique non identity $n$-cell. We then define $(O_n)_t$ to be the following pushout: 
\[\begin{tikzcd}
	{\Gb_n} & {O_n} \\
	{\Gb_{n-1}} & {(O_n)_t.}
	\arrow[from=1-1, to=1-2]
	\arrow["id"', from=1-1, to=2-1]
	\arrow[from=2-1, to=2-2]
	\arrow[from=1-2, to=2-2]
	\arrow["\lrcorner"{anchor=center, pos=0.125, rotate=180}, draw=none, from=2-2, to=1-1]
\end{tikzcd}\]
This induces a colimit preserving realization 
$$R:\stratSset\to \infcat.$$
This realization functor preserves the join, the suspension and the wedge. Moreover, \cite[Theorem 249]{verity2008complicial} states that this functor sends complicial horn inclusions and complicial thinness extensions  to isomorphisms. 
However, this is not true for saturation extensions. We then have to consider Gaunt $\omega$-categories, first introduced in \cite{barwick2021unicity}.

\begin{definition}
Let $C$ be a $\omega$-category $C$. A $n$-cell is an \textit{equivalence} if there exists two cells $g,h:b\to a$ such that $f*_{n-1}g=id$ and $h*_{n-1}f =id$.  An $\omega$-category is  \textit{Gaunt} if all the equivalences are identities. We denote $\satinfcat$ the full subcategory of $\infcat$ whose objects are Gaunt $\omega$-categories. This category admits all colimits.
\end{definition}

All the orientals are Gaunt categories. Furthermore the category $\satinfcat$ is closed by Gray tensor product, suspension, join and wedge. The realization functor $\tilde{R}:\stratSset\to \satinfcat$ sends all elementary anodyne extensions to isomorphisms. This functor lifts to a realization $\mSset\to \satinfcat$ that we also denote $\tilde{R}$.

Principal results of  \cite{gagna2021nerves} and \cite{ozornova2020suspension} states that 
for all $l,n$, the following morphism is an acyclic cofibration: 
$$\Sigma^l\Delta[n]\to N_{str}(\tilde{R}(\Sigma^l\Delta[n])).$$

\begin{definition}
Let $l$ be an integer. We define the \textit{Street endofunctor} $i^l_{str}$ to be the colimit preserving functor defined on representables by: 
$$i^l_{str}([n]) := N_{str}(\tilde{R}(\Sigma^l\Delta[n]))~~~\mbox{ and  }~~~i^l_{str}([n]_t) :=\tau_{n+l} (i^l_{str}([n]))$$
\end{definition}

\begin{prop}
\label{prop:i_str_is_Quillen}
 The functor $i^l_{srt}$ is left Quillen and 
the natural transformation 
$$\Sigma^l\to i^l_{srt}$$ 
is weakly invertible.
\end{prop}
\begin{proof}
We recall the truncation functor $\tau_n:\mSset\to \mSset$ is a left Quillen functor, and so preserves weak equivalences between cofibrant objects. Furthermore the $l$-suspension fulfills:
$$\Sigma^l \Delta[n]_t = \tau_{n+l}\Sigma^l\Delta[n]$$
and the morphism  $\Sigma^l\Delta[n]_t\to i_{str}([n]_t)$ is then a weak equivalence.
The set of objects $X$ such that the morphism $\Sigma^lX\to i_{srt}^lX$ is a weak equivalence is saturated by monomorphisms, includes all representables and then consists of all marked simplicial sets. Now let $K\to L$ be an acyclic cofibration. We have a commutative square:
\[\begin{tikzcd}
	\Sigma^lK & {i^l_{str}(K)} \\
	\Sigma^lL & {i^l_{str}(L)}
	\arrow["\sim"', from=1-1, to=2-1]
	\arrow["\sim", from=1-1, to=1-2]
	\arrow["\sim"', from=2-1, to=2-2]
	\arrow[from=1-2, to=2-2]
\end{tikzcd}\]
By two out of three, $i^l_{str}(K)\to i^l_{str}(L)$ is then an acyclic cofibration. The functor $i^l_{srt}$ is then left Quillen. 
\end{proof}

We conclude this section by a technical lemma.

\begin{lemma}
\label{lemma:unicity_of_composition}
Let $n,m$ be two integers and $X$ a category admitting a loop free and atomic basis. Let $$f:\Sigma^m(\Delta[n]\star \Sigma X)\to\Sigma^m( \Delta[n]\star (\Delta[1]\vee \Sigma X))$$ be a morphism fitting in the following diagram:
\[\begin{tikzcd}
	{\Sigma^m(\Delta[n]\star\Sigma \emptyset)} && {\Sigma^m( \Delta[n]\star (\Delta[1]\vee \Sigma X))} \\
	{\Sigma^m(\Delta[n]\star \Sigma X)} && {\Sigma^m(\Delta[n]\star\Sigma X)}
	\arrow["f", from=2-1, to=1-3]
	\arrow[from=1-1, to=2-1]
	\arrow["{\Sigma^m(\Delta[n]\star \triangledown_{\emptyset})}", from=1-1, to=1-3]
	\arrow[from=1-3, to=2-3]
	\arrow["id"', from=2-1, to=2-3]
\end{tikzcd}\]
Then $f =\Sigma^m(\Delta[n]\star\triangledown_X)$.
\end{lemma}
\begin{proof}
We can suppose without lost of generality that $m=0$.
All these categories admit loop free and atomic basis. We can then show this lemma in the category of  augmented directed complexes using Steiner theory developed in \cite{steiner2004omega}. 
Let $K_\emptyset, K_X$ and $K_{\Delta[1],X}$ be the augmented directed complexes associated to $\Delta[n]\star\Sigma \emptyset$, $\Delta[n]\star \Sigma X$ and 
$\Delta[n]\star (\Delta[1]\vee \Sigma X)$. For sake of readability, we will denote $f$ the non trivial one cell of $\Delta[1]$, and $\perp,\invperp$ respectively the $0$-source and the $0$-target of cells of $\Sigma X$, and we then have $t(f)=\perp$.

Basis of $K_\emptyset, K_X$ and $K_{\Delta[1],X}$ are defined by formulas:
$$
\begin{array}{rcl}
(B_{K_\emptyset})_k &:=& \{x\star \perp, x\star \invperp , x\in \Delta[n]_l  , l+1 = k\}\cup (\Delta[n])_k\\
(B_{K_X})_k &:=& \{x\star y, x\in \Delta[n]_l  , y\in (B_K)_p, l+p+2 = k\}\cup (B_{K_\emptyset})_k\\
(B_{K_{\Delta[1],X}})_k &:=& \{x\star f, x\in \Delta[n]_l  , l+2 = k\}\cup(B_{K_X})_k 
\end{array}$$

The commutativity of the diagram  implies that $$
\begin{array}{rclr}
f(x\star \perp)&=& x\star s(f)\\
f(x\star \invperp)&=& x\star \invperp\\
f(x\star y)&=&  x\star y +r_{x,y} &\mbox{ if $|y|\geq 1$}\\
\end{array}
$$
where $r_{x,y}$ is a positive sum of  elements of $(\Delta[n]\star\Delta[1])_{|x|+|y|}$.
We will show by induction on $|x|+|y|$ that: 
$$\begin{array}{rcll}
r_{x,y}&=& x\star f &\mbox{ if $|y|= 1$}\\
&=&0&\mbox{ if $|y|< 1$} .\\
\end{array}$$

Suppose the result when the sum of dimensions of $x$ and $y$ is $k-1$. Let $x, y$ be two cells such that $|x|+|y|=k$.
\textbf{Case $|y|=1$.} The commutativity of $f$ with $\partial$ implies that 
$$\begin{array}{rcl}
\partial r_{x,y} &=&  f(\partial (x\star y)) - \partial (x\star y)\\
&=& r_{\partial x,y} + (-1)^{|x|+1} r_{x,\partial y} \\
&=& (\partial x)\star f + (-1)^{|x|+1}(x\star t(f) -x \star s(f))
\end{array}$$
and $r_{x,y}$ is then equal to $x\star f$. \textbf{Case $|y|>1$.} The commutativity of $f$ with $\partial$ implies that 
$$ \partial r_{x,y} = 0$$
and $r_{x,y}$ is then equal to $0$.
\end{proof}

\section{About the links between suspension, join and Gray tensor product with the directed interval}

\subsection{Formulas for tensor product with the directed interval}

\begin{construction}

Let $C$ be the following colimit:
\[\begin{tikzcd}
	{\Delta[3]\times\{0\}\coprod \Delta[3]\times\{1\}} & {\Delta[3]\times\Delta[1]} \\
	{\Delta[1]\coprod\Delta[1]} & {C.}
	\arrow["{s^0s^0\coprod s^2s^3}"', from=1-1, to=2-1]
	\arrow[""{name=0, anchor=center, inner sep=0}, from=1-1, to=1-2]
	\arrow[from=1-2, to=2-2]
	\arrow[from=2-1, to=2-2]
	\arrow["\lrcorner"{anchor=center, pos=0.125, rotate=180}, draw=none, from=2-2, to=0]
\end{tikzcd}\]

We define several marked simplicial sets whose underlying simplicial sets are sub objects of C: 
\[\begin{tikzcd}
	{} & 00 & 01 && 00 & 01 \\
	& 10 & 11 && 20 & 21 \\
	& 20 & 21 && 00 & 01 \\
	& 30 & 31 & {} & 30 & 31
	\arrow[from=1-3, to=2-3]
	\arrow["{\large{A_1:=~~~~~~}}"', Rightarrow, no head, from=2-2, to=3-2]
	\arrow[Rightarrow, no head, from=2-3, to=3-3]
	\arrow["{\large{A_0:=~~~~~~}}"', Rightarrow, no head, from=1-2, to=2-2]
	\arrow[from=1-2, to=1-3]
	\arrow[from=3-2, to=3-3]
	\arrow[from=2-2, to=2-3]
	\arrow[""{name=0, anchor=center, inner sep=0}, from=1-2, to=2-3]
	\arrow[""{name=1, anchor=center, inner sep=0}, from=2-2, to=3-3]
	\arrow["{\large{A_3:=~~~~~~}}"', Rightarrow, no head, from=1-5, to=2-5]
	\arrow[from=1-6, to=2-6]
	\arrow[from=2-5, to=2-6]
	\arrow[""{name=2, anchor=center, inner sep=0}, from=1-5, to=2-6]
	\arrow[from=4-2, to=4-3]
	\arrow["{\large{A_4:=~~~~~~}}"', from=3-5, to=4-5]
	\arrow[from=4-5, to=4-6]
	\arrow[from=3-6, to=4-6]
	\arrow[from=3-5, to=3-6]
	\arrow[""{name=3, anchor=center, inner sep=0}, from=3-5, to=4-6]
	\arrow[from=1-5, to=1-6]
	\arrow["{\large{A_2:=~~~~~~}}"', from=3-2, to=4-2]
	\arrow[""{name=4, anchor=center, inner sep=0}, from=3-2, to=4-3]
	\arrow[Rightarrow, no head, from=3-3, to=4-3]
	\arrow["\sim"{description}, Rightarrow, draw=none, from=0, to=1-3]
	\arrow["\sim"{description}, Rightarrow, draw=none, from=1, to=2-3]
	\arrow["\sim"{description}, Rightarrow, draw=none, from=0, to=2-2]
	\arrow["\sim"{description}, Rightarrow, draw=none, from=2, to=1-6]
	\arrow[shorten <=2pt, Rightarrow, from=2, to=2-5]
	\arrow["\sim"{description}, Rightarrow, draw=none, from=3, to=3-6]
	\arrow[shorten <=2pt, Rightarrow, from=3, to=4-5]
	\arrow[shorten <=2pt, Rightarrow, from=1, to=3-2]
	\arrow["\sim"{description}, Rightarrow, draw=none, from=4, to=4-2]
	\arrow["\sim"{description}, Rightarrow, draw=none, from=4, to=3-3]
\end{tikzcd}\]
where arrows labeled by $=$ are degenerate and arrow labeled by $\sim$ are thin.

Let $B_0$ be the sub object corresponding to the image of $[0,1,2]\times[0,1]$ where the marking includes all cells of dimension $\leq 2$, except $[10,20,21]$ and $[00,20,21]$.

Let $B_1$ be the sub object corresponding to the image of $[0,2,3]\times[0,1]$  where the marking includes all cells of dimension $\leq 2$, except $[00,20,21]$, $[00,30,31]$ and $[00,20,31]$.

Let $B$ be the reunion of $[0,1,2]\times[0,1]$ and $[0,2,3]\times[0,1]$ where the marking is the reunion of $B_0$ and $B_1$.
\end{construction}

\begin{lemma}
Morphisms $A_0\cup A_1\to B_0$ and $A_3\to B_0$ are acyclic cofibrations. 
\end{lemma}
\begin{proof}
The cofibration $A_0\cup A_1\to B_0$  fits in the following pushout square:
\[\begin{tikzcd}
	{\Lambda^{1}[2]\otimes \Delta[1]\cup\Delta[2]_t\otimes \partial\Delta[1]} & {A_1\cup A_2} \\
	{\Delta[2]_t\otimes \Delta[1]} & {B_0.}
	\arrow[""{name=0, anchor=center, inner sep=0}, from=1-1, to=1-2]
	\arrow[from=1-1, to=2-1]
	\arrow[from=1-2, to=2-2]
	\arrow["{[012]\times[01]}"', from=2-1, to=2-2]
	\arrow["\lrcorner"{anchor=center, pos=0.125, rotate=180}, draw=none, from=2-2, to=0]
\end{tikzcd}\]

The cofibration $A_3\to B_0$ is a sequence of inclusions:
$$A_3=:(D_0,M_0)\subset (D_1,M_1)\subset (D_2,M_2)\subset(D_3,M_3)\subset(D_4,M_4)\subset (D_5,M_5)\subset (D_6,M_6):= B_0,$$ where 

\begin{itemize}[leftmargin=* ,parsep=0cm,itemsep=0cm,topsep=0cm]
\item $D_1 = D_0\cup [00,{01},11]$ ;
\item $D_2 = D_1\cup [ {00},10,11]$ ;
\item $D_2 = D_1\cup [ {00},10,21]$ ;
\item $D_4 = D_3\cup [00, {01},11,21]$; 
\item $D_5 = D_4\cup [ {00},10,11,21]$; 
\item $D_6 = D_5\cup [ {00},10,20,21]$; 
\end{itemize} and
\begin{itemize}[leftmargin=* ,parsep=0cm,itemsep=0cm,topsep=0cm]
\item $(D_0,M_0)\to (D_1,M_1)$ is a pushout of $\Lambda^1[2]\to \Delta^1[2]$;
\item $(D_1,M_1)\to (D_2,M_2)$ is a pushout of $\Lambda^0[2]\to \Delta^0[2]$;
\item $(D_2,M_2)\to (D_3,M_3)$ is a pushout of $\Lambda^0[2]\to \Delta^0[2]$;
\item $(D_3,M_3)\to (D_4,M_4)$ is a pushout of $\Lambda^1[3]\to \Delta^1[3]$;
\item $(D_4,M_4)\to (D_5,M_5)$ is a pushout of $\Lambda^0[3]\to \Delta^0[3]$;
\item $(D_5,M_5)\to (D_6,M_6)$ is a pushout of $\Lambda^0[3]\to \Delta^0[3]$.
\end{itemize}
\end{proof}

\begin{lemma}
Morphisms $A_2\cup A_3\to B_1$ and $A_4\to B_1$ are acyclic cofibrations. 
\end{lemma}
\begin{proof}
The cofibration $A_2\cup A_3\to B_1$ fits in the pushout square:
\[\begin{tikzcd}
	{\Lambda^{1}[2]\otimes \Delta[1]\cup\Delta[2]'\otimes \partial\Delta[1]} & {A_2\cup A_3} \\
	{\Delta[2]'\otimes \Delta[1]} & {B_1}
	\arrow[from=1-1, to=1-2]
	\arrow["{[023]\times[01]}"', from=2-1, to=2-2]
	\arrow[from=1-1, to=2-1]
	\arrow[from=1-2, to=2-2]
\end{tikzcd}\]

The cofibration $A_4\to B_1$ is a sequence of inclusions:
$$A_4=:(D_0,M_0)\subset (D_1,M_1)\subset (D_2,M_2)\subset(D_3,M_3)\subset(D_4,M_4)\subset (D_5,M_5)\subset (D_6,M_6):= B_1$$ where 
\begin{itemize}[leftmargin=* ,parsep=0cm,itemsep=0cm,topsep=0cm]
\item $D_1 = D_0\cup [00,21, {31}]$ ;
\item $D_2 = D_1\cup [20, {30},31]$ ;
\item $D_3 = D_2\cup [20,21, {31}]$;
\item $D_4 = D_3\cup [00,01,21, {31}]$;
\item $D_5 = D_4\cup [00,20, {30},31]$ ;
\item $D_6 = D_5\cup [00,20,21, {31}]$ ;
\end{itemize}
and
\begin{itemize}[leftmargin=* ,parsep=0cm,itemsep=0cm,topsep=0cm]
\item $(D_0,M_0)\to (D_1,M_1)$ is a pushout of $\Lambda^2[2]\to \Delta^2[2]$;
\item $(D_1,M_1)\to (D_2,M_2)$ is a pushout of $\Lambda^1[2]\to \Delta^1[2]$;
\item $(D_2,M_2)\to (D_3,M_3)$ is a pushout of $\Lambda^2[2]\to \Delta^2[2]$;
\item $(D_3,M_3)\to (D_4,M_4)$ is a pushout of $\Lambda^3[3]\to \Delta^3[3]$;
\item $(D_4,M_4)\to (D_5,M_5)$ is a pushout of $\Lambda^2[3]\to \Delta^2[3]$;
\item $(D_5,M_5)\to (D_6,M_6)$ is a pushout of $\Lambda^3[3]\to \Delta^3[3]$.
\end{itemize}
\end{proof}

\begin{lemma}
The maps $A_0\cup A_1\cup A_2\to B$ and $A_4\to B$ are acyclic cofibrations. 
\end{lemma}
\begin{proof}
This is a direct consequence of the last two lemmas.
\end{proof}

\subsubsection*{First formula}
\label{subsubsection:First_formula} We now prove that there is a zigzag of acyclic cofibrations, natural in $X$, between the colimit of
$$\Delta[1]\fwedge\Sigma X\xleftarrow{\triangledown} \Sigma X\hookrightarrow \Sigma(X\otimes\Delta[1])\hookleftarrow \Sigma X\xrightarrow{\triangledown} \Sigma X\fwedge\Delta[1]$$
and $(\Sigma X)\otimes \Delta[1]$.

\begin{notation}
In the remaining part of this section, $(X\times A_{ij},M_{X\times A_{ij}})$ stands for $(X\times(A_i\cup A_j),M_{X\times A_i}\cup M_{X\times A_j})$.
\end{notation}

\begin{construction}
The marked simplicial set
$\overline{X\otimes B}$ is the pushout:
\[\begin{tikzcd}
	{X\otimes([00,01]\coprod [30,31])} & {X\otimes B} \\
	{[00,01]\coprod [30,31]} & {\overline{X\otimes B}.}
	\arrow[from=1-1, to=2-1]
	\arrow[from=2-1, to=2-2]
	\arrow[from=1-2, to=2-2]
	\arrow[""{name=0, anchor=center, inner sep=0}, from=1-1, to=1-2]
	\arrow["\lrcorner"{anchor=center, pos=0.125, rotate=180}, draw=none, from=2-2, to=0]
\end{tikzcd}\]

Let $\overline{X\otimes A_i}$ and $\overline{X\otimes B_i}$  be the sub-objects of $\overline{X\otimes B}$ corresponding to image of ${X\otimes A_i}$  and $	{X\otimes B_i}$.

\end{construction}

\begin{lemma}
Morphisms $\overline{X\otimes A_0} \to \Delta[1]\fwedge \Sigma X$ and $\overline{X\otimes A_2} \to \Sigma X\fwedge \Delta[1]$ are acyclic cofibrations. 
\end{lemma}
\begin{proof}
We have a colimit diagram:
\[\begin{tikzcd}
	{X\otimes ([00,01]\coprod [11])} & {X\otimes [00,01]\coprod_{X\otimes[01]} X\otimes[01,11]} & {X\otimes A_0} \\
	{[00,01]\coprod [11]} & {\Delta[1]\coprod_{\Delta[0]}\Sigma X} & {\overline{X\otimes A_0}.}
	\arrow[from=1-1, to=2-1]
	\arrow[""{name=0, anchor=center, inner sep=0}, "\sim", from=1-2, to=1-3]
	\arrow[""{name=1, anchor=center, inner sep=0}, from=1-1, to=1-2]
	\arrow[from=2-1, to=2-2]
	\arrow["\sim", from=2-2, to=2-3]
	\arrow[from=1-2, to=2-2]
	\arrow[from=1-3, to=2-3]
	\arrow["\lrcorner"{anchor=center, pos=0.125, rotate=180}, draw=none, from=2-2, to=1]
	\arrow["\lrcorner"{anchor=center, pos=0.125, rotate=180}, draw=none, from=2-3, to=0]
\end{tikzcd}\]
That shows that $\Delta[1]\coprod_{\Delta[0]} \Sigma X \to \overline{X\otimes A_0}$ is an acyclic cofibration. We then have a commutative diagram: 
\[\begin{tikzcd}
	{\Delta[1]\coprod_{\Delta[0]}\Sigma X} & {\overline{X\otimes A_0}} & {\Delta[1]\fwedge\Sigma X}
	\arrow["\sim", from=1-1, to=1-2]
	\arrow[from=1-2, to=1-3]
	\arrow["\sim", curve={height=-24pt}, from=1-1, to=1-3]
\end{tikzcd}\]
and by two out of three, that shows that $\overline{X\otimes A_0} \to \Delta[1]\fwedge \Sigma X$ is an acyclic cofibration. 
We proceed similarly for the second morphism. 
\end{proof}

\begin{lemma}
Marked simplicial sets $\overline{X\otimes A_1}$ and  $\overline{X\otimes A_4}$ are respectively equal to $\Sigma (X\otimes \Delta[1])$ and $(\Sigma X)\otimes \Delta[1]$.
\end{lemma}
\begin{proof}
This is a direct computation.
\end{proof} 

\begin{prop}
\label{prop:interval_first_formula}
There is a zigzag of acyclic cofibrations: 
$$\Delta[1]\fwedge\Sigma X\coprod_{\Sigma X} \Sigma(X\otimes\Delta[1])\coprod_{\Sigma X}\Sigma X\fwedge\Delta[1] \leftrightsquigarrow (\Sigma X)\otimes \Delta[1].$$
\end{prop}
\begin{proof}
Acyclic cofibrations $X\otimes A_{012}\to X\otimes B$ and $X\otimes A_4\to X\otimes B$ induce acyclic cofibrations
$$\Delta[1]\fwedge\Sigma X\coprod \Sigma(X\otimes\Delta[1])\coprod \Sigma X\fwedge\Delta[1]\xrightarrow{\sim} \overline{X\otimes B}\xleftarrow{\sim} (\Sigma X)\otimes \Delta[1].$$
\end{proof}

\subsubsection*{Second formula}
\label{subsubsection:second_formula}
We now prove that there is a zigzag of acyclic cofibrations, natural in $X$, between the colimit of
$$\Sigma X\fwedge\Delta[1]\xleftarrow{\triangledown} \Sigma X \hookrightarrow\Sigma(\Delta[1]\otimes X) \hookleftarrow \Sigma X \xrightarrow{\triangledown} \Delta[1]\fwedge\Sigma X$$
and $\Delta[1]\otimes (\Sigma X)$.

\begin{construction}
We define several marked simplicial sets:\newline
$\mathbf{(\Delta[n]\times A_0,M_{X\times A_{0}})}$. 
The marking $M_{X\times A_0}$  is the reunion of  all simplices $(s^px,v)$ such that $[v_p,v_{p+1}] = [00,10]$, of  all simplices $(x,v)$ such that $v$ is a thin simplex of $[01,11]$ and $x$ is thin,   and of    $M'_{X\times A_{0}}$  which is  the marking of the following pushout:
\[\begin{tikzcd}
	{[0,1]\otimes \Delta[n]\otimes \{0\}} & {[0,1]\otimes \Delta[n]\otimes[0,1]} \\
	{X\otimes \{0\}} & {(X\times A_0,M'_{A_{0}}).}
	\arrow[from=1-1, to=2-1]
	\arrow[from=2-1, to=2-2]
	\arrow[from=1-2, to=2-2]
	\arrow[""{name=0, anchor=center, inner sep=0}, from=1-1, to=1-2]
	\arrow["\lrcorner"{anchor=center, pos=0.125, rotate=180}, draw=none, from=2-2, to=0]
\end{tikzcd}\]

$\mathbf{( X\times A_1,M_{X\times A_{1}})}$. The marking $M_{X\times A_1}$ is obtained as the colimit:
\[\begin{tikzcd}
	{[12]\otimes  X\otimes (\{0\}\coprod\{1\})} & {[1,2]\otimes  X\otimes[0,1]} \\
	{ X\otimes( \{0\}\coprod\{1\})} & {( X\times A_1,M_{X\times A_{1}}).}
	\arrow[""{name=0, anchor=center, inner sep=0}, from=1-1, to=1-2]
	\arrow[from=1-2, to=2-2]
	\arrow[from=1-1, to=2-1]
	\arrow[from=2-1, to=2-2]
	\arrow["\lrcorner"{anchor=center, pos=0.125, rotate=180}, draw=none, from=2-2, to=0]
\end{tikzcd}\]

$\mathbf{( X\times A_2,M_{X\times A_{2}})}$.
The marking $M_{X\times A_2}$  is the reunion of  all simplices $(x,v)$ such that $v$ is a thin simplex of $[20,30]$ and $x$ is thin, and of   $M'_{X\times A_{2}}$  which is the marking of the following pushout:
\[\begin{tikzcd}
	{[2,3]\otimes  X\otimes \{1\}} & {[2,3]\otimes X\otimes [0,1]} \\
	{ X\otimes \{1\}} & {( X\times A_2,M'_{X\times A_{2}}).}
	\arrow[from=1-1, to=2-1]
	\arrow[from=2-1, to=2-2]
	\arrow[from=1-2, to=2-2]
	\arrow[""{name=0, anchor=center, inner sep=0}, from=1-1, to=1-2]
	\arrow["\lrcorner"{anchor=center, pos=0.125, rotate=180}, draw=none, from=2-2, to=0]
\end{tikzcd}\]

$\mathbf{( X\times A_3,M_{X\times A_{3}})}$. 
 The marking $M_{X\times A_3}$ is the reunion of  all simplices $(x,v)$ such that $v$ is a thin simplex of $[01,21]$ and $x$ is thin, and  $M_{X\times A_{3}}$ which
 is the marking of the following pushout:
\[\begin{tikzcd}
	{[0,2]\otimes X\otimes \{0\}} & {[0,2]\otimes X\otimes[0,1]} \\
	{X\otimes \{0\}} & {(X\times A_3,M'_{X\times A_{3}}).}
	\arrow[from=1-1, to=2-1]
	\arrow[""{name=0, anchor=center, inner sep=0}, from=1-1, to=1-2]
	\arrow[from=2-1, to=2-2]
	\arrow[from=1-2, to=2-2]
	\arrow["\lrcorner"{anchor=center, pos=0.125, rotate=180}, draw=none, from=2-2, to=0]
\end{tikzcd}\]

$\mathbf{( X\times A_4,M_{X\times A_{4}})}$. 
The marking $M_{X\times A_4}$  is the reunion of the marking  of  $[0,3]\otimes  X\otimes [0,1]$, of  all simplices $(x,v)$ such that $v$ is a thin simplex of $[00,30]$ and $x$ is thin, and all simplices $(x,v)$ such that $v$ is a thin simplex of $[01,31]$ and $x$ is thin, .
\newline

$\mathbf{( X\times B_0,M_{X\times B_{0}})}$.
The marking $M_{X\times B_0}$  is the reunion of all simplices $(s^px,v)$ such that $[v_p,v_{p+1}] = [00,10]$,  of  all simplices $(x,v)$ such that $v$ is a thin simplex of $[01,11]$ and $x$ is thin,   and of   $M'_{X\times B_{0}}$  which is the marking of the following pushout: 
\[\begin{tikzcd}
	{[0,1,2]\otimes X\otimes ([0] \coprod [1])} & {[0,1,2]_t\otimes X\otimes[0,1]} \\
	{[20]\coprod[01,11]} & {(X\times B_0,M'_{B_{0}}).}
	\arrow[from=1-1, to=2-1]
	\arrow[""{name=0, anchor=center, inner sep=0}, from=1-1, to=1-2]
	\arrow[from=2-1, to=2-2]
	\arrow[from=1-2, to=2-2]
	\arrow["\lrcorner"{anchor=center, pos=0.125, rotate=180}, draw=none, from=2-2, to=0]
\end{tikzcd}\]

$\mathbf{( X\times B_1,M_{X\times B_{1}})}$.
The marking $M_{X\times B_1}$  is the reunion of   all simplices $(x,v)$ such that $v$ is a thin simplex of $[00,30]$ and $x$ is thin, and all simplices $(x,v)$ such that $v$ is a thin simplex of $[01,31]$ and $x$ is thin, and of $M'_{X\times B_1}$ which is the marking of the following pushout: 
\[\begin{tikzcd}
	{[0,2,3]\otimes X\otimes [1]} & {[0,2,3]_t\otimes X\otimes[0,1]} \\
	{[21,31]} & {(X\times B_1,M'_{B_{1}}).}
	\arrow[from=1-1, to=2-1]
	\arrow[""{name=0, anchor=center, inner sep=0}, from=1-1, to=1-2]
	\arrow[from=2-1, to=2-2]
	\arrow[from=1-2, to=2-2]
	\arrow["\lrcorner"{anchor=center, pos=0.125, rotate=180}, draw=none, from=2-2, to=0]
\end{tikzcd}\]

$\mathbf{( X\times B,M_{X\times B})}$.
The marking $M_{X\times B}$ is the reunion of  $M_{B_0}$ and $M_{B_1}$.
\end{construction}

\begin{lemma}
The morphism $( X\times A_{01},M_{X\times A_{01}})\to ( X\times B_0,M_{X\times B_0})$ is an acyclic cofibration.
\end{lemma}
\begin{proof}
We have a pushout product:
\[\begin{tikzcd}
	{\Lambda^{1}[2]\otimes \Delta[1]\cup\Delta[2]_t\otimes \partial\Delta[1]} & {(X\times A_{12},M_{X\times A_{12}})} \\
	{\Delta[2]_t\otimes \Delta[1]} & {(X\times B_0,\overline{M'_{X\times B_{0}}\cup M_{X\times A_{12}}}).}
	\arrow[""{name=0, anchor=center, inner sep=0}, from=1-1, to=1-2]
	\arrow["{[012]\times[01]}"', from=2-1, to=2-2]
	\arrow[from=1-1, to=2-1]
	\arrow[from=1-2, to=2-2]
	\arrow["\lrcorner"{anchor=center, pos=0.125, rotate=180}, draw=none, from=2-2, to=0]
\end{tikzcd}\]
We then have to show that $${M_{X\times B_0}}\subset\overline{M'_{X\times B_0}\cup M_{X\times A_{12}}}.$$
Let $(v,s^pv)$ be a $n$-simplex such that $v_{[p,p+1]}=[00,10]$.  There are inclusions: 
\[\begin{tikzcd}
	{\Delta[3]\times X} & {(X\times B_0,\overline{M'_{X\times B_{0}}\cup M_{X\times A_{12}}})} & {\Delta[2]\times X} \\
	{\Delta^1[3]\oslash X} && {\Delta^1[2]\oslash X}
	\arrow["{[00,10,20,21]\times id}"', from=2-1, to=1-2]
	\arrow[from=1-1, to=1-2]
	\arrow[from=1-1, to=2-1]
	\arrow["{[00,11,21]\times id}", from=2-3, to=1-2]
	\arrow[from=1-3, to=1-2]
	\arrow[from=1-3, to=2-3]
\end{tikzcd}\]

According to lemma \ref{lemma:oslash_saturation_one_way}, the fact that  $(v[20/10,21/11],x)$ is in $M'_{X\times A_1}$ implies that $(v,x)$ is in 
$\overline{M'_{X\times B_0}\cup M_{X\times A_{12}}}$.
This implies the desired inclusion.
\end{proof}
\begin{lemma}
The maps  $( X\times A_3,M_{X\times A_3})\to ( X\times B_0,M_{X\times B_0})$ is an acyclic cofibration.
\end{lemma}
\begin{proof}
The cofibration $(X\times A_3,M_{X\times A_3})\to (X\times B_0,M_{X\times B_0})$ is a sequence of inclusions:
$$
\begin{array}{l}
(X\times A_3,M_{X\times A_3}):=(D_0,M_0)\subset (D_1,M_1)\subset (D_2,M_2)\subset(D_3,M_3)\\
~~~~~~~~~~~~~~~~~~~~~~~~~~~~~\subset(D_4,M_4)\subset (D_5,M_5)\subset (D_6,M_6)\to(X\times B_0,M_{X\times B_0})
\end{array}$$ where 

\begin{itemize}[leftmargin=* ,parsep=0cm,itemsep=0cm,topsep=0cm]
\item $D_1 = D_0\cup X\times [00, {01},11]$ ;
\item $D_2 = D_1\cup X\times [ {00},10,11]$ ;
\item $D_2 = D_1\cup X\times [ {00},10,21]$ ;
\item $D_4 = D_3\cup X\times [00, {01},11,21]$; 
\item $D_5 = D_4\cup X\times [ {00},10,11,21]$; 
\item $D_6 = D_5\cup X\times [ {00},10,20,21]$; 
\end{itemize} and
\begin{itemize}[leftmargin=* ,parsep=0cm,itemsep=0cm,topsep=0cm]
\item $(D_0,M_0)\to (D_1,M_1)$ is a pushout of $\Lambda^1[2]\oslash X\to \Delta^1[2]\oslash X$;
\item $(D_1,M_1)\to (D_2,M_2)$ is a pushout of $\Lambda^0[2]\invoslash X\to \Delta^0[2]\oslash X$;
\item $(D_2,M_2)\to (D_3,M_3)$ is a pushout of $\Lambda^1[2] \invoslash X\to \Delta^1[2]\invoslash X$;
\item $(D_3,M_3)\to (D_4,M_4)$ is a pushout of $\Lambda^1[3]\oslash X\to \Delta^1[3]\oslash X$;
\item $(D_4,M_4)\to (D_5,M_5)$ is a pushout of $\Lambda^0[3]\invoslash X\to \Delta^0[3]\invoslash X$;
\item $(D_5,M_5)\to (D_6,M_6)$ is a pushout of $\Lambda^0[3]\invoslash X\to \Delta^0[3]\invoslash X$;
\end{itemize}

Now, we have to show that $M_{X\times B_0}\subset M_6$.
First we notice that all simplices $(v,s^px)$ with $v_{[p,p+1]}= [00,10]$ are in $M_6$. According to \ref{lemma:oslash_saturation_two way_case_extremum}, a simplex $(v,x)$ is in $M_6$ if and only if $(v[00/10],x)$ is in $M_6$, if and only if  $(v[10/00],x)$ is  in $M_6$.

The marking $M_4\cup M_6$ and \ref{lemma:oslash_saturation_one_way_variation} imply that a simplex $(x,v)$ such that $v$ includes $00$, $11$ and $21$, is in $M_6$, whenever $(x,v[11/01])$ is in $M_6$. 

Eventually let's see by case disjunction that every simplex $(v,s^{p-1}x)$ where $[v_{p-1},v_p]=[11,21]$   is in $M_6$. 
\textbf{Case $v_0\in\{01,11\}$.} The simplex $(v,s^{p-1}x)$ is $M_{X\times A_3}\subset M_6$. \textbf{Case $v_0=00$ and $10\notin v$.} The simplex  $(v[11/01],s^{p-1}x)$ is in $M_{X\times A_3}$. According to the last paragraph, the simplex $(v,s^{p-1}x)$ is then in $M_6$. \textbf{Case $v_0=00$ and $10\in v$.} The last case implies that $(v[10/00],s^{p-1}x)$ is in $M_6$. According to the first paragraph, this implies that $(v,s^{p-1}x)$ is in $M_6$. 
 According to lemma \ref{lemma:oslash_saturation_two way_case_extremum} a simplex $(v,x)$ is in $M_6$ thin if and only if $(v[21/11],x)$ is in $M_6$, if and only if  $(v[11/21],x)$ is in $M_6$.

If $v$ doesn't include $00$, $10$ or $01$, it is in $A_3$, where it is thin. If $v$ doesn't include $10$, the last paragraph implies that $(v,s^{p-1}x)$  is in $D_6$ whenever $(v[11/01],s^{p-1}x)$ is thin. But this second simplex is in $A_3$, where it is thin. If it includes $10$, then  $(v,s^{p-1}x)$ is thin if and only if $(v[10/00],s^{p-1}x)$ is thin, and so we come back to the last case. 

Now let $(v,x)$ be in $M_{B_0}$. Then $(v[10/00][11/21],x)$ is in $M_{A_3}$, and so in $M_6$. The two last paragraphs and \ref{lemma:oslash_saturation_two way} then imply that $(v,x)$ is in 
$D_6$.
\end{proof}

\begin{lemma}
The morphism 
$( X\times A_{23},M_{X\times A_{23}}) \to ( X\times B_1,M_{X\times B_1})$  is an acyclic cofibration.
\end{lemma}
\begin{proof}
The cofibration $( X\times A_{23},M_{X\times A_{23}}) \to ( X\times B_1,M_{X\times B_1})$ fits in the pushout square:
\[\begin{tikzcd}
	{\Lambda^{1}[2]\otimes \Delta[1]\cup\Delta[2]'\otimes \partial\Delta[1]} & {( X\times A_{23},M_{X\times A_{23}})} \\
	{\Delta[2]'\otimes \Delta[1]} & {( X\times B_1,M_{X\times B_1}).}
	\arrow[from=1-1, to=1-2]
	\arrow["{[023]\times[01]}"', from=2-1, to=2-2]
	\arrow[from=1-1, to=2-1]
	\arrow[from=1-2, to=2-2]
\end{tikzcd}\]
\end{proof}

\begin{lemma}
The morphism 
$( X\times A_{4},M_{X\times A_{4}}) \to ( X\times B_1,M_{X\times B_1})$  is an acyclic cofibration.
\end{lemma}
\begin{proof}

The cofibration $( X\times A_4,M_{X\times A_4})\to ( X\times B_1,M_{X\times B_1})$ is a sequence of inclusions:
$$
\begin{array}{l}
(X\times A_4,M_{X\times A_3}):=(D_0,M_0)\subset (D_1,M_1)\subset (D_2,M_2)\subset(D_3,M_3)\\
~~~~~~~~~~~~~~~~~~~~~~~~~~~~~\subset(D_4,M_4)\subset (D_5,M_5)\subset (D_6,M_6)\subset (X\times B_1,M_{X\times B_1})
\end{array}$$
where 
\begin{itemize}[leftmargin=* ,parsep=0cm,itemsep=0cm,topsep=0cm]
\item $D_1 = D_0\cup X\times [00,21, {31}]$ ;
\item $D_2 = D_1\cup X\times [20, {30},31]$ ;
\item $D_3 = D_2\cup X\times [20,21, {31}]$;
\item $D_4 = D_3\cup X\times [00,01,21, {31}]$;
\item $D_5 = D_4\cup X\times [00,20, {30},31]$ ;
\item $D_6 = D_5\cup X\times [00,20,21, {31}]$ ;
\end{itemize}
and
\begin{itemize}[leftmargin=* ,parsep=0cm,itemsep=0cm,topsep=0cm]
\item $(D_0,M_0)\to (D_1,M_1)$ is a pushout of $X~\rotatebox{90}{$\oslash$}~\Lambda^2[2]\to X~\rotatebox{90}{$\oslash$}~\Delta^2[2]$;
\item $(D_1,M_1)\to (D_2,M_2)$ is a pushout of $X~\rotatebox{90}{$\oslash$}~\Lambda^1[2]\to X~\rotatebox{90}{$\oslash$}~\Delta^1[2]$;
\item $(D_2,M_2)\to (D_3,M_3)$ is a pushout of $X~\rotatebox{90}{$\oslash$}~\Lambda^2[2]\to X~\rotatebox{90}{$\oslash$}~\Delta^2[2]$;
\item $(D_3,M_3)\to (D_4,M_4)$ is a pushout of $X~\rotatebox{90}{$\oslash$}~\Lambda^3[3]\to X~\rotatebox{90}{$\oslash$}~\Delta^3[3]$;
\item $(D_4,M_4)\to (D_5,M_5)$ is a pushout of $X\oslash\Lambda^2[3]\to X\oslash\Delta^2[3]$;
\item $(D_5,M_5)\to (D_6,M_6)$ is a pushout of $X~\rotatebox{90}{$\oslash$}~\Lambda^3[3]\to X~\rotatebox{90}{$\oslash$}~\Delta^3[3]$.
\end{itemize}

Now, we have to show that   $ M_{X\times B_1}\subset M_6$. 
First let's remarks that all  simplices $(v,s^{p}x)$ where $v_{[p,p+1]} = [21,31]$ are in $M_6$.  According to lemma \ref{lemma:oslash_saturation_two way_case_extremum}, a simplex $(v,x)$ is in $M_6$ thin if and only if $(v[21/31],x)$ is in $M_6$, if and only if  $(v[31/21],x)$ is in $M_6$.

Secondly, we show by case disjunction that all simplices $(v,s^{p-1}x)$ where $v_{[p-1,p]}=[20,30]$ are thin. If $v$ doesn't include $31$, it is in $M_4\subset M_6$. If $v$ includes $31$, then it is in $M_5\subset M_6$. According to lemma \ref{lemma:oslash_saturation_two way_case_extremum}  a simplex $(v,x)$ is in $M_6$ thin if and only if $(v[20/30],x)$ is in $M_6$, if and only if  $(v[30/20],x)$ is in $M_6$.

Now let $(v,x)$ be in $M_{B_1}$. Then $(v[20/00][21/31],x)$ is in $M_{A_4}\subset M_6$. The two last paragraphs then imply that $(v,x)$ is in 
$M_6$.
\end{proof}

\begin{lemma}
Maps $( X\times A_{012},M_{A_{012}})\to ( X\times B,M_{X\times B})\leftarrow ( X\times A_4,M_{X\times A_4})$
are acyclic cofibrations.
\end{lemma}
\begin{proof}
This is a direct consequence of last four lemmas.
\end{proof}

\begin{construction}
The marked simplicial set
$\underline{X\times B}$ is the pushout:
\[\begin{tikzcd}
	{[0,3]\otimes X\otimes ([0]\coprod [1])} & {(X\times B, M_B)} \\
	{[00,30]\coprod [01,31]} & {\underline{X\times B}.}
	\arrow[from=1-1, to=2-1]
	\arrow[from=2-1, to=2-2]
	\arrow[from=1-2, to=2-2]
	\arrow[""{name=0, anchor=center, inner sep=0}, from=1-1, to=1-2]
	\arrow["\lrcorner"{anchor=center, pos=0.125, rotate=180}, draw=none, from=2-2, to=0]
\end{tikzcd}\]

Let $\underline{X\times A_i}$ and $\underline{X\times B_i}$  be the sub-objects of $\underline{X\times B}$ corresponding to image of $(X\times A_i,M_{X\times A_i})$  and $(X\times B_i,M_{X\times B_i})$.

\end{construction}

\begin{lemma}
Morphisms $\underline{X\times A_0} \to \Sigma X\fwedge \Delta[1]$ and $\underline{X\times A_2} \to \Delta[1]\fwedge\Sigma X$ are acyclic cofibrations. 
\end{lemma}
\begin{proof}
We have a diagram:
\[\begin{tikzcd}
	{[0]\otimes X\otimes[0]\cup [0,1]\otimes X\otimes[1]} & {[00]\cup[01,11]} \\
	{[0]\otimes X\otimes[0,1]\cup [0,1]\otimes X\otimes [1]} & {\Sigma X\coprod_{\Delta[0]}\Delta[1]} \\
	{(X\times A_0,M_{X\times A_0})} & {\underline{X\times A_0}.}
	\arrow[""{name=0, anchor=center, inner sep=0}, from=1-1, to=1-2]
	\arrow["\sim", from=2-1, to=3-1]
	\arrow[from=1-1, to=2-1]
	\arrow[from=1-2, to=2-2]
	\arrow["\sim", from=2-2, to=3-2]
	\arrow[""{name=1, anchor=center, inner sep=0}, from=2-1, to=2-2]
	\arrow[from=3-1, to=3-2]
	\arrow["\lrcorner"{anchor=center, pos=0.125, rotate=180}, draw=none, from=2-2, to=0]
	\arrow["\lrcorner"{anchor=center, pos=0.125, rotate=180}, draw=none, from=3-2, to=1]
\end{tikzcd}\]
That shows that $\Delta[1]\coprod_{\Delta[0]} \Sigma X \to \underline{X\otimes A_0}$ is an acyclic cofibration. We then have a commutative diagram: 
\[\begin{tikzcd}
	{\Sigma  X\coprod_{\Delta[0]}\Delta[1]} & {\underline{X\times A_0}} & {\Sigma  X\fwedge\Delta[1]}
	\arrow["\sim", from=1-1, to=1-2]
	\arrow[from=1-2, to=1-3]
	\arrow["\sim", curve={height=-24pt}, from=1-1, to=1-3]
\end{tikzcd}\]
and by two out of three, $\underline{X\times A_0} \to \Sigma X\fwedge \Delta[1]$ is an acyclic cofibration. 
We proceed similarly for the second morphism. 	
\end{proof}

\begin{lemma}
Marked simplicial sets $\overline{X\otimes A_1}$ and  $\overline{X\otimes A_4}$ are respectively equal to $\Sigma (X\otimes \Delta[1])$ and $\Delta[1]\otimes \Sigma X$.
\end{lemma}
\begin{proof}
Analogue computation.
\end{proof} 

\begin{prop}
\label{prop:interval_second_formula}
There is a zigzag of acyclic cofibrations:
$$\Sigma X\fwedge\Delta[1]\coprod_{\Sigma X} \Sigma(\Delta[1]\otimes X)\coprod_{\Sigma X} \Delta[1]\fwedge\Sigma X \leftrightsquigarrow \Delta[1]\otimes \Sigma X.$$
\end{prop}
\begin{proof}
The acyclic cofibration $X\otimes A_{012}\to X\otimes B$ induces an acyclic cofibration
$$\Sigma X\fwedge\Delta[1]\coprod_{\Sigma X} \Sigma(\Delta[1]\otimes X)\coprod_{\Sigma X} \Delta[1]\fwedge\Sigma X\to \underline{X\times B},$$
and the acyclic cofibration  $X\otimes A_{4}\to X\otimes B$, an acyclic cofibration
$$\Delta[1]\otimes(\Sigma X)\to \underline{X\times B}.$$
\end{proof}

\subsubsection*{Summary of results}
\label{section:result_for_interval}
Propositions \ref{prop:interval_first_formula}, \ref{prop:interval_second_formula} and \ref{prop:lifting_property_zigzag_of_acyclic_cofibration} imply the following results:

Let $f:X\to Y$ be a fibration between $\infty$-categories and $K\to L$ a cofibration. Then:
\begin{itemize}
\item $f$ has the right lifting property against:
$$(\Sigma K)\otimes \Delta[1]\coprod_{(\Sigma K)\otimes \partial\Delta[1]} (\Sigma L)\otimes \partial\Delta[1]\to (\Sigma L)\otimes \Delta[1]$$ if and only if $f$ has the right lifting property against:
$$\Sigma (K\otimes \Delta[1])\coprod_{\Sigma (K\otimes \partial\Delta[1])} \Sigma (L\otimes \partial\Delta[1])\to \Sigma (L\otimes \Delta[1]).$$
\item $f$ has the right lifting property against:
$$\Delta[1]\otimes(\Sigma K)\coprod_{\partial\Delta[1]\otimes(\Sigma K)} \partial \Delta[1]\otimes (\Sigma L)\to \Delta[1]\otimes(\Sigma L) $$ if and only if $f$ has the right lifting property against:
$$\Sigma (\Delta[1]\otimes K )\coprod_{\Sigma (\Delta[1]\otimes K	)}\Sigma (\partial\Delta[1]\otimes L)\to \Sigma (\Delta[1]\otimes L).$$
\end{itemize}

\begin{construction}
\label{cons:composedwedge}
Let $K\to L$ be a cofibration. We define two important cofibrations: 
$\Delta[1]\fwedge \Sigma K\cup_{\triangledown} \Sigma L\to \Delta[1]\fwedge\Sigma L$ is obtained from the universal property of the pushout in the following diagram:
\[\begin{tikzcd}
	{\Sigma K} & {\Delta[1]\fwedge\Sigma K} \\
	{\Sigma L} & \Delta[1]\fwedge \Sigma K\cup_{\triangledown} \Sigma L \\
	&& {\Delta[1]\fwedge\Sigma L,}
	\arrow["\triangledown", from=1-1, to=1-2]
	\arrow[from=1-1, to=2-1]
	\arrow[from=1-2, to=2-2]
	\arrow[from=2-1, to=2-2]
	\arrow[from=2-2, to=3-3]
	\arrow[curve={height=-12pt}, from=1-2, to=3-3]
	\arrow["\triangledown"', curve={height=12pt}, from=2-1, to=3-3]
	\arrow["\lrcorner"{anchor=center, pos=0.125, rotate=180}, draw=none, from=2-2, to=1-1]
\end{tikzcd}\]
and the acyclic cofibration $\Sigma K\fwedge \Delta[1]\cup_{\triangledown} \Sigma L\to \Sigma L\fwedge\Delta[1]$ is obtained from the universal property of the pushout in the following diagram:
\[\begin{tikzcd}
	{\Sigma K} & {\Sigma K\fwedge\Delta[1]} \\
	{\Sigma L} & \Sigma K\fwedge \Delta[1]\cup_{\triangledown} \Sigma L \\
	&& {\Sigma L\fwedge\Delta[1].}
	\arrow["\triangledown", from=1-1, to=1-2]
	\arrow[from=1-1, to=2-1]
	\arrow[from=1-2, to=2-2]
	\arrow[from=2-1, to=2-2]
	\arrow[from=2-2, to=3-3]
	\arrow[curve={height=-12pt}, from=1-2, to=3-3]
	\arrow["\triangledown"', curve={height=12pt}, from=2-1, to=3-3]
	\arrow["\lrcorner"{anchor=center, pos=0.125, rotate=180}, draw=none, from=2-2, to=1-1]
\end{tikzcd}\]
\end{construction}

Suppose now that $f$ has the right lifting property against  $\Delta[1]\fwedge \Sigma K\cup_{\triangledown} \Sigma L\to \Delta[1]\fwedge\Sigma L$.
\begin{itemize}
\item 
$f$ has the right lifting property against:
$$(\Sigma K)\otimes \Delta[1]\coprod_{(\Sigma K)\otimes \{0\}} (\Sigma L)\otimes \{0\}\to (\Sigma L)\otimes \Delta[1]$$ if and only if $f$ has the right lifting property against:
$$\Sigma (K\otimes \Delta[1])\coprod_{\Sigma (K\otimes \{1\}} \Sigma (L\otimes \{1\})\to \Sigma (L\otimes \Delta[1]).$$
 
\item 
$f$ has the right lifting property against:
$$\Delta[1]\otimes(\Sigma K)\coprod_{\{0\}\otimes(\Sigma K)} \{0\}\otimes (\Sigma L)\to \Delta[1]\otimes(\Sigma L) $$  if and only if $f$ has the right lifting property against:
$$\Sigma (\Delta[1]\otimes K )\coprod_{\Sigma (\{0\}\otimes K)}\Sigma (\{0\}\otimes L)\to \Sigma (\Delta[1]\otimes L).$$
\end{itemize}

If $f$ has the right lifting property against $\Sigma K\fwedge \Delta[1]\cup_{\triangledown} \Sigma L\to \Sigma L\fwedge\Delta[1]$,
then 
\begin{itemize}
\item
$f$ has the right lifting property against:
$$(\Sigma K)\otimes \Delta[1]\coprod_{(\Sigma K)\otimes \{1\}} (\Sigma L)\otimes \{1\}\to (\Sigma L)\otimes \Delta[1]$$ if and only if $f$ has the right lifting property against:
$$\Sigma (K\otimes \Delta[1])\coprod_{\Sigma (K\otimes \{0\}} \Sigma (L\otimes \{0\})\to \Sigma (L\otimes \Delta[1]).$$
\item 
$f$ has the right lifting property against:
$$\Delta[1]\otimes(\Sigma K)\coprod_{\{1\}\otimes(\Sigma K)} \{1\}\otimes (\Sigma L)\to \Delta[1]\otimes(\Sigma L) $$ if and only if $f$ has the right lifting property against:
$$\Sigma (\Delta[1]\otimes K )\coprod_{\Sigma (\{1\}\otimes K)}\Sigma (\{1\}\otimes L)\to \Sigma (\Delta[1]\otimes L).$$
\end{itemize}

\subsection{Formulas for join}
 The purpose of this section is to show that we have zigzags of acyclic cofibrations, natural in $X$, fitting in:
$$
\begin{array}{rcl}
\Delta[0]\star\Sigma X & \leftrightsquigarrow & \Sigma (X\star \Delta[0]) \coprod_{\Sigma X} \Delta[1]\fwedge\Sigma X\\
\\
\Sigma X\costar \Delta[0] & \leftrightsquigarrow &  \Sigma X\fwedge \Delta[1] \coprod_{\Sigma X} \Sigma (X \costar \Delta[0]) \\
\\
\Delta[0]\costar \Sigma X & \leftrightsquigarrow & \Sigma(X\star\Delta[0]) \coprod_{\Sigma X}\Delta[1]\fwedge\Sigma X\\
\\
\Sigma X \star\Delta[0]  & \leftrightsquigarrow &  \Sigma X\fwedge \Delta[1]\coprod_{\Sigma X} \Sigma(\Delta[0]\costar  X)\\
\end{array} 
$$

\begin{proof}[Proof of the first formula]
We consider the diagram:
\[\begin{tikzcd}
	{\Delta[1]} & {\Sigma X\coprod_{\Delta[0]} \Delta[1]} & {\Sigma X\fwedge\Delta[1]\coprod_{\Sigma X} \Sigma(  \Delta[1]\otimes X) \coprod_{\Sigma X}\Delta[1]\fwedge\Sigma X} \\
	{\Delta[1]} & {\Sigma X\fwedge \Delta[1]} & {\Sigma X\fwedge\Delta[1]\coprod_{\Sigma X} \Sigma(  \Delta[1]\otimes X) \coprod_{\Sigma X}\Delta[1]\fwedge\Sigma X.}
	\arrow["id", from=1-3, to=2-3]
	\arrow[from=1-2, to=1-1]
	\arrow[from=2-2, to=2-1]
	\arrow[from=1-2, to=1-3]
	\arrow[from=2-2, to=2-3]
	\arrow["\sim"', from=1-2, to=2-2]
	\arrow["id"', from=1-1, to=2-1]
\end{tikzcd}\]
All vertical morphisms are weak equivalences.
We denote $A$ the colimit of the first line. The proposition \ref{prop:interval_first_formula} implies that there is a zigzag of acyclic cofibrations between $A$  and $\Delta[0]\diamond X$. Colimits of the two lines are homotopy colimits, and the comparison morphism is then an acyclic cofibration. 
We then have a zigzag of acyclic cofibrations: 

$$\Delta[0]\star X \leftarrow \Delta[0]\diamond X\leftrightsquigarrow A\to \Sigma(\Delta[0]\diamond  X)\coprod_{\sigma X} \Delta[1]\fwedge\Sigma X\to \Sigma(\Delta[0]\star  X)\coprod_{\sigma X} \Delta[1]\fwedge\Sigma X.$$

\end{proof}

The three other formulas are showed similarly.

\subsubsection*{Summary of results}
\label{section:recap_of_result_for_joint}
 Let $f:X\to Y$ be a fibration between $\infty$-categories and $K\to L$ a cofibration. Then:
\begin{itemize}
\item $f$ has the right lifting property against
$$\Delta[0]\star (\Sigma K)\coprod \Sigma L\to \Delta[0]\star (\Sigma L)$$
if and only if it has the right lifting property against
$$\Sigma (\Delta[0]\star K\coprod L)\to \Sigma (\Delta[0]\star L).$$
\item $f$ has the right lifting property against
$$\Delta[0]\costar   (\Sigma K)\coprod \Sigma L\to \Delta[0]\costar   (\Sigma L)$$
if and only if it has the right lifting property against
$$\Sigma (K\star \Delta[0]\coprod L)\to \Sigma  L\star\Delta[0]).$$
\item $f$ has the right lifting property against
$$(\Sigma K)\star \Delta[0]\coprod (\Sigma L)\to \Delta[0]\star\Sigma L$$
if and only if it has the right lifting property against
$$\Sigma (\Delta[0]\costar  K\coprod L)\to \Sigma (\Delta[0]\costar  L).$$
\item $f$ has the right lifting property against
$$(\Sigma K)\costar  \Delta[0]\coprod (\Sigma L)\to \Delta[0]\costar \Sigma L$$
if and only if it has the right lifting property against
$$\Sigma (K\costar  \Delta[0]\coprod L)\to \Sigma  L\costar \Delta[0]).$$
\end{itemize}

\subsection{Formula for the $\circ$-suspension}

\begin{construction}
Let $X$ be a simplicial set. We define the \textit{$\circ$-suspension of $X$} to be the colimit of the following diagram:
\[\begin{tikzcd}
	{\partial\Delta[1]\otimes X} & { \Delta[1]\otimes X} \\
	{\partial\Delta[1]} & {\Sigma^{\circ} X.}
	\arrow[from=1-2, to=2-2]
	\arrow["\lrcorner"{anchor=center, pos=0.125, rotate=180}, draw=none, from=2-2, to=1-1]
	\arrow[from=1-1, to=2-1]
	\arrow[from=2-1, to=2-2]
	\arrow[from=1-1, to=1-2]
\end{tikzcd}\]
This assignation defines a cocontinuous functor $\Sigma^\circ:\mSset\to \mSset_{\partial\Delta[1]/}.$ For every acyclic cofibration, we have a colimit diagram: 
\[\begin{tikzcd}
	{\partial\Delta[1]\otimes L} & {\Delta[1]\otimes K\cup \partial\Delta[1]\otimes L} & {\Delta[1]\otimes L} \\
	{\partial\Delta[1]} & {\Sigma^{\circ} K} & {\Sigma^{\circ} L.}
	\arrow[from=1-1, to=2-1]
	\arrow[""{name=0, anchor=center, inner sep=0}, from=1-1, to=1-2]
	\arrow[from=2-1, to=2-2]
	\arrow[from=1-2, to=2-2]
	\arrow[from=1-3, to=2-3]
	\arrow[from=2-2, to=2-3]
	\arrow[""{name=1, anchor=center, inner sep=0}, from=1-2, to=1-3]
	\arrow["\lrcorner"{anchor=center, pos=0.125, rotate=180}, draw=none, from=2-2, to=0]
	\arrow["\lrcorner"{anchor=center, pos=0.125, rotate=180}, draw=none, from=2-3, to=1]
\end{tikzcd}\]
This shows that the $\circ$-suspension is a left Quillen functor.

Furthermore, this functor admits a right adjoint, that sends a pair $(a,b,C)$ to $C(a,b)$ where $a,b$ are two $0$-simplices of $C$.
\end{construction}

\begin{prop}
\label{prop:formula_for_the_coop_suspension}
There is a zigzag of acyclic cofibrations between $\Sigma \Sigma^\circ X$ and $\Sigma^{\circ}\Sigma X$ natural in $X$.
\end{prop}
\begin{proof}
We consider the following diagram:
\[\begin{tikzcd}
	{\Delta[1]\coprod\Delta[1]} & {\Sigma X\coprod_{\Delta[0]}\Delta[1]\coprod  \Delta[1]\coprod_{\Delta[0]}\Sigma X} & {\Sigma X\fwedge\Delta[1]\coprod_{\Sigma X} \Sigma(\Delta[1]\otimes X)\coprod_{\Sigma X} \Delta[1]\fwedge\Sigma X} \\
	{\Delta[1]\coprod\Delta[1]} & {\Sigma X\fwedge\Delta[1]\coprod  \Delta[1]\fwedge\Sigma X} & {\Sigma X\fwedge\Delta[1]\coprod_{\Sigma X} \Sigma(\Delta[1]\otimes X)\coprod_{\Sigma X} \Delta[1]\fwedge\Sigma X.}
	\arrow[from=2-2, to=2-3]
	\arrow[from=2-2, to=2-1]
	\arrow["\sim", from=1-2, to=2-2]
	\arrow[from=1-2, to=1-1]
	\arrow[from=1-2, to=1-3]
	\arrow[from=1-3, to=2-3]
	\arrow[from=1-1, to=2-1]
\end{tikzcd}\]
All vertical morphisms are weak equivalences.
We denote $A$ the colimit of the first line. The proposition \ref{prop:interval_second_formula} implies that there is a zigzag of acyclic cofibrations between $A$  and $\Sigma^{\circ} \Sigma X$. The colimit of the second line is $\Sigma\Sigma^{\circ}X$. Colimits of the two lines are homotopy colimits, and the comparison morphism is then an acyclic cofibration. 
We then have a zigzag of acyclic cofibrations: 

$$\Sigma^{\circ} \Sigma X \leftrightsquigarrow A\to \Sigma \Sigma^{\circ} X.$$

\end{proof}

\section{Application I: Globular equivalences}

\subsection{Homotopy categories}

\begin{definition}
The \textit{simplicial $n$-globe} is the marked simplicial set $\Gb_n:=\Sigma^n \Delta[0]$.
$\Gb_0:=\Delta[0]$ and  $\Gb_{n+1}:= \Sigma \Gb_n$ 
This defines a globular object in $\mSset$:
\[\begin{tikzcd}
	{\Gb_0} & {\Gb_1} & \cdots & \cdots & {\Gb_n} & {\Gb_{n+1}} & \cdots
	\arrow["{\delta^0_1}", curve={height=-12pt}, from=1-1, to=1-2]
	\arrow["{\delta^1_1}"', curve={height=12pt}, from=1-1, to=1-2]
	\arrow["{\delta^0_{n+1}}", curve={height=-12pt}, from=1-5, to=1-6]
	\arrow["{\delta^1_{n+1}}"', curve={height=12pt}, from=1-5, to=1-6]
	\arrow["{\delta^0_{n+2}}", curve={height=-12pt}, from=1-6, to=1-7]
	\arrow["{\delta^1_{n+2}}"', curve={height=12pt}, from=1-6, to=1-7]
	\arrow["{\delta^0_2}", curve={height=-12pt}, from=1-2, to=1-3]
	\arrow["{\delta^0_2}"', curve={height=12pt}, from=1-2, to=1-3]
	\arrow["{\delta^0_n}", curve={height=-12pt}, from=1-4, to=1-5]
	\arrow["{\delta^0_n}"', curve={height=12pt}, from=1-4, to=1-5]
\end{tikzcd}\]
and we have equalities:
$$\delta^0_{n+1} \delta^0_n=\delta^1_{n+1} \delta^0_n~~~~\delta^1_{n+1} \delta^0_n=\delta^1_{n+1} \delta^1_n.$$
 We also set $(\Gb_n)_t:= \tau_n(\Gb_n)$. 
\end{definition}

\begin{definition}
Let $C$ be a $\infty$-category. A \textit{$n$-cell} $a$ of $C$ is a morphism $a:\Gb_n\to C$. If $n$ is non null, the \textit{source} of $a$ (resp. the \textit{target} of $a$)
is the $(n-1)$-cell $\delta^0_{n-1}(a)$ (resp. $\delta^1_{n-1}(a)$). The cell $a$ is thin if the corresponding morphism $\Gb_n\to C$ factorizes via $(\Gb_n)_t$.
\end{definition}

From now on, and until the end of this section, we fix an $\infty$-category $C$. All considered cells are cells of $C$.

\begin{definition}
Let $n$ be a non null integer, and $a,b$ two $n$-cells.
Cells $a$ and $b$  are \textit{parallel} if they share the same source and  the same target. They are \textit{composable} if the source of $a$ is the target of $b$.
\end{definition}

\begin{definition}
Let $a$ and $b$ be two parallel cells. The cell $a$ is \textit{equivalent} to the cell $b$ if there exists a thin $(n+1)$-cell $d:a\to b$, or equivalently, if there exists a homotopy $\Db_n\otimes I$ between $a$ and $b$, and  constant on $\partial \Db_n\otimes I$. This relation is denoted by $\sim$.
\end{definition}

\begin{lemma}
The relation $\sim$ is reflexive, symmetric and transitive. 
\end{lemma}
\begin{proof}
This comes from usual properties of fibrant objects.
\end{proof}

\begin{construction}
\label{cons:composition_homotopy_category}
Let $a,b$ be two composable $n$-cells . A composition of ${a}$ and ${b}$ is a $n$-cell $a\circ b$ that fits in a diagram:
\[\begin{tikzcd}
	{\Gb_n\coprod_{\Gb_{n-1}}\Gb_n} \\
	{\Sigma^{n-1}(\Delta[2]_t)} & C \\
	{\Gb_{n}.}
	\arrow[from=1-1, to=2-1]
	\arrow["{a\coprod b}", from=1-1, to=2-2]
	\arrow["{a\circ b}"', from=3-1, to=2-2]
	\arrow[from=3-1, to=2-1]
	\arrow[from=2-1, to=2-2]
\end{tikzcd}\]
As $C$ is a fibrant object, if $(a\circ b)'$ is any other composition, $(a\circ b)'\sim a\circ b$.
\end{construction}

\begin{lemma}
\label{lemma:associativity_of_composition_in_homotopy_category}
Let $a,b,c$ be three composable cells. There exists compositions such that $(a\circ b)\circ c = a\circ (b\circ c)$.
\end{lemma}
\begin{proof}
Let $M$ be the marking on $\Delta[3]$ that includes all simplices of dimension superior or equal to $2$. The cofibration $Sp^3\to (\Delta[3],M)$ is acyclic.
\[\begin{tikzcd}
	{\Gb_n\coprod_{\Gb_{n-1}}\Gb_n\coprod_{\Gb_{n-1}}\Gb_n} \\
	{\Sigma^{n-1}(\Delta[3],M)} & C \\
	{\Gb_{n}}
	\arrow[from=1-1, to=2-1]
	\arrow["{a\coprod b\coprod c}", from=1-1, to=2-2]
	\arrow["{a\circ b\circ c}"', from=3-1, to=2-2]
	\arrow["{\Sigma^{n-1}d^1d^1}", from=3-1, to=2-1]
	\arrow["f", from=2-1, to=2-2]
\end{tikzcd}\]
The morphism $f$ provides all desired compositions.
\end{proof}

\begin{definition}
We define the category $\pi_0(C)$ whose objects are $0$-cells $x:s\to t$, and edges between $x,y:s\to t$ are equivalence classes of the set of $1$-cells $f:x\to y$ quotiented by the relation $\sim$.  Let $n>0$ be an integer, and $s,t$ two parallel $(n-1)$-cells. We define the category $\pi_n(s,t,C)$ whose objects are $n$-cells $x:s\to t$, and edges between $x,y:s\to t$ are equivalence classes of the set of $n+1$-cells $f:x\to y$ quotiented by the relation $\sim$. 

The composition is given by construction \ref{cons:composition_homotopy_category} which is associative according to lemma \ref{lemma:associativity_of_composition_in_homotopy_category}.
\end{definition}

\begin{prop}
\label{prop:in_the_homotopy_category_thin_is_iso}
Let $x,y:s\to t$ be two parallel $n$-cells, and $f:x\to y$ a $n+1$-cell. The cell $f$ is thin if and only if $[f]:x\to y$ is a isomorphism in $\pi_n(s,t,C)$.
\end{prop}
\begin{proof}
Suppose first that $f$ is thin. There are liftings in the following diagrams:
\[\begin{tikzcd}
	{\Sigma^{n-1}\Lambda^0[2]} & C & {\Sigma^{n-1}\Lambda^2[2]} & C \\
	{\Sigma^{n-1}\Delta^0[2],} && {\Sigma^{n-1}\Delta^0[2].}
	\arrow["{f\amalg id}", from=1-1, to=1-2]
	\arrow[from=1-1, to=2-1]
	\arrow["h"', dotted, from=2-1, to=1-2]
	\arrow["{id\amalg f}", from=1-3, to=1-4]
	\arrow["k"', dotted, from=2-3, to=1-4]
	\arrow[from=1-3, to=2-3]
\end{tikzcd}\]	
Let $g:y\to z$ be the restriction of $h$ to $\Sigma^{n-1}[1,2]$ and $l:y\to z$ be the restriction of $h$ to $\Sigma^{n-1}[0,1]$. We then have $[f][g]= id$, and $[h][f]=id$,  and $[f]$ is then an isomorphism. 

In the other direction, suppose that $[f]$ is an isomorphism. Let $M$ be the marking on $\Delta[3]$ that includes all simplices of dimension superior or equal to $2$. There is a lifting in the following diagram:
\[\begin{tikzcd}
	{\Sigma^n([0,1]\coprod_{[1]}[1,2]\coprod_{[2]}[2,3])} && C \\
	{\Sigma^n(\Delta[3],M).}
	\arrow["{f^{-1}\amalg f\amalg f^{-1}}", from=1-1, to=1-3]
	\arrow[from=1-1, to=2-1]
	\arrow["h"', dotted, from=2-1, to=1-3]
\end{tikzcd}\]
Now $h(\Sigma^n[0,3])$ and $h(\Sigma^n[0,2])$ are respectively compositions of $(f,f^{-1})$ and $(f^{-1},f)$. Hypotheses imply that these compositions are equal to identities, and so are thin. The morphism then lifts to $\Sigma^n \Delta[3]^{eq}$. The object $C$ being fibrant, $h$ lifts to $\Sigma^n \Delta[3]^\sharp$, and $f$ is then thin.
\end{proof}

\begin{lemma}
\label{lemma:homotopycategory_are_idenpendant_of}
Let $s,t$ and $s',t'$ be two pairs of parallel cells, and $\psi:\partial Gb_n\times I\to C$ a homotopy between $s\cup t$ and $s'\cup t'$. Then 
$$\pi_n(s,t,C)\cong \pi_n(s',t',C)$$
\end{lemma}
\begin{proof}
For each $x:s\to t$, there exists a lifting   $h_x$  in the following diagram:
\[\begin{tikzcd}
	{\Gb_n\times\{0\}\cup\partial\Gb_n\times I} & C \\
	{\Gb_n\times I}
	\arrow[from=1-1, to=2-1]
	\arrow["x\cup\psi", from=1-1, to=1-2]
	\arrow["h"', dotted, from=2-1, to=1-2]
\end{tikzcd}\]
and we define $F(x)$ as the restriction of $h_x$ to $\Db_n\times \{1\}$. For a $(n+1)$-cell $f:x\to y$,  there exists a lifting $h_f$ in the following diagram:
\[\begin{tikzcd}
	{\Gb_{n+1}\times \{0\}\cup \partial\Gb_{n+1}\times I} & C \\
	{\Gb_{n+1}\times I}
	\arrow["{f\cup h_x\cup h_y}", from=1-1, to=1-2]
	\arrow[from=1-1, to=2-1]
	\arrow["{h_f}"', from=2-1, to=1-2]
\end{tikzcd}\]
and we define $F(f)$ as the restriction of $h_f$ to $\Db_{n+1}\times \{1\}$. Furthermore, we can check that $[F(f)]$ is independent of the choice of the lifting, and if $f\sim g$, $[F(f)]=[F(g)]$.
If $g:y\to z$ is an other morphism, and $\psi:\Sigma^n\Delta[2]_t \to C$ expressing the composition of $f$ and $g$, 
there is a lift in the following diagram:
\[\begin{tikzcd}
	{\Sigma^n \Delta[2]_t\cup (\Sigma^n\partial\Delta[2])\times I} && C \\
	{\Sigma^n \Delta[2]_t\times I.}
	\arrow["{h_f\cup h_g\cup h_{f\circ g}\cup \phi}", from=1-1, to=1-3]
	\arrow[from=1-1, to=2-1]
	\arrow[dotted, no head, from=2-1, to=1-3]
\end{tikzcd}\]
Restricted to $\Sigma^n \Delta[2]_t\times \{1\}$ this shows that $F$ commutes with compositions. We yet have to define a functor 
$$F:\pi_n(s,t,C)\to \pi_n(s',t',C).$$

Using exactly the same procedure, where we just inverse the role of $0$ and $1$, we define a functor:
$$G:\pi_n(s',t',C)\to \pi_n(s,t,C).$$
Now, we have a lift in the following diagram:
\[\begin{tikzcd}
	{\Gb_{n}\times \Lambda^{2}[2]^\sharp\cup\partial\Gb_n\times\Delta[2]^\sharp} &&& C \\
	{\Gb_n\times\Delta[2]^\sharp.}
	\arrow["{h_x\cup h_{F(x)}\cup\psi(id\times s^0)}", from=1-1, to=1-4]
	\arrow[from=1-1, to=2-1]
	\arrow["{k_x}"', dotted, from=2-1, to=1-4]
\end{tikzcd}\]
The restriction of $k_x$ to $\Db_n\times [0,1]_t$ provides a  thin cell $x\to G(F(x))$, which corresponds to an isomorphism in $\pi_n(s,t,C)$, according to proposition \ref{prop:in_the_homotopy_category_thin_is_iso}. If $f:x\to y$ is a $(n+1)$-cells, there is a lifting in the following diagram:
\[\begin{tikzcd}
	{\Gb_{n+1}\times \Lambda^{2}[2]^\sharp\cup\partial\Gb_{n+1}\times\Delta[2]^\sharp} &&& C \\
	{\Gb_{n+1}\times\Delta[2]^\sharp.}
	\arrow["{h_f\cup h_{F(f)}\cup k_x\cup k_y}", from=1-1, to=1-4]
	\arrow[from=1-1, to=2-1]
	\arrow["{k_f}"', dotted, from=2-1, to=1-4]
\end{tikzcd}\]
The restriction of $k_f$ to $\Gb_{n+1}\times[0,1]_t$ induces in $\pi_n(s,t,C)$ a commutative diagram:
\[\begin{tikzcd}
	x & GFx \\
	y & GFy.
	\arrow["{[GFf]}", from=1-2, to=2-2]
	\arrow["{[f]}"', from=1-1, to=2-1]
	\arrow[from=2-1, to=2-2]
	\arrow[from=1-1, to=1-2]
\end{tikzcd}\]
We then have a invertible natural transformation $\psi: id\to GF$. Similarly we can construct an other natural transformation $id\to GF$, which shows the desired equivalence of categories.
\end{proof}

\begin{definition}
Let $a$ be an element of the set of equivalence classes $[\partial \Gb_n, C]$. We define 
$$\pi_n(a, X) := \pi_n(s,t,C)$$
where $s,t$ is a pair of parallel arrows such that $s\cup t$ represents $a$.
The previous proposition shows that this is well defined.
\end{definition}
 
\subsection{A criterium to be a weak equivalence}

\begin{definition}
A maps $p:C\to D$ between $\infty$-categories is a \textit{$\Gb$-equivalence} if 
$$\pi_0(C)\to \pi_0(D)$$
is an equivalence of categories, and for all $n>0$ and pair of parallel arrow $s,t$, the induced functor
$$\pi_n(s,t,C)\to \pi_n(ps,pt,D)$$
is an equivalence of categories. 

A \textit{$\Gb$-trivial fibration} is a fibration having the right lifting property against $\partial\Gb_n\to \Gb_n$ and $\Gb_n\to (\Gb_{n})_t$.
\end{definition}

\begin{lemma}
\label{lemma:fibration_are_isofibration}
Let $\alpha\in\mathbb{Z}/2\mathbb{Z}$.
The morphism $\delta^\alpha_{n+1}:\Gb_n\to (\Gb_{n+1})_t$ is an acyclic cofibration. 
\end{lemma}
\begin{proof}
We have a pushout diagram
\[\begin{tikzcd}
	{\Gb_n\times\{\alpha\}\cup\partial\Gb_n\times I} & {\Gb_n\times \{\alpha\}} \\
	{\Gb_n\times I} & {(\Gb_n)_t.}
	\arrow[""{name=0, anchor=center, inner sep=0}, "{id\cup \partial\times s^0}", from=1-1, to=1-2]
	\arrow[from=1-1, to=2-1]
	\arrow[from=2-1, to=2-2]
	\arrow["{\delta^\alpha_{n+1}}", from=1-2, to=2-2]
	\arrow["\lrcorner"{anchor=center, pos=0.125, rotate=180}, draw=none, from=2-2, to=0]
\end{tikzcd}\]
The left hand morphism being an acyclic cofibration, this concludes the proof.
\end{proof}

\begin{lemma}
\label{lemma:acyclic_cofibration_are_G_equivalence}
Acyclic cofibrations between $\infty$-categories are $\Gb$-equivalences.
\end{lemma}
\begin{proof}
Let $i:A\to B$ be an acyclic cofibration. Let $r:B\to A$ be a retraction: 
\[\begin{tikzcd}
	A & A \\
	B.
	\arrow["id", from=1-1, to=1-2]
	\arrow["i"', from=1-1, to=2-1]
	\arrow["r"', from=2-1, to=1-2]
\end{tikzcd}\]
There exists a homotopy  $\psi$  between $id_A$ and $ir$. 
The morphism $\pi_n(s,t,A)\xrightarrow{i_!} \pi_n(a,B)$  admits a retraction, and is then faithfull.  Let $y$ be an object of $\pi_n(is,it,B)$. The homotopy $\psi$ induces  the existence of a thin cell $ir(y)\to y$ which corresponds to an isomorphism in $\pi_n(is,it,B)$, according to proposition \ref{prop:in_the_homotopy_category_thin_is_iso}. The functor $i_!$ is then essentially surjective. Eventually, for all $f:i(x)\to i(y)$, $\psi$ induces an equivalence $ir(f)\sim f$. The functor $i_!$ is then fully faithfull. 
\end{proof}

\begin{lemma}
\label{lemma:2_out_of_3_for_G_equivalence}
If we have a commutative diagram: 
\[\begin{tikzcd}
	& B \\
	A && C
	\arrow["g"', from=2-1, to=2-3]
	\arrow["f", from=1-2, to=2-3]
	\arrow["i", from=2-1, to=1-2]
\end{tikzcd}\]
where $i$ is an acyclic cofibration, and $g$ is a $\Gb$-equivalence, then $f$ is a $\Gb$-equivalence.
\end{lemma}
\begin{proof}
Let $s,t$ be any pair of parallel arrows in $B$. There exists a pair of parallel arrows $s',t'$ in $A$ such that $s\cup t$ and $is'\cup it'$ correspond to the same element in $[\partial\Gb_n,B]$.  We then have a diagram:
\[\begin{tikzcd}
	& { \pi(s,t,B)} & {\pi(fs,ft,C)} \\
	{\pi(s,t,B)} & { \pi(is,it,B)} & {\pi(gs,gt,C).}
	\arrow["\sim", from=2-1, to=2-2]
	\arrow[from=2-2, to=2-3]
	\arrow["\sim", from=1-3, to=2-3]
	\arrow["\sim", from=1-2, to=2-2]
	\arrow[from=1-2, to=1-3]
	\arrow["\sim"', curve={height=18pt}, from=2-1, to=2-3]
\end{tikzcd}\]
where  arrows labeled by $\sim$ are isomorphisms according to \ref{lemma:homotopycategory_are_idenpendant_of} and \ref{lemma:acyclic_cofibration_are_G_equivalence}. By two out of three, that shows that $ \pi(s,t,B)\to  \pi(fs,ft,C)$ is an isomorphism, and $f$ is then a $\Gb$ equivalence.
\end{proof}

\begin{prop}
\label{prop:caracterisation_of_G_fibration}
A map is a $\Gb$-trivial fibration if and only if it is a fibration and a $\Gb$-equivalence.
\end{prop}
\begin{proof}
If $p$ is a $\Gb$-trivial fibration, it is obvious that it is both a fibration and a $\Gb$-equivalence. For the converse, suppose $p$ is a fibration and a $\Gb$-equivalence, and consider a diagram 
\[\begin{tikzcd}
	{\partial\Gb_n} & X \\
	{\Gb_n} & Y
	\arrow[from=1-1, to=2-1]
	\arrow[from=1-1, to=1-2]
	\arrow["x"', from=2-1, to=2-2]
	\arrow["p", from=1-2, to=2-2]
\end{tikzcd}\]
As $p$ is an $\Gb$-equivalence this implies that there exists a cell $\overline{x}:\Gb_n\to X$ together with a thin $(n+1)$-cell $y:p(\overline{x})\to y$. All this data corresponds to a diagram:
\[\begin{tikzcd}
	{\Gb_n} & X \\
	{(\Gb_{n+1})_t} & Y
	\arrow["p", from=1-2, to=2-2]
	\arrow["{\delta^0_{n+1}}"', from=1-1, to=2-1]
	\arrow["{\bar{x}}", from=1-1, to=1-2]
	\arrow["y"', from=2-1, to=2-2]
\end{tikzcd}\]
The left hand morphism being an acyclic cofibration according to \ref{lemma:fibration_are_isofibration}, this diagram admits a lift $h:(\Gb_{n+1})_t\to X$. The restriction of $h$ to $\delta^1_{n+1}$ provides a lift in the first diagram.  Now, we consider a diagram of shape: 

\[\begin{tikzcd}
	{\Gb_n} & X \\
	{(\Gb_n)_t} & Y
	\arrow[from=1-1, to=2-1]
	\arrow["g", from=1-1, to=1-2]
	\arrow[from=2-1, to=2-2]
	\arrow["p", from=1-2, to=2-2]
\end{tikzcd}\]
with $n>1$.
Let $s,t$ be respectively the $(n-2)$-source and the $(n-2)$-target of $g$. Hypotheses imply that $[p(g)]$ is an isomorphism in $\pi_n(s,t,C)$ and because $p$ is a $\Gb$-equivalence, so is $[g]$. According to lemma \ref{prop:in_the_homotopy_category_thin_is_iso}, this implies that $g$ is thin. There exists then a lifting in the previous diagram. The case $n=1$ is similar.
  The morphism $f$ is then a $\Gb$-trivial fibration.
\end{proof}

\begin{lemma}
\label{lemma:slice_G_fibrations}
Let $p:X\to Y$ be a $\Gb$-trivial fibration between $\infty$-categories. Then for all $x\in X$, the induced fibration: 
$$X_{/x}\to X\times_Y Y_{/p(x)}$$
is a $\Gb$-trivial fibration.
\end{lemma}
\begin{proof}
We define $\mathbb{P}(p,n)$ to be the statement that $p$ has the right lifting property against 
$$\Delta[0]\star \partial \Gb_n\cup \Gb_n\to \Delta[0]\star  \Gb_{n+1} \mbox{ and }\Delta[0]\star \Gb_n\cup (\Gb_n)_t\to \Delta[0]\star (\Gb_n)_t.$$
First, it is obvious that each $\Gb$-equivalence $p$ satisfies $\mathbb{P}(p,0)$. 
Secondly, according to results \ref{section:recap_of_result_for_joint}, $\mathbb{P}(p,n+1)$ is equivalent to $\mathbb{P}(p(a,b),n)$ for all $a,b\in C_0$. 

Using the fact that $p(a,b)$ is $\Gb$-trivial fibration as soon as $p$ is, a direct induction on $n$ shows the desired result.
\end{proof}

\begin{lemma}
\label{lemma:G_fibration_right lifting property_against_partial}
 $\Gb$-Trivial fibrations  between $\infty$-categories have the right lifting property against $\partial\Delta[n]\to \Delta[n]$.
\end{lemma}
\begin{proof}
Let $C$ be the class of cofibrations having the left lifting property against $\Gb$-equivalences. The lemma \ref{lemma:slice_G_fibrations} implies that for all 
 $K\to L$ in $C$, the induced morphism:
$$\Delta[0]\star K\cup L\to \Delta[0]\star L$$
is in $C$. 
The class $C$ is then closed by Leibniz join. Furthermore, it includes  $\partial\Delta[1]\to \Delta[1]$, and then, by induction, it includes  $\partial\Delta[n]\to\Delta[n]$ for all integer $n$.
\end{proof}

\begin{lemma}
\label{lemma:G_fibration_right lifting property_against_sat}
 $\Gb$-Trivial fibrations  between $\infty$-categories have the right lifting property against $\Delta[n]\to \Delta[n]_t$.
\end{lemma}
\begin{proof}
Let $p$ be $\Gb$-trivial fibrations  between $\infty$-categories, and
 $C_{n,p}$ be the set of objects $A$ such that $p$ has the right lifting property against:
$$A\to \tau_n(A).$$
This set is then closed by colimits, and by zigzags of acyclic cofibrations.
Let $k\leq n$ be two integers. We define  $\mathbb{P}(k,n)$ to be the statement that 
$$\Delta[k-1]\star \Sigma \Delta[n-k]$$
is in $C_{n+1,p}$.
The statement  $\mathbb{P}(0,1,f)$ corresponds to the belonging of $\Gb_2$ to $C_{2,p}$, which is obviously true. Suppose that $0<k$ and $\mathbb{P}(k-1,n,p)$. 
According to \ref{section:recap_of_result_for_joint}, the object $\Delta[k-1]\star \Sigma \Delta[n-k]$ is  linked by a zigzag of acyclic cofibrations to the colimit of 
\[\begin{tikzcd}
	{\Delta[k-2]\star(\Sigma \Delta[n-k+1])} & {\Delta[k-2]\star (\Sigma\Delta[n-k])} & {\Delta[k-2]\star (\Delta[0]\fwedge\Sigma\Delta[n-k]).}
	\arrow[from=1-2, to=1-3]
	\arrow[from=1-2, to=1-1]
\end{tikzcd}\]
The center object and the right hand object are in $C_{n+1,p}$ because there are $(n+1)$-trivial, and the 
 left hand object is in $C_{n+1,p}$ by induction hypothesis. This then implies $P(k,n,p)$. Eventually, $P(0,n+1,p)$ is equivalent to $\mathbb{P}(n,n,p(a,b))$ for all pair of objects $(a,b)\in X_0$.
The statement $\mathbb{P}(k,n,p)$ is then true for all $k,n$ and $\Gb$-trivial fibrations between $\infty$-categories $p$. This implies that $p$  has the right lifting property against $\Delta[n]\to \Delta[n]_t$.
\end{proof}

\begin{theorem}
\label{theo:f_weak_equivalence_ssi_f_G_equivalence}
Let $p$ be a map between $\infty$-categories. Then $p$ is a weak equivalence if and only if it is a $\Gb$-equivalence.
\end{theorem}
\begin{proof}
According to lemmas   \ref{lemma:acyclic_cofibration_are_G_equivalence} and \ref{lemma:2_out_of_3_for_G_equivalence} we can restrict ourselves to the case where $p$ is a fibration. If it is a weak equivalence, $p$ is then a trivial fibration and then a $\Gb$-equivalence. Suppose now that $p$ is a $\Gb$-trivial fibration. According to proposition \ref{prop:caracterisation_of_G_fibration}, $p$ is then a $\Gb$-trivial fibration. Lemmas \ref{lemma:G_fibration_right lifting property_against_partial} and \ref{lemma:G_fibration_right lifting property_against_sat} imply that $p$ is a trivial fibration.
\end{proof}

\begin{definition}
Let $p:X\to Y$ be a morphism between $\infty$-categories. The morphism $p$ is \textit{essentially surjective} if for all $x\in Y_0$, there exists $\bar{x}\in X_0$ together with a thin cell $\bar{x}\to x$.
The morphism $f$ is \textit{fully faithfull} if the induced morphisms: 
$$X(a,b)\to Y(pa,pb)$$
are weak equivalences for all $a,b\in X_0$.
\end{definition}

\begin{cor}
Let $p$ be a map between $\infty$-categories. Then $p$ is fully faithfull and essentially surjective if and only if it is a weak equivalence.
\end{cor}
\begin{proof}
If $p$ is a weak equivalence, it is then fully faithfull and essentially surjective. Conversely, suppose $p$ is fully faithfull and essentially surjective. 
The morphism $\pi_0(X)\to \pi_0(Y)$ is fully faithfull and essentially surjective, and then an equivalence of category. For $(a,b)$ a pair of $0$-cells, we have equalities:
\[\begin{tikzcd}
	{\pi_1(a,b,X)} & {\pi_0(X(a,b))} \\
	{\pi_1(pa,pb,Y)} & {\pi_0(Y(pa,pb)).}
	\arrow["{\pi_0p(a,b)}", from=1-2, to=2-2]
	\arrow["{\pi_1p}"', from=1-1, to=2-1]
	\arrow[Rightarrow, no head, from=2-1, to=2-2]
	\arrow[Rightarrow, no head, from=1-1, to=1-2]
\end{tikzcd}\]
The morphism $\pi_1(a,b,p)$ is then an equivalence of categories. For $(s,t)$ a pair of parallel arrows of dimension $>1$, if we denote $a$ and $b$ the $0$-source  and the $0$-target of $s$ and $t$, we have a diagram:
\[\begin{tikzcd}
	{\pi_n(s,t,X)} & {\pi_{n-1}(s,t,X(a,b))} \\
	{\pi_n(pa,pb,Y)} & {\pi_{n-1}(s,t,Y(pa,pb)).}
	\arrow["{\pi_{n-1}(s,t ,p(a,b))}", from=1-2, to=2-2]
	\arrow["{\pi_np}"', from=1-1, to=2-1]
	\arrow[Rightarrow, no head, from=2-1, to=2-2]
	\arrow[Rightarrow, no head, from=1-1, to=1-2]
\end{tikzcd}\]
The morphism $\pi_n(a,b,p)$ is then an equivalence of categories.
The morphism $p$ is then a $\Gb$-equivalence, and according to \ref{theo:f_weak_equivalence_ssi_f_G_equivalence}, a weak equivalence.
\end{proof}

\subsection{A criterium to be a weakly invertible transformation}
\label{section:A criterium to be a weakly invertible transformation}
The purpose of this section is to show   the next proposition:
\begin{prop}
\label{prop:criterimu_to_be_an_weak_equivalence}
Let $i$ and $j$ be two left Quillen adjoints and $\psi:i\to j$ a natural transformation. If 
$\psi(\Gb_n):i (\Gb_n) \to j (\Gb_n)$ is a weak equivalence, then $\psi(X):i(X)\to j(X)$ is a weak equivalence for all $X$.
\end{prop}
For the remaining of this section, we fix two left Quillen functors $i$, $j$ and a natural transformation $\psi:i\to j$ satisfying previous hypotheses. We denote $N_i$ and $N_j$ the associated right adjoints.

\begin{lemma}
\label{lemma:psipartial_is_a_weak_equivalence}
Morphisms $\psi(\partial\Gb_n):i(\partial\Gb_n)\to j(\partial\Gb_n)$ are weak equivalences. 
\end{lemma}
\begin{proof}
Objects $i(\partial\Gb_n)$ (resp $i(\partial\Gb_n)$ ) are homotopy colimits of $\Gb_{/\partial Gb_n}\xrightarrow{i} mSSet$ (resp. $\Gb_{/\partial Gb_n}\xrightarrow{j} mSSet$) and $\psi$ induces a natural transformation which is a pointwise weak equivalence between these two diagrams. The induced map 
$\psi(\partial\Gb_n):i(\partial\Gb_n)\to j(\partial\Gb_n)$ is then a weak equivalence.
\end{proof}

\begin{lemma}
\label{lemma:psisat_is_a_weak_equivalence}
Morphisms $\psi((\Gb_n)_t):i((\Gb_n)_t)\to j((\Gb_n)_t)$ are weak equivalences. 
\end{lemma}
\begin{proof}
There is a diagram:
\[\begin{tikzcd}
	& {i\Gb_{n-1}} && {j\Gb_{n-1}} \\
	{} & {i(\Gb_n)_t} && {j(\Gb_n)_t.}
	\arrow["{\psi(\Gb_n)}", from=1-2, to=1-4]
	\arrow["{i\delta^0_n}"', from=1-2, to=2-2]
	\arrow["{j\delta^0_n}", from=1-4, to=2-4]
	\arrow["{\psi((\Gb_n)_t)}"', from=2-2, to=2-4]
	\arrow["\sim"', draw=none, from=1-2, to=1-4]
	\arrow["\sim"', draw=none, from=1-4, to=2-4]
	\arrow["\sim", draw=none, from=1-2, to=2-2]
\end{tikzcd}\]
By two out of three, this shows that $\psi((\Gb_n)_t)$ is a weak equivalence.
\end{proof}

\begin{lemma}
If $\psi$ is a pointwise cofibrations , $N_jY\to N_i Y$ is a fibration for all fibrant $Y$.
\end{lemma}
\begin{proof}
First, for all trivial cofibration $K\to L$, we have a square:
\[\begin{tikzcd}
	{iK} & {iL} \\
	{jK} & {jL}
	\arrow[from=1-1, to=2-1]
	\arrow["\sim"', from=2-1, to=2-2]
	\arrow["\sim", from=1-1, to=1-2]
	\arrow[from=1-2, to=2-2]
\end{tikzcd}\]
where all horizontal morphisms are weak equivalences and all map are cofibrations. That shows that $iL\cup jK\to j L$ is a trivial cofibration. 
Now, there is a lifting in the left diagram if and only if there is a lifting in the right diagram.
\[\begin{tikzcd}
	K & {N_jY} & {iL\cup jK} & Y \\
	L & {N_i Y} & {jL.}
	\arrow[""{name=0, anchor=center, inner sep=0}, from=1-2, to=2-2]
	\arrow[from=1-1, to=1-2]
	\arrow[from=1-1, to=2-1]
	\arrow[from=2-1, to=2-2]
	\arrow["\exists"{description}, dotted, from=2-1, to=1-2]
	\arrow[""{name=1, anchor=center, inner sep=0}, from=1-3, to=2-3]
	\arrow[from=1-3, to=1-4]
	\arrow["\exists"{description}, dotted, from=2-3, to=1-4]
	\arrow[shorten <=13pt, shorten >=13pt, Rightarrow, 2tail reversed, from=0, to=1]
\end{tikzcd}\]
\end{proof}

\begin{lemma}
\label{lemma:j*is_a_trivial_fibration}
If $\psi$ is a pointwise cofibration, and $\psi(\Gb_n)$ is a weak equivalence for all $n$,  $N_jY\to N_i Y$ is a trivial fibration for all fibrant $Y$.
\end{lemma}
\begin{proof}
According to theorem \ref{theo:f_weak_equivalence_ssi_f_G_equivalence}, it is enough
to show that for all fibrant $Y$,  $N_j Y\to N_i Y$ is a $\Gb$-trivial fibration.
Let $K\to L$ be either $\partial\Gb_n\to \Gb_n$ or $\Gb_n\to (\Gb_n)_t$.
We have a diagram:
\[\begin{tikzcd}
	{iK} & {iL} \\
	{jK} & {jL}
	\arrow["\sim"', from=1-1, to=2-1]
	\arrow[from=1-1, to=1-2]
	\arrow[from=2-1, to=2-2]
	\arrow["\sim", from=1-2, to=2-2]
\end{tikzcd}\]
where all maps are cofibrations, and vertical ones are weak equivalences according to lemma \ref{lemma:psipartial_is_a_weak_equivalence} and \ref{lemma:psisat_is_a_weak_equivalence}. We then have a trivial cofibration $iL\cup jK\to jL$.
Now, there is a lifting in the left diagram if and only if there is a lifting in the right diagram.
\[\begin{tikzcd}
	K & {N_j Y} & {iL\cup jK} & Y \\
	L & {N_i Y} & {jL}
	\arrow[""{name=0, anchor=center, inner sep=0}, from=1-2, to=2-2]
	\arrow[from=1-1, to=1-2]
	\arrow[from=1-1, to=2-1]
	\arrow[from=2-1, to=2-2]
	\arrow["\exists"{description}, dotted, from=2-1, to=1-2]
	\arrow[""{name=1, anchor=center, inner sep=0}, from=1-3, to=2-3]
	\arrow[from=1-3, to=1-4]
	\arrow["\exists"{description}, dotted, from=2-3, to=1-4]
	\arrow[shorten <=12pt, shorten >=12pt, Rightarrow, 2tail reversed, from=0, to=1]
\end{tikzcd}\]
which is the case because the left hand morphism is an acyclic cofibration, and  $Y$ is fibrant. 

\end{proof}

\begin{proof}[Proof of the  proposition \ref{prop:criterimu_to_be_an_weak_equivalence}]
We can suppose without lost of generality that $\psi$ is a pointwise cofibration. 
Let $X$ be any marked simplicial set and $Y$ an $\infty$-category. We have equalities:
\[\begin{tikzcd}
	{[jX,Y]} & {[X,N_j Y]} \\
	{[iX,Y]} & {[X,N_i Y].}
	\arrow[from=1-2, to=2-2]
	\arrow[from=1-1, to=2-1]
	\arrow[Rightarrow, no head, from=1-1, to=1-2]
	\arrow[Rightarrow, no head, from=2-1, to=2-2]
\end{tikzcd}\]
Lemma \ref{lemma:j*is_a_trivial_fibration} implies that  the right hand morphism is a bijection, and so is the left hand morphism. 
For all $X$, $\psi(X)$ is then a weak equivalence.
\end{proof}

\subsection{Weak characterization of the $l$-suspension}

Let $l$ be an integer. For the rest of this section, we fix a left Quillen functor $i:\textbf{mSset}\to \textbf{mSset}$ such that there exists a zigzag of weakly invertible natural transformations:
$$i(\Gb_{\uvar}) \leftrightsquigarrow \Gb_{l+\uvar}.$$

\begin{lemma}
Let $n$ be any integer, the following natural transformations are pointwise acyclic cofibrations:
$$i\tau_n\to \tau_{n+l}i\tau_n \leftarrow \tau_{n+l}i.$$
\end{lemma}
\begin{proof}
These are natural transformations between left Quillen functors. They induce weak equivalences on globes of dimension strictly inferior to $n$.  Let $k\geq n$. We have a commutative diagram: 
\[\begin{tikzcd}
	{i\tau_n(\Gb_k)} & {\tau_{n+l}i \tau_n(\Gb_k)} & {\tau_{n+l}i(\Gb_k)} \\
	& {i(\Gb_{n-1})}
	\arrow["\sim"', from=2-2, to=1-1]
	\arrow["\sim"', from=2-2, to=1-2]
	\arrow["\sim", from=2-2, to=1-3]
	\arrow[from=1-1, to=1-2]
	\arrow[from=1-3, to=1-2]
	\arrow["{\delta^0_k}", draw=none, from=2-2, to=1-1]
	\arrow["{\delta^0_k}", draw=none, from=2-2, to=1-2]
	\arrow["{\delta^0_k}"', draw=none, from=2-2, to=1-3]
\end{tikzcd}\]
By two out of three, this implies that theses natural transformations induce weak equivalences on all globes, and  proposition \ref{prop:criterimu_to_be_an_weak_equivalence} concludes the proof.
\end{proof}

\begin{prop}
\label{prop:modification_of_the_value_on_thin_representables}
There exist a left Quillen functor $j$ such that $j(\Delta[n])=i(\Delta[n])$ and $j(\Delta[n]_t)=\tau_{n+l}i(\Delta[n])$ together with a zigzag of weakly invertible natural transformations 
$$i\leftrightsquigarrow j.$$
\end{prop} 
\begin{proof}
We define $\tilde{i}$ (resp. $j$) to be the colimit preserving functor defines on representables by  $\tilde{i}(\Delta[n]):=i(\Delta[n])$ and $\tilde{i}:=(\Delta[n]_t)=\tau_{n+l}i(\Delta[n]_t)$ (resp.  $j(\Delta[n]):=i(\Delta[n])$ and $j(\Delta[n]_t):=\tau_{n+l}i(\Delta[n])$). We then have a zigzag of natural transformations that are pointwise acyclic cofibrations:
$$i\xrightarrow{\sim} \tilde{i}\xleftarrow{\sim} j.$$
This implies that both $\tilde{i}$ and $j$ are left Quillen functors.
\end{proof}

Let $C$ be the subcategory of marked simplicial set such that $\tilde{R}(iX) =\tilde{R}(\Sigma^lX)$.
We define the statement $\mathbb{P}(k,n,m)$ as the conjunction of the following assertions:
\begin{enumerate}
\item  $\Sigma^m(\Delta[k-1]\star \Sigma \Delta[n-k])$ is in $C$.
\item All monomorphisms:
$\Sigma^m(\Delta[l]\star \Sigma \Delta[p])\to \Sigma^m(\Delta[k-1]\star \Sigma \Delta[n-k])$
are in $C$.
\item All monomorphisms:
$\Sigma^m(\Delta[l]\star   \{\alpha\}) \to \Sigma^m(\Delta[k-1]\star \Sigma\Delta[n-k])$
are in $C$, where $\alpha\in\mathbb{Z}/2\mathbb{Z}$.
\end{enumerate}

\begin{lemma}
\label{lemma:about_P}
If for all $k,n,m$ $\mathbb{P}(k,n,m)$, $C$ includes $\Delta^+$. 
\end{lemma}
\begin{proof}
It is clear that hypotheses imply that $C$ includes  $(\Delta^+)_m$, the sub category of $\Delta^+$ whose morphisms are monomorphisms. We yet suppose that for all $k<n$, morphisms $\Delta[k]\to \Delta[l]$ are in $C$. Let $j:\Delta[n]\to \Delta[l]$ be any morphism which is not a monomorphism.  We have a \textit{a priori} non commutative diagram:
\[\begin{tikzcd}
	{\tilde{R}(i\partial \Delta[n-1])} & {\tilde{R}(\partial\Delta[n-1])} \\
	{\tilde{R}(i\Delta[n])} & {\tilde{R}(\Delta[n])} \\
	{\tilde{R}(i\Delta[l])} & {\tilde{R}(\Delta[l])}
	\arrow[from=2-1, to=3-1]
	\arrow[Rightarrow, no head, from=2-1, to=2-2]
	\arrow[from=1-1, to=2-1]
	\arrow[from=1-2, to=2-2]
	\arrow[from=2-2, to=3-2]
	\arrow[Rightarrow, no head, from=3-1, to=3-2]
	\arrow[Rightarrow, no head, from=1-1, to=1-2]
\end{tikzcd}\]
by induction hypothesis, the outer diagram commutes. For the downer diagram to commutes, we have to check that the top cell of $\tilde{R}(i\Delta[n])$  is sent on the same element on $\tilde{R}(\Delta[l])$. This is the case because the two paths send it to an unity. 
\end{proof}

\begin{prop}
\label{prop:existence_of_comparaison_with_street}
We have the equality
 $\tilde{R}i = \tilde{R}\Sigma^l$.
\end{prop} 
\begin{proof}
The sub category $C$ is closed by colimits and by zigzags of acyclic cofibrations. According to \ref{lemma:about_P}, it is enough to show $\mathbb{P}(k,n,m)$ for all $k,n,m$. We will proceed by induction.
The property $\forall_m\mathbb{P}(0,1,m)$ corresponds to the belonging of globes to $C$, which is true.

Property  $\forall_m\mathbb{P}(0,n,m)$ is equivalent to $\forall_m\mathbb{P}(n-1,n-1,m+1)$. Suppose $\forall_m\mathbb{P}(k-1,n,m)$ and $\forall_m\mathbb{P}(k-1,n-1,m)$  for $0<k\leq n$.
Induction hypothesis and lemma \ref{lemma:unicity_of_composition} imply that the morphism   
$$\Sigma^m ( \Delta[k-2]\star (\Sigma\Delta[n-k]))\to \Sigma^m (\Delta[k-2]\star (\Delta[0]\fwedge\Sigma\Delta[n-k]))$$ 
is in $C$. 
 
Results \ref{section:recap_of_result_for_joint}, implies that the object $\Sigma^m(\Delta[k-1]\star \Sigma \Delta[n-k])$ is  linked by a zigzag of acyclic cofibrations to the colimit of 
$$\Sigma^m (\Delta[k-2]\star(\Sigma \Delta[n-k+1]))\leftarrow  \Sigma^m (\Delta[k-2]\star (\Sigma\Delta[n-k])) \to \Sigma^m (\Delta[k-2]\star (\Delta[0]\fwedge\Sigma\Delta[n-k]))$$
This implies  $\mathbb{P}(k,n,m)$.
\end{proof}

\begin{theorem}
\label{theo:criterium_to_be_linked_to_identity}
Let $i:\textbf{mSset}\to \textbf{mSset}$ be a left Quillen functor. Suppose that there exists a zigzag of weakly invertible natural transformations:
$$i(\Gb_{\uvar}) \leftrightsquigarrow \Gb_{l+\uvar}.$$
Then, there exists a zigzag of weakly invertible natural transformations between $i$ and $\Sigma^l$.
\end{theorem} 
\begin{proof} 
Accorded to proposition \ref{prop:modification_of_the_value_on_thin_representables}, we can suppose that $i(\Delta[n]_t)=\tau_{n+l}(i(\Delta[n]))$.
The proposition \ref{prop:existence_of_comparaison_with_street} then  implies that we have a natural transformation $\psi:i\to i^l_{str}$. 
Furthermore, hypotheses imply that this natural transformation is a weak equivalence on globes. According to proposition \ref{prop:criterimu_to_be_an_weak_equivalence}, $\psi$ is then a weakly invertible natural transformation.
We then have a zigzag of weakly invertible natural transformations: 
$$i\xrightarrow{\sim} i^l_{str}\xleftarrow{\sim}  \Sigma^l.$$
\end{proof}

\section{Application II: Others dualities}
\subsection{Globular objects}
\begin{definition}
We define several globular objects. 
\begin{itemize}[leftmargin=* ,parsep=0cm,itemsep=0cm,topsep=0cm]
\item $(\Gb_{co})_{n} :=  (\Sigma^\circ)^n(\Delta[0])$. 
\item $\Gb^{op}_{n} := (\Gb_{n})^{op}$.
\item $(\Gb_{\circ})_n := (\Gb_{co})_n^{op}$.
\end{itemize}
\end{definition}

\begin{lemma}
\label{lemma:inversion_on_globes}
For $\alpha\in\mathbb{Z}/2\mathbb{Z}$, we have a zigzag of acyclic cofibrations:
\[\begin{tikzcd}
	{\Sigma\Delta[0]} & {\Sigma^{\circ}\Delta[0]} \\
	{\Sigma\Sigma^{\circ}\Delta[0]} & {\Sigma^{\circ}\Sigma\Delta[0].}
	\arrow["{\Sigma d^{\alpha}}"', from=1-1, to=2-1]
	\arrow["{\Sigma d^{\alpha+1}}", from=1-2, to=2-2]
	\arrow[""{name=0, anchor=center, inner sep=0}, curve={height=-12pt}, draw=none, from=1-1, to=2-1]
	\arrow[""{name=1, anchor=center, inner sep=0}, curve={height=12pt}, draw=none, from=1-2, to=2-2]
	\arrow[squiggly, tail reversed, from=0, to=1]
\end{tikzcd}\]
\end{lemma}
\begin{proof}
First of all, $\Sigma^{\circ} \Delta[0]=\Sigma \Delta[0]$. The proposition \ref{prop:formula_for_the_coop_suspension}  implies that there is a zigzag of acyclic cofibrations between $\Sigma^{\circ}\Sigma\Delta[0]$ and $\Sigma\Sigma^{\circ}\Delta[0]$. A closer look on the proof shows that this zigzag of acyclic cofibrations inverses sources and targets.
\end{proof}

\begin{prop}
\label{prop:link_between_globes_and_coglobes}
For $\alpha\in\mathbb{Z}/2\mathbb{Z}$,
there are zigzags of acyclic cofibrations $\Gb_n\leftrightsquigarrow (\Gb_{co})_n$ and commutative diagrams:
\[\begin{tikzcd}
	{\Gb_n} & {\Gb_{n}^{co}} \\
	{\Gb_{n+1}} & {\Gb_{n+1}^{co}.}
	\arrow["{\delta^{\alpha+n}_{n+1}}", from=1-2, to=2-2]
	\arrow["{\delta^{\alpha}_{n+1}}"', from=1-1, to=2-1]
	\arrow[""{name=0, anchor=center, inner sep=0}, curve={height=-12pt}, draw=none, from=1-1, to=2-1]
	\arrow[""{name=1, anchor=center, inner sep=0}, curve={height=12pt}, draw=none, from=1-2, to=2-2]
	\arrow[squiggly, tail reversed, from=0, to=1]
\end{tikzcd}\]
\end{prop}
\begin{proof}
We proceed by induction on $n$. We then have zigzags:
\[\begin{tikzcd}
	{\Gb_{n+1}=\Sigma^{n-1}\Sigma\Delta[0]} & {\Sigma^{n-1}\Sigma^{\circ}\Delta[0]} & {\Sigma^{\circ}\Sigma^{n-1}\Delta[0] = \Sigma^{\circ}\Gb_{n}} & {\Gb_{n+1}^{co}} \\
	\\
	{\Gb_{n+2}=\Sigma^{n-1}\Sigma \Delta[1]} & {\Sigma^{n-1}\Sigma^{\circ} \Delta[1]} & {\Sigma^{\circ}\Sigma^{n-1}\Delta[1]=\Sigma^{\circ}\Gb_{n+1}} & {\Gb_{n+1}^{co}.}
	\arrow["{\delta_{n+1}^{\alpha}}"', from=1-1, to=3-1]
	\arrow["{\Sigma^{n-1}\delta_{2}^{\alpha+1}}"{description}, from=1-2, to=3-2]
	\arrow["{\Sigma^{\circ}\delta_{n+1}^{\alpha+1}}"{description}, from=1-3, to=3-3]
	\arrow["{\delta_{n+1}^{\alpha+n+1}}", from=1-4, to=3-4]
	\arrow[""{name=0, anchor=center, inner sep=0}, curve={height=-18pt}, draw=none, from=1-1, to=3-1]
	\arrow[""{name=1, anchor=center, inner sep=0}, curve={height=30pt}, draw=none, from=1-2, to=3-2]
	\arrow[""{name=2, anchor=center, inner sep=0}, curve={height=-30pt}, draw=none, from=1-2, to=3-2]
	\arrow[""{name=3, anchor=center, inner sep=0}, curve={height=30pt}, draw=none, from=1-3, to=3-3]
	\arrow[""{name=4, anchor=center, inner sep=0}, curve={height=-30pt}, draw=none, from=1-3, to=3-3]
	\arrow[""{name=5, anchor=center, inner sep=0}, curve={height=18pt}, draw=none, from=1-4, to=3-4]
	\arrow[squiggly, tail reversed, from=0, to=1]
	\arrow[squiggly, tail reversed, from=2, to=3]
	\arrow[squiggly, tail reversed, from=4, to=5]
\end{tikzcd}\]
The left zigzag comes from proposition \ref{prop:formula_for_the_coop_suspension}, the center one from  \ref{lemma:inversion_on_globes}, and the right one from induction hypothesis.
\end{proof}

We can show similarly:

\begin{prop}
\label{prop:link_between_globes_and_opglobes}
For $\alpha\in\mathbb{Z}/2\mathbb{Z}$,
there are zigzags of acyclic cofibrations $\Gb_n\leftrightsquigarrow \Gb_n^{op}$ and commutative diagrams:
\[\begin{tikzcd}
	{\Gb_n} & {\Gb_{n}^{op}} \\
	{\Gb_{n+1}} & {\Gb_{n+1}^{op}.}
	\arrow["{\delta^{\alpha+n+1}_{n+1}}", from=1-2, to=2-2]
	\arrow["{\delta^{\alpha}_{n+1}}"', from=1-1, to=2-1]
	\arrow[""{name=0, anchor=center, inner sep=0}, curve={height=-12pt}, draw=none, from=1-1, to=2-1]
	\arrow[""{name=1, anchor=center, inner sep=0}, curve={height=12pt}, draw=none, from=1-2, to=2-2]
	\arrow[squiggly, tail reversed, from=0, to=1]
\end{tikzcd}\]
\end{prop}
\begin{prop}
\label{prop:link_between_globes_and_coopglobes}
For $\alpha\in\mathbb{Z}/2\mathbb{Z}$,
there are zigzags of acyclic cofibrations $\Gb_n\leftrightsquigarrow \Gb_n^{\circ}$ and commutative diagrams:
\[\begin{tikzcd}
	{\Gb_n} & {\Gb_{n}^{\circ}} \\
	{\Gb_{n+1}} & {\Gb_{n+1}^{\circ}.}
	\arrow["{\delta^{\alpha+1}_{n+1}}", from=1-2, to=2-2]
	\arrow["{\delta^{\alpha}_{n+1}}"', from=1-1, to=2-1]
	\arrow[""{name=0, anchor=center, inner sep=0}, curve={height=-12pt}, draw=none, from=1-1, to=2-1]
	\arrow[""{name=1, anchor=center, inner sep=0}, curve={height=12pt}, draw=none, from=1-2, to=2-2]
	\arrow[squiggly, tail reversed, from=0, to=1]
\end{tikzcd}\]
\end{prop}

\subsection{The even duality}
\label{section:The even duality}
\begin{definition}
The category of endomorphisms of marked simplicial set has a monoidal structure given by the composition. The endomorphism $\uvar\costar $ admits a monoid structure, where the multiplication is the natural transformation:
$(X\costar \Delta[0])\costar \Delta[0]\to X \costar \Delta[0]$, induces by the pairing: 
$$
\begin{array}{rcl}
I\otimes I\otimes X&\to& I\otimes X\\
(i,j,x)&\mapsto& (i\vee j, x).
\end{array}$$

This defines a cosimplicial object in $End(\mSset)$, which evaluated on $\emptyset$, provides a cosimplicial object in $\mSset$: 
$$\begin{array}{rrcl}
(\uvar)_{co}:&\Delta &\to & \mSset\\
&n&\mapsto& (((\Delta[0]\costar  \Delta[0])...)\costar \Delta[0].
\end{array}$$
We define $(\Delta[n]_t)_{co} := \tau_n(\Delta[n]_{co})$. This data extends in a colimit preserving functor:
$$(\uvar)_{co}:\stratSset \to \mSset.$$ 
\end{definition}

\begin{definition}
Let $p,q,r,n$ be four integers such that $p+q+r = n$. We define  three simplicial operators:
$$\begin{array}{rrcl}
\amalg^1_{p,q,r} &[p]&\to& [n]\\
&k&\mapsto & k
\end{array}
~~~~
\begin{array}{rrcl}
\amalg^2_{p,q,r} &[q]&\to& [n]\\
&k&\mapsto &k+p
\end{array}
~~~~
\begin{array}{rrcl}
\amalg^1_{p,q,r} &[r]&\to& [n]\\
&k&\mapsto &k+p+q
\end{array}
$$
\end{definition}

\begin{lemma}
Let $(X,N_0)$ and $(Y,N_1)$ be two stratified simplicial sets and $n,m$ two integers such 
\begin{enumerate}
\item $(X,N_0)$ admits a unique non thin $n$-simplex $x$ and simplices of dimension $>n$ are thin.
\item $(Y,N_1)$ admits a unique non thin $m$-simplex $y$ and simplices of dimension  $>m$ are thin.
\end{enumerate}

Let $M$ be the marking of $(X,N_0)\otimes (\Delta[2]_t)_{co})\otimes (Y,N_1)$, we define two other sets of cell of $X\times (\Delta[2]_{co})\times Y$:
$$
\begin{array}{rcl}
M_0 &:=& \{(x',v,y'),~ \amalg^1_{n,2,m}(x')=x, ~\amalg^2_{n,2,m}(v)\in (\Lambda^1[2])_{co}~, \amalg^2_{n,2,m}(y')=y\}\\
M_1 &:=& \{(x',v,y'),~ \amalg^1_{n,2,m}(x')=x, ~\amalg^2_{n,2,m}(v)\in (\Delta[2]_t)_{co}~, \amalg^2_{n,2,m}(y')=y\}\\
\end{array}$$

Then $$\overline{M\cup M_0} = \overline{M\cup M_1}.$$
\end{lemma} 
\begin{proof}
Let $w:=(x',v,y')$ be a simplex of $M_1\diagdown M_0$. Suppose that  $\Pi^2_{n,2,m}(v)=[00,11]$. We define  the  $(n+m+2)$-simplex $\bar{v}$:
$$\begin{array}{rcll}
\bar{v}_k & =& v_k &\mbox{if $k< n+1$}\\
 & = &01 &\mbox{if $k= n+1$}\\
& = &v_{k-1} &\mbox{if $n+1<k$}\\
\end{array}
$$
We set $\bar{w}:=(s^{n}x,\bar{v},s^{n+1}y)$. We remark that $d^{n+1}\bar{w}=w$. The simplex $\bar{w}$ is in $M$ and for all $d:\Delta[k]\to\Delta[n+m+2]$ reaching $n,n+1$ and $n+2$, so is $d^*\bar{w}$. Furthermore, $d^{n-1}\bar{w}$ and $d^{n+1}\bar{w}$ are in $M_0$. The simplex $d^{n+1}\bar{w}=w$ is then in $\overline{M\cup M_0}$. 

Suppose now that  $\Pi^2_{n,2,m}(v)=[00,10]$. Every simplex of shape $(s^pa,v,s^pb)$ with $v_{[p,p+1]}=[10,11]$ are in $M$. We can then apply lemma \ref{lemma:oslash_saturation_two way_case_extremum} that implies that  $w$ is in $\overline{M\cup M_0}$ if and only if $(x',v[10/11],y')$ is, which is the case. 
\end{proof}

\begin{lemma}
\label{lemma:compatibility_of_co_with_saturation1}
In $\stratSset$,
the cofibration
$$(\Delta^k[n]')_{co}\to (\Delta^k[n]'')_{co}$$
is an acyclic cofibration.
\end{lemma}
\begin{proof}
Taking the same notation than the previous lemma, we 
 have a pushout square: 
\[\begin{tikzcd}
	{(\Delta[n-k]_{co}\times(\Delta[2]_t)_{co}\times\Delta[k-2]_{co},M\cup M_0)} & {(\Delta^k[n]')_{co}} \\
	{(\Delta[n-k]_{co}\times(\Delta[2]_t)_{co}\times\Delta[k-2]_{co},\overline{M\cup M_0})} & {(\Delta^k[n]'')_{co}}
	\arrow[""{name=0, anchor=center, inner sep=0}, from=1-1, to=1-2]
	\arrow[from=1-1, to=2-1]
	\arrow[from=1-2, to=2-2]
	\arrow[from=2-1, to=2-2]
	\arrow["\lrcorner"{anchor=center, pos=0.125, rotate=180}, draw=none, from=2-2, to=0]
\end{tikzcd}\]
The previous lemma implies that the left hand morphism is an acyclic cofibration. 
\end{proof}

\begin{lemma}
\label{lemma:compatibility_of_co_with_saturation2}
In $\stratSset$,
The cofibration: 
$$(\Delta[l]\star \Delta[3]^{eq}\star \Delta[q])_{co}\to (\Delta[l]\star \Delta[3]^{\sharp}\star \Delta[q])_{co}$$
is an acyclic cofibration.
\end{lemma}
\begin{proof}
First, we remark that we have an equality: 
$$(\Delta[3]^{eq})_{co} = (\Delta[0]\diamond (\Delta[0]\diamond \Delta[1])),M)$$
where $M$ is the saturation that includes all simplices of dimension superior to $2$, together with $\emptyset\diamond(\Delta[0]\diamond \{1\})$ and $\Delta[0]\diamond(\emptyset\diamond \{0\})$.

Eventually, the comparison morphism $\gamma: \uvar\diamond\uvar\to \uvar\star\uvar$ from \ref{prop:equivalence between diamond and join product}, induces a diagram:
\[\begin{tikzcd}
	{ (\Delta[3]^{eq})_{co}} & { \Delta[3]^{eq}} \\
	{( \Delta[3]^{\sharp})_{co}} & { \Delta[3]^{\sharp}.}
	\arrow["\sim", from=1-2, to=2-2]
	\arrow["\sim", from=1-1, to=1-2]
	\arrow["\sim", from=2-1, to=2-2]
	\arrow[from=1-1, to=2-1]
	\arrow["\gamma"', draw=none, from=1-1, to=1-2]
	\arrow["\gamma"', draw=none, from=2-1, to=2-2]
\end{tikzcd}\]
where horizontal morphisms are weak equivalences. This
 implies that $\Delta[3]^{eq}\to  (\Delta[3]^{\sharp})_{co}$ is an acyclic cofibration. For the general case, we have a pushout diagram:
\[\begin{tikzcd}
	{ \Delta[q]_{co}\otimes (\Delta[3]^{eq})_{co}\otimes\Delta[l]_{co}} & {(\Delta[l]\star \Delta[3]^{eq}\star \Delta[q])_{co}} \\
	{ \Delta[q]_{co}\otimes (\Delta[3]^{\sharp})_{co}\otimes\Delta[l]_{co}} & {(\Delta[l]\star \Delta[3]^{\sharp}\star \Delta[q])_{co}}
	\arrow[from=1-1, to=2-1]
	\arrow[""{name=0, anchor=center, inner sep=0}, from=1-1, to=1-2]
	\arrow[from=1-2, to=2-2]
	\arrow[from=2-1, to=2-2]
	\arrow["\lrcorner"{anchor=center, pos=0.125, rotate=180}, draw=none, from=2-2, to=0]
\end{tikzcd}\]
where the left hand morphism is an acyclic cofibration.
\end{proof}

\begin{construction}
Lemmas \ref{lemma:compatibility_of_co_with_saturation1} and \ref{lemma:compatibility_of_co_with_saturation2} imply that the composite functor:
$$\stratSset\xrightarrow{(\uvar)_{co}}\stratSset\to \mSset$$
sends complicial thinness extensions and saturation extensions to identities. This implies that the even duality lifts to marked simplicial sets, and we yet have endofunctor:
$$(\uvar)_{co}: \mSset\to \mSset.$$ 
\end{construction}

\begin{prop}
The endofunctor $(\uvar)_{co}$ is a left Quillen functor. 
\end{prop}
\begin{proof}
We show by induction on $n$ that for all $k\leq n$, the morphism 
$$(\Lambda^k[n])_{co}\to (\Delta^k[n])_{co}$$
is an acyclic cofibration. 
Suppose the result true for $n$, and suppose $k\leq n$. We have a diagram:
\[\begin{tikzcd}
	{(\Lambda^k[n+1])_{co}} & {(\Lambda^k[n])_{co}\costar \Delta[0]} \\
	{ (\Delta^k[n+1])_{co}} & {(\Delta^k[n])_{co}\costar \Delta[0]}
	\arrow[from=1-2, to=2-2]
	\arrow[from=1-1, to=2-1]
	\arrow[Rightarrow, no head, from=1-1, to=1-2]
	\arrow[Rightarrow, no head, from=2-1, to=2-2]
\end{tikzcd}\]
The right hand morphism is an acyclic cofibration. We still have to deal the case $k=n+1$. Using proposition \ref{prop:almost_associativity_of_star_co}, we have a zigzag of acyclic cofibrations
\[\begin{tikzcd}
	{(\Lambda^{n+1}[n+1])_{co}} & {\Delta[0]\costar (\Lambda^n[n])_{co}} \\
	{ (\Delta^{n+1}[n+1])_{co}} & {\Delta[0]\costar (\Delta^n[n])_{co}}
	\arrow[from=1-2, to=2-2]
	\arrow[from=1-1, to=2-1]
	\arrow[""{name=0, anchor=center, inner sep=0}, curve={height=12pt}, draw=none, from=1-2, to=2-2]
	\arrow[""{name=1, anchor=center, inner sep=0}, curve={height=-12pt}, draw=none, from=1-1, to=2-1]
	\arrow[squiggly, tail reversed, from=1, to=0]
\end{tikzcd}\]
where the right hand morphism is an acyclic cofibration.
\end{proof}

\begin{construction}
\label{cons:of_the_comparaison_between_sigma_and_sigma_co}
Let $\Sigma^\star X$ be the \textit{joint-suspension}, define as the following pushout:
\[\begin{tikzcd}
	X & {X\diamond\Delta[0]} & {X\star\Delta[0]} \\
	{\Delta[0]} & {\Sigma X} & {\Sigma^{\star} X.}
	\arrow[from=1-1, to=2-1]
	\arrow[from=1-3, to=2-3]
	\arrow[from=1-2, to=2-2]
	\arrow[""{name=0, anchor=center, inner sep=0}, "\sim", from=1-2, to=1-3]
	\arrow[""{name=1, anchor=center, inner sep=0}, from=1-1, to=1-2]
	\arrow[from=2-1, to=2-2]
	\arrow["\sim"', from=2-2, to=2-3]
	\arrow["\lrcorner"{anchor=center, pos=0.125, rotate=180}, draw=none, from=2-2, to=1]
	\arrow["\lrcorner"{anchor=center, pos=0.125, rotate=180}, draw=none, from=2-3, to=0]
\end{tikzcd}\]
We then have a weakly invertible natural transformation: 
$$\Sigma \to \Sigma^\star$$
Furthermore, $\Sigma^{\circ}(X^{co})$ fits in the following pushout:
\[\begin{tikzcd}
	{X_{co}} & {X_{co}\costar \Delta[0]} \\
	{\Delta[0]} & {\Sigma^{\circ} X^{co}.}
	\arrow[from=1-1, to=2-1]
	\arrow[from=1-2, to=2-2]
	\arrow[""{name=0, anchor=center, inner sep=0}, from=1-1, to=1-2]
	\arrow[from=2-1, to=2-2]
	\arrow["\lrcorner"{anchor=center, pos=0.125, rotate=180}, draw=none, from=2-2, to=0]
\end{tikzcd}\]
The functor $co$ commutes with colimits, we then have a
weakly invertible natural transformation: 
$$(\Sigma\uvar)_{co}\to (\Sigma^\star\uvar)_{co}= \Sigma^{\circ}((\uvar)_{co}).$$
In particular we have an invertible natural transformation: 
$$(\Gb_{\uvar})_{co}\to (\Gb_{co})_{\uvar}.$$
\end{construction}

\begin{theorem}
\label{theo:co_is_an_duality}
There is a zigzag of weakly invertible natural transformations:
$$((\uvar)_{co})_{co}\leftrightsquigarrow id.$$
\end{theorem}
\begin{proof}
In construction \ref{cons:of_the_comparaison_between_sigma_and_sigma_co} we remark that we have a zigzag of weakly invertible natural transformations:
$$(\Gb_{\uvar})_{co}\to (\Gb_{co})_{\uvar}.$$
Using proposition \ref{prop:link_between_globes_and_coglobes} twice, this implies that there is a zigzag of weakly invertible natural transformations:
$$((\Gb_{\uvar})_{co})_{co}\leftrightsquigarrow \Gb_{\uvar}.$$
The theorem \ref{theo:criterium_to_be_linked_to_identity} then implies the result.
\end{proof}

\begin{cor}
\label{cor:Sigma_and_duality}
There are  zigzags of weakly invertible natural transformations:
$$(\Sigma \uvar)_{co}\leftrightsquigarrow \Sigma (\uvar)^{op},~~~~(\Sigma \uvar)^{op}\leftrightsquigarrow \Sigma (\uvar)_{co},~~~(\Sigma \uvar)^{\circ}\leftrightsquigarrow \Sigma (\uvar)_{\circ}.$$
\end{cor}
\begin{proof}
Propositions \ref{prop:link_between_globes_and_coglobes} and \ref{prop:link_between_globes_and_opglobes} imply that we
 have zigzags of weakly invertible natural transformations:
$$(\Sigma (\Gb_{\uvar})^{op})_{co}\leftrightsquigarrow \Gb_{\uvar+1} \leftrightsquigarrow(\Sigma (\Gb_{\uvar})_{co})_{op}.$$
According to theorem \ref{theo:criterium_to_be_linked_to_identity} we then have  zigzags of weakly invertible natural transformations:
$$(\Sigma (\uvar)^{op})_{co}\leftrightsquigarrow \Sigma \leftrightsquigarrow(\Sigma (\uvar)_{co})_{op}.$$
Using the fact that $((\uvar)_{co})_{co}\leftrightsquigarrow id$,  and $((\uvar)^{op})^{op}= id$, this proves the two first formula. The last one is a direct consequence.
\end{proof}

\subsection{The full duality}

\begin{prop}
\label{prop:co_and_op_commutes}
There are a zigzag of equivalences: 
$$(X_{co})^{op} \leftrightsquigarrow (X^{op})_{co}$$
natural in $X$.
\end{prop}
\begin{proof}
Proposition \ref{prop:link_between_globes_and_coglobes} and \ref{prop:link_between_globes_and_opglobes} implies that we
 have zigzag of weakly invertible natural transformations:
$$((((\Gb_{\uvar})_{co})^{op})_{co})^{op}\leftrightsquigarrow \Gb_{\uvar}.$$
According to theorem \ref{theo:criterium_to_be_linked_to_identity} we then have  zigzag of weakly invertible natural transformations:
$$((((\uvar)_{co})^{op})_{co})^{op}\leftrightsquigarrow id.$$
Using the fact that $((\uvar)_{co})_{co}\leftrightsquigarrow id$,  and $((\uvar)^{op})^{op}= id$, we conclude the proof.
\end{proof}

\begin{definition}
We define the \textit{full duality} to be the composite functor 
$$(\uvar)_{\circ}: X\mapsto (X_{co})^{op}.$$
This functor is then a left Quillen functor and according to \ref{theo:co_is_an_duality} and \ref{prop:co_and_op_commutes}, satisfies
$$((\uvar)_{\circ})_{\circ}\leftrightsquigarrow id.$$
\end{definition}

\subsection{Action of dualities on homotopy categories}

We denote  $(\uvar)^{co}$  the right adjoint of $(\uvar)_{co}$ and  $(\uvar)^{\circ}$  the right adjoint of $(\uvar)_{\circ}$. If $s\sqcup t$ is an element of $[\partial \Gb_n,X]$ and  $s,t$ two parallel cells such that $s\cup t$ represents $s\sqcup t$, we note $t\sqcup s$ the equivalence class corresponding to $t\cup s$. Eventually, we also denote $s\sqcup t$ the equivalence class corresponding to: 
\[\begin{tikzcd}
	{\partial\Gb_n} & {(\partial\Gb_n)^{co}} & {X^{co},} & {\partial\Gb_n} & {(\partial\Gb_n)^{op}} & {X^{op},}
	\arrow["{(s\cup t)^{co}}", from=1-2, to=1-3]
	\arrow[squiggly, tail reversed, from=1-1, to=1-2]
	\arrow["{(s\cup t)^{op}}", from=1-5, to=1-6]
	\arrow[squiggly, tail reversed, from=1-4, to=1-5]
\end{tikzcd}\]
\[\begin{tikzcd}
	{\partial\Gb_n} & {(\partial\Gb_n)^{\circ}} & {X^{\circ}.}
	\arrow["{(s\amalg t)^{\circ}}", from=1-2, to=1-3]
	\arrow[squiggly, tail reversed, from=1-1, to=1-2]
\end{tikzcd}\]
\begin{prop}
Let $X$ be a $\infty$-category.
$$\pi_0(X) = \pi_0(X^{co}) =\pi_0(X^{op})=\pi_0(X^{\circ})$$
Let $s\sqcup t \in [\partial \Gb_n,X]$. 
If $n$ is odd, 
$$\pi_n(s\sqcup t,X^{co}) := \pi_n(t\sqcup s,X),~~~~ \pi_n(s\sqcup t,X^{op}) := \pi_n(s\sqcup t,X)^{op},$$
if $n$ is even:
$$\pi_n(s\sqcup t,X^{co}) := \pi_n(s\sqcup t,X)^{op},~~~~ \pi_n(s\sqcup t,X^{op}) := \pi_n(t\sqcup s,X),$$
for all $n$,
$$\pi_n(s\sqcup t,X^{\circ}) := \pi_n(t\sqcup s,X)^{op}.$$
\end{prop}
\begin{proof}
This is a consequence of the propositions \ref{prop:link_between_globes_and_coglobes}, \ref{prop:link_between_globes_and_opglobes} and \ref{prop:link_between_globes_and_coopglobes}.	
\end{proof}

\section{Application III: Grothendieck fibrations of $\infty$-categories.}

\subsection{Model structure on bimarked simplicial sets}
\begin{definition}
A \textit{bistratified simplicial set} is a triple $(X,tX,cX)$ where  $X$ is a simplicial set and  $tX := \cup_{n>0}tX_n$, $cX := \cup_{n>0}cX_n$ are graded sets such that for all $n\geq 1$, $tX_n$ and $cX_n$ are  subsets of $X_n$ that include all degenerate simplices, and $tX\subset cX$. A simplex in $tX$ is called \textit{thin}, and a simplex in $cX$ is called \textit{cartesian}. 
A \textit{bistratified morphism} $f:(X,tX,cX)\to (Y,tY,cX)$ is the data of a morphism on the underlying simplicial sets  such that $f(tX)\subset tY$ and $f(cX)\subset cY$.
The category of bistratified simplicial sets is denoted $\bstratSset$.  
\end{definition}

If $(X,tX)$ is a stratified simplicial set, $(X,tX,tX)$ is a bistratified simplicial set, and if $(X,tX,cX)$ is a bistratified simplicial set, it's underlying stratified simplicial set is $(X,tX)$.  This assignation shows that stratified simplicial sets can be identified as a full subcategory of bistratified simplicial sets.

\begin{definition}
Let $n$ and $0<k<n$ be two integers. We define several bistratified structures on $\Delta[n]$:
\begin{enumerate}
\item $\Delta[n]_c$. The top $n$-simplices is cartesian.
\item $\Delta^{0}[n]^\centerdot$. All simplices that include $\{0,1\}$ are cartesian.
\item $\Delta^{n}[n]^\centerdot$. All simplices that include $\{n-1,n\}$ are cartesian.
\item $\Delta^k[n]^{\ast}$. All simplices that include $\{k-1,k,k+1\}\cap[n]$ are thin.  The $(k-1)$-face and the $(k+1)$ face are cartesian.
\item $\Delta^k[n]^{\ast,\ast}$. All simplices that include $\{k-1,k,k+1\}\cap[n]$ are thin.  The $(k-1)$-face, the $k$-face and the $(k+1)$ face are cartesian.
\end{enumerate}
\end{definition}

\begin{definition}
We define several classes of cofibrations.
\begin{enumerate}
\item The \textit{cartesian terminal horn inclusions}:
$$\Lambda^{0}[n]^\centerdot\to^e \Delta^{0}[n]^\centerdot.$$
\item  The \textit{cartesian co-initial horn inclusions}:
$$\Lambda^{n}[n]^\centerdot\to^e \Delta^{n}[n]^\centerdot.$$
\item  The \textit{cartesian thinness extensions}:
$$\Delta^k[n]^{\ast}\to \Delta^k[n]^{\ast,\ast}.$$
\item  The \textit{cartesian trivializations}:
$$\Delta[n]_c\to \Delta[n]_t.$$
\end{enumerate}
The set of cartesian trivializations is denoted $T$.
\end{definition}

\begin{definition}
A \textit{bimarked simplicial set} is a bistratified simplicial set having the right lifting property against complicial thinness extensions, saturation extensions and cartesian thinness extensions.

The category of bimarked simplicial sets is denoted $\mSset$.
\end{definition}

We can extend $\star$ and $\otimes$ to bimarked simplicial sets: 
$$(X,tX,cX)\star(Y,tY,cY) := (X\star Y, \overline{tX\star tY}, \overline{cX\star cY}),$$ 
$$(X,tX,cX)\otimes(Y,tY,cY) := (X\otimes Y, \overline{tX\otimes tY}, \overline{cX\otimes cY}).$$
Functors $\Sigma, \costar , (\uvar)_{co}, (\uvar)_{\circ}, \fwedge$ being defined using these operations, they can be promoted in functors on bimarked simplicial sets. 
The interval $I$ is again $\Delta[1]_t$.

\begin{definition}
A \textit{marked $\infty$-category} is a bimarked simplicial set  $(X,tX,cX)$ such that $(X,tX)$ is an $\infty$-category.
\end{definition}

\begin{theorem}
There exists a $(\Lambda,I)$-local model structure on $\bmSset$, called the \emph{blind model structure}, where a morphism $f:(X,tX,cX)\to (Y,tY,cY)$ is a cofibration (resp. a fibration) if and only if the underlying morphism $(X,tX)\to (Y,tY)$ is a cofibration (resp. a fibration). A morphism is a weak equivalence if and only if the underlying morphism $(X,tX)\to (Y,tY)$ is and  if $\overline{f(cX)} = cY$ and $f^{-1}(cY)=cX$. Fibrant objects are marked $\infty$-categories.
\end{theorem}
\begin{proof}
Theorem $\ref{theo:local_model_structure_on_saturated_bistratified_presheave}$ provides a $(\Gamma(\Lambda),I)$-local model structure satisfying all the desired conditions. We can use proposition \ref{prop:martina} to show that this model structure is $(\Lambda,I)$-local.
\end{proof}
Functor $\Sigma, \costar , (\uvar)_{co}, (\uvar)_{\circ}, \fwedge$ are left Quillen functor for this model structure.

\subsection{Cartesian fibrations}

\begin{definition}
For $\alpha\in\{0,1\}$, we set $\partial^c_\alpha:\{\alpha\}\to\Delta[1]_t$.
We define several class of cofibration:
\begin{enumerate}
\item $J_{R}$ is the set of morphisms of shape $i~\hat{\otimes}~\partial_0^c$ or $j~\hat{\times}~ \partial_0^c$ ;
\item $J_{L}$ is the set of morphisms of shape $i~\hat{\otimes}~\partial_1^c$ or  $j~\hat{\times}~\partial_1^c$ ;
\item $J_{coR}$ is the set of morphisms of shape $\partial_0^c~\hat{\otimes}~i$ or $\partial_0^c~\hat{\times}~j$ ;
\item $J_{coL}$ is the set of morphisms of shape $\partial_1^c~\hat{\otimes}~i$ or $\partial_1^c~\hat{\times}~j$ ;
\end{enumerate}
where $i$ is any generating cofibration of marked simplicial set, and $j$ any cartesian trivialization.
\end{definition}

\begin{definition}
Let $f$ be a blind fibration.
\begin{enumerate}
\item $f$ is a \textit{right fibration} if it has the right lifting property against morphisms of $J_{R}\cup T$.
\item $f$ is a \textit{left fibration} if it has the right lifting property against morphisms of $J_{L}\cup T$.
\item $f$ is a \textit{co-right fibration} if it has the right lifting property against morphisms of $J_{coR}\cup T$.
\item $f$ is a \textit{co-left fibration} if it has the right lifting property against morphisms of $J_{coL}\cup T$.
\end{enumerate}
\end{definition}

\begin{theorem}
\label{theo:the_foor_model_structure}
Let $C$ be a marked $\infty$-category. There exists several model structures on $\bmSset_{/C}$, which are localizations of the induced blind model structure:
\begin{enumerate}
\item The \emph{right model structure}, which is $(J_{R}\cup T,I)$-local,   whose fibrant objects and fibrations between fibrant objects are right fibrations and weak equivalences between fibrant objects are  blind weak equivalences. 
\item The \emph{left model structure}, which is $(J_{L}\cup T,I)$-local,   whose fibrant objects and fibrations between fibrant objects are left fibrations and weak equivalences between fibrant objects are  blind weak equivalences. 
\item The \emph{co-right model structure}, which is $(J_{cR}\cup T,I)$-local,   whose fibrant objects and fibrations between fibrant objects are co-right fibrations and weak equivalences between fibrant objects are  blind weak equivalences. 
\item The \emph{co-left model structure}, which is $(J_{cL}\cup T,I)$-local,  whose fibrant objects and fibrations between fibrant objects are co-left fibrations and weak equivalences between fibrant objects are  blind weak equivalences. 
\end{enumerate}
\end{theorem}
\begin{proof}
We will prove the assertion for the right model structure, the other ones are proved similarly. We denote $(J_{R}\cup T)_{/C}$ the set of morphisms of $\bmSset_{/C}$ of shape
\[\begin{tikzcd}
	K \\
	L & C.
	\arrow["{ J_R\cup T  \ni}"', from=1-1, to=2-1]
	\arrow[from=2-1, to=2-2]
	\arrow[from=1-1, to=2-2]
\end{tikzcd}\]

Theorem \ref{theo:local_model_structure_on_saturated_bistratified_presheave} implies the existence of a $(\Gamma((J_{R}\cup T)_{/C}),I)$-local model structure verifying all the desired properties. We then have to check that for all couples of morphisms $i$, $j$, $k$ where $i$ is a generating cofibration of bistratified simplicial sets, $j$ a morphism in $(J_{R}\cup T)$, and $k$ a cartesian trivialization, the induced morphisms:
$$i~ \hat{\times}~ (j~\hat{\otimes}~\partial^c_0)~~~~\mbox{and}~~~~ i~ \hat{\times}~ k~\hat{\times}~\partial^c_0$$
are $(J_R\cup T)$-anodyne extensions. For the first one, we have a pushout:
\[\begin{tikzcd}
	A & {A'} \\
	B & {B''} \\
	&& {B'.}
	\arrow["{i~ \hat{\times}~ (j~\hat{\otimes}~\partial^c_0)}", curve={height=-12pt}, from=1-2, to=3-3]
	\arrow["{i~ \hat{\otimes}~ j~\hat{\otimes}~\partial^c_0}"', from=1-1, to=2-1]
	\arrow[curve={height=12pt}, from=2-1, to=3-3]
	\arrow[from=1-2, to=2-2]
	\arrow[from=1-1, to=1-2]
	\arrow[from=2-1, to=2-2]
	\arrow["\lrcorner"{anchor=center, pos=0.125, rotate=180}, draw=none, from=2-2, to=1-1]
	\arrow[from=2-2, to=3-3]
\end{tikzcd}\]
The morphism $A\to B$ is a $J_R$-anodyne extension, so is $A'\to B''$. The underlying simplicial set of $B''$ and $B'$ are the same, and they also share the same set of cartesian simplices. The  morphism $B''\to B'$ is  then a $T$-anodyne extension. For the second one, two cases are possible. Either $i$ is $\emptyset\to\Delta[0]$, but $i~ \hat{\times}~ k~\hat{\times}~\partial^c_0$ is then equal to $k~\hat{\times}~\partial^c_0$ which is in $J_R$. If $i$ is any other generating cofibration, the maps  $i~ \hat{\times}~\partial^c_0$ is a monomorphism surjective on $0$-simplices. The proposition \ref{prop:behavious_of_times_with_cofibration} then implies that  $i~ \hat{\times}~ k~\hat{\times}~\partial^c_0$ is an equality.
\end{proof}

\begin{prop}
\label{prop:suspension_and_right_model_structure}
The $\circ$-suspension is a Quillen functor for  the right and the left model structure. The  suspension is a Quillen functor for the co-right and the co-left model structure.
\end{prop}
\begin{proof}
First, remark that if $i:K\to L$ is a co-right or a co-left acyclic cofibration, so is $i\hat{\otimes}~(\partial\Delta[1]\to \Delta[1]))$. Furthermore, we have a colimit diagram
\[\begin{tikzcd}
	{L\otimes\partial\Delta[1]} & {K\otimes \Delta[1]\cup L\otimes\partial\Delta[1]} & {L\otimes \Delta[1]} \\
	{\Delta[1]} & {\Sigma K} & {\Sigma L.}
	\arrow[""{name=0, anchor=center, inner sep=0}, from=1-1, to=1-2]
	\arrow[from=1-1, to=2-1]
	\arrow[from=1-2, to=2-2]
	\arrow[""{name=1, anchor=center, inner sep=0}, "\sim", from=1-2, to=1-3]
	\arrow[from=2-1, to=2-2]
	\arrow["\sim"', from=2-2, to=2-3]
	\arrow[from=1-3, to=2-3]
	\arrow["\lrcorner"{anchor=center, pos=0.125, rotate=180}, draw=none, from=2-3, to=1]
	\arrow["\lrcorner"{anchor=center, pos=0.125, rotate=180}, draw=none, from=2-2, to=0]
\end{tikzcd}\]
that shows that $\Sigma i$ is an acyclic cofibration. We proceed similarly for the $\circ$-suspension.
\end{proof}

Let $f:A\to B$ be a morphism. This induces an adjunction: 
\[\begin{tikzcd}
	{\bmSset_{/A}} && {\bmSset_{/B}.}
	\arrow[""{name=0, anchor=center, inner sep=0}, "{f_!}"', curve={height=12pt}, from=1-1, to=1-3]
	\arrow[""{name=1, anchor=center, inner sep=0}, "{f^*}"', curve={height=12pt}, from=1-3, to=1-1]
	\arrow["\dashv"{anchor=center, rotate=-90}, draw=none, from=1, to=0]
\end{tikzcd}\]

We can easily see that this is a Quillen adjunction for all model structure of theorem \ref{theo:the_foor_model_structure}.

\begin{prop}
Let $f:A\to B$ be a blind weak equivalence between marked $\infty$-categories. The induced adjoint pair $(f_!,f^*)$ is a Quillen equivalence for all model structures of  \ref{theo:the_foor_model_structure}.
\end{prop}
\begin{proof}
By the Ken-brown lemma, apply to the blind model structure, it is enough to show this when  $f$ a trivial fibration. The functor $f_!$ then preserves fibrant objects. To show that this is a Quillen equivalence, it is enough to show that for all pair a fibrant  objects $g:X\to A$, $h:Y\to B$, the two induced morphisms:
$$g \to f^*f_! g ~~~~~~ f_!f^* h \to h.$$
are blind weak equivalences. 
For the first one, we consider the following diagram:
\[\begin{tikzcd}
	X & {f^*f_!X} & X \\
	& A & B
	\arrow["g"', from=1-1, to=2-2]
	\arrow[from=1-2, to=2-2]
	\arrow["{ \sim}"', from=1-2, to=1-3]
	\arrow["\sim", from=2-2, to=2-3]
	\arrow["{f_!g}", from=1-3, to=2-3]
	\arrow[dotted, from=1-1, to=1-2]
	\arrow["id", curve={height=-18pt}, from=1-1, to=1-3]
	\arrow["\lrcorner"{anchor=center, pos=0.125}, draw=none, from=1-2, to=2-3]
	\arrow["f"', draw=none, from=2-2, to=2-3]
\end{tikzcd}\]
by diagram chasing it shows that $g \to f^*f_! g$ is a blind weak equivalence. For the second one, we consider
\[\begin{tikzcd}
	{f^*Y} & Y \\
	A & B
	\arrow["h", from=1-2, to=2-2]
	\arrow["f"', from=2-1, to=2-2]
	\arrow["{f^*h}"', from=1-1, to=2-1]
	\arrow["\sim", from=1-1, to=1-2]
	\arrow["\lrcorner"{anchor=center, pos=0.125}, draw=none, from=1-1, to=2-2]
	\arrow["\sim", draw=none, from=2-1, to=2-2]
\end{tikzcd}\]
The morphism $f_!f^* h \to h$ is then a trivial fibration, and so a blind weak equivalence.
\end{proof}

\subsection{Naive cartesian fibrations}

\begin{definition}
Let $C$ be a $\infty$-category, and $p:C\to D$ an blind fibration. A cartesian $1$-simplex $c$ is \textit{right-cancellable}  over $p(c)$ if all diagrams:
\[\begin{tikzcd}
	{(K\fwedge\Delta[1]_c)\cup_{\triangledown} L} & C \\
	{L\fwedge\Delta[1]_c } & D
	\arrow[from=1-1, to=1-2]
	\arrow[from=1-1, to=2-1]
	\arrow["p", from=1-2, to=2-2]
	\arrow[from=2-1, to=2-2]
	\arrow[dotted, from=2-1, to=1-2]
\end{tikzcd}\]
admit liftings, where $K\to L$ is either $\partial\Gb_n\to \Gb_n$, $\Gb_n\to (\Gb_n)_c$ or $(\Gb_n)_c\to (\Gb_n)_t$.
 It is \textit{left cancellable} over $p(c)$ if all diagrams:
\[\begin{tikzcd}
	{L\cup_{\triangledown}(\Delta[1]_c\fwedge K)} & C \\
	{\Delta[1]_c\fwedge L} & D
	\arrow[from=1-1, to=1-2]
	\arrow[from=1-1, to=2-1]
	\arrow["p", from=1-2, to=2-2]
	\arrow[from=2-1, to=2-2]
	\arrow[dotted, from=2-1, to=1-2]
\end{tikzcd}\]
admits  liftings, where $K\to L$ is either $\partial\Gb_n\to \Gb_n$, $\Gb_n\to (\Gb_n)_c$ or $(\Gb_n)_c\to (\Gb_n)_t$. We recall that on each diagram, left hand morphisms are defined as in \ref{cons:composedwedge}.
\end{definition}

\begin{definition}
Let $X, Y$ be two marked $\infty$-categories and $p:X\to Y$  a blind fibration having the right lifting property against cartesian trivializations.
\begin{enumerate}
\item The morphism $p$ is a \textit{naive right $1$-fibration} if it has the right lifting property against $\partial^c_1$, and if all cartesian $1$-simplices $c$ of $X$ are right cancellable over $p(c)$.
\item The morphism $p$ is a \textit{naive left $1$-fibration} if it has the right lifting property against $\partial^c_0$, and if all cartesian $1$-simplices $c$ of $X$ are left cancellable over $p(c)$.
\end{enumerate}
\end{definition}

\begin{lemma}
\label{lemma:duality_and_wedge}
There are zigzags of acyclic cofibrations:
\[\begin{tikzcd}
	{(\Gb_n)_{co}} & {\Gb_n} & {(\Gb_n)_{co}} & {\Gb_n} \\
	{(\Gb^n\vee \Delta[1]_t)_{co}} & {\Gb^n\vee \Delta[1]_t,} & {(\Delta[1]_t\vee \Gb^n)_{co}} & {\Delta[1]_t\vee \Gb^n,} \\
	{(\Gb_n)^{op}} & {\Gb_n} & {(\Gb_n)^{op}} & {\Gb_n} \\
	{(\Gb^n\vee \Delta[1]_t)^{op}} & {\Delta[1]_t\vee \Gb^n,} & {(\Delta[1]_t\vee \Gb^n)^{op}} & {\Gb^n\vee \Delta[1]_t.}
	\arrow["{\triangledown_{co}}"', from=1-1, to=2-1]
	\arrow["\triangledown", from=1-2, to=2-2]
	\arrow["{\triangledown_{co}}"', from=1-3, to=2-3]
	\arrow["\triangledown", from=1-4, to=2-4]
	\arrow["{\triangledown^{op}}"', from=3-1, to=4-1]
	\arrow["\triangledown", from=3-2, to=4-2]
	\arrow["{\triangledown^{op}}"', from=3-3, to=4-3]
	\arrow["\triangledown", from=3-4, to=4-4]
	\arrow[""{name=0, anchor=center, inner sep=0}, curve={height=-12pt}, draw=none, from=1-1, to=2-1]
	\arrow[""{name=1, anchor=center, inner sep=0}, curve={height=-12pt}, draw=none, from=3-1, to=4-1]
	\arrow[""{name=2, anchor=center, inner sep=0}, curve={height=12pt}, draw=none, from=1-2, to=2-2]
	\arrow[""{name=3, anchor=center, inner sep=0}, curve={height=12pt}, draw=none, from=3-2, to=4-2]
	\arrow[""{name=4, anchor=center, inner sep=0}, curve={height=-12pt}, draw=none, from=1-3, to=2-3]
	\arrow[""{name=5, anchor=center, inner sep=0}, curve={height=12pt}, draw=none, from=1-4, to=2-4]
	\arrow[""{name=6, anchor=center, inner sep=0}, curve={height=-12pt}, draw=none, from=3-3, to=4-3]
	\arrow[""{name=7, anchor=center, inner sep=0}, curve={height=12pt}, draw=none, from=3-4, to=4-4]
	\arrow[squiggly, tail reversed, from=0, to=2]
	\arrow[squiggly, tail reversed, from=1, to=3]
	\arrow[squiggly, tail reversed, from=4, to=5]
	\arrow[squiggly, tail reversed, from=6, to=7]
\end{tikzcd}\]
\end{lemma}
\begin{proof}
We remark
that  the realization functor $\tilde{R}$ commutes with the odd and the even dualities. We can then apply the results of \cite{ozornova2020suspension} and \cite{ozornova2020fundamental}.
\end{proof}

\begin{definition}
Let $X, Y$ be two marked $\infty$-categories and $p:X\to Y$  a blind fibration having the right lifting property against cartesian trivializations.
\begin{enumerate}
\item $p$ is a \textit{naive right fibration} if and only if for all pairs of paralleled $n$-cells $a,b$, when $n$ is odd (resp. when $n$ is even) the map $X(a,b)\to Y(pa,pb)$ is a naive right $1$-fibration (resp. naive  left $1$-fibration).
\item $p$ is a \textit{naive left fibration} if and only if for all pairs of paralleled $n$-cells $a,b$, when $n$ is odd (resp. when $n$ is even) the map $X(a,b)\to Y(pa,pb)$ is a naive left $1$-fibration (resp. naive  right $1$-fibration).
\item $p$ is a \textit{naive co-right fibration} if and only if for all pairs of paralleled $n$-cells $a,b$, the map $X(a,b)\to Y(pa,pb)$ is a naive right $1$-fibration. 
\item $p$ is a \textit{naive co-left fibration} if and only if for all pairs of parallel $n$-cells $a,b$, the map $X(a,b)\to Y(pa,pb)$ is a naive left $1$-fibration. 
\end{enumerate}
\end{definition}

\begin{prop}
\label{prop:action_of_duality_on_naiv_fib}
Let $X, Y$ be two marked $\infty$-categories and $p:X\to Y$  a blind fibration having the right lifting property against cartesian trivializations.
\begin{enumerate}
\item $p$ is a naive left fibration if and only if $p^{\circ}$ is a naive right fibration.
\item $p$ is a naive co-right fibration if and only if $p^{co}$ is a naive right fibration.
\item $p$ is a naive co-left fibration if and only if $p^{op}$ is a naive right fibration.
\end{enumerate}
\end{prop}
\begin{proof}
According to corollary \ref{cor:Sigma_and_duality}, lemma \ref{lemma:duality_and_wedge} and adjunction property, if $n$ is even, $p^{\circ}$  has the right lifting property against 
$$\Sigma^n \{0\}\to \Sigma^n \Delta[1]_c~~~ \mbox{  and  }~~~\Sigma^n((K\fwedge\Delta[1]_c)\cup_{\triangledown} L)\to \Sigma^n (K\fwedge\Delta[1]_c)$$
if and only if 
$p$ has the right lifting property against:
$$\Sigma^n \{1\}\to \Sigma^n \Delta[1]_c ~~~\mbox{  and  }~~~\Sigma^n((\Delta[1]_c\fwedge K)\cup_{\triangledown}L)\to \Sigma^n (\Delta[1]_c\fwedge L)$$
where $K\to L$ is $\partial\Gb_n\to\Gb_n$, $\Gb_n\to (\Gb_n)_c$ and   $(\Gb_n)_c\to(\Gb_n)_t$. If $n$ is odd, $p^{\circ}$  has the right lifting property against 
$$\Sigma^n \{1\}\to \Sigma^n \Delta[1]_c~~~ \mbox{  and  }~~~\Sigma^n((\Delta[1]_c\fwedge K)\cup_{\triangledown}L)\to \Sigma^n (\Delta[1]_c\fwedge L)$$
if and only if 
$p$ has the right lifting property against:
$$\Sigma^n \{0\}\to \Sigma^n \Delta[1]_c ~~~\mbox{  and  }~~~\Sigma^n((K\fwedge\Delta[1]_c)\cup_{\triangledown} L)\to \Sigma^n (K\fwedge\Delta[1]_c)$$
where $K\to L$ is $\partial\Gb_n\to\Gb_n$, $\Gb_n\to (\Gb_n)_c$ and   $(\Gb_n)_c\to(\Gb_n)_t$. This proves the first assertion and the others are demonstrated with similar means.
\end{proof}

\begin{theorem}
\label{theo:naive_fib_are_fib}
Let $X, Y$ be two marked $\infty$-categories and $p:X\to Y$  a blind fibration having the right lifting property against cartesian trivializations. Then $p$ is a right (resp. left, co-right, co-left) fibration if and only if it is a naive right (resp. left, co-right, co-left) fibration.
\end{theorem}
\begin{proof}
Suppose $f$ is a right fibration. It is obvious that it is a naive right $1$-fibration. According to corollary \ref{cor:Sigma_and_duality}, lemma \ref{lemma:duality_and_wedge} if $n$ is even, 
$$\Sigma^m \{1\}\to \Sigma^m \Delta[1]_t~~~ \mbox{  and  }~~~\Sigma^m(K\fwedge\Delta[1]_c)\cup_{\triangledown} L\to \Sigma^m L\fwedge\Delta[1]_c$$
is a right acyclic cofibration, and if $n$ is odd, 
$$\Sigma^m \{0\}\to \Sigma^m \Delta[1]_t~~~ \mbox{  and  }~~~\Sigma^m(L\cup_{\triangledown} \Delta[1]_c\fwedge K)\to \Delta[1]_c\fwedge\Sigma^m L$$
is a right acyclic cofibration.
That shows that $f$ is a right naive fibration.

For the converse, suppose that $f$ is a naive right fibration. Results of \ref{section:result_for_interval} then imply that $f$ has the right lifting property against: 
$$K\otimes \Delta[1]_c\cup L\otimes \{1\}\to L\otimes \Delta[1]_c~~~\mbox{and}~~~
\Gb_n\times \Delta[1]_c\cup (\Gb_n)_t\times \{1\}\to (\Gb_n)_t\times \Delta[1]_c
$$
where $K\to L$ is $\partial\Gb_n\to \Gb_n$ or $\Gb_n\to (\Gb_n)_t$.
The morphism
$$\Hom_{\otimes}(\Delta[1]_c,X)\to \Hom_{\otimes}(\Delta[1]_c,Y)\times_{\Hom_{\otimes}(\{1\},Y)}\Hom_{\otimes}(\{1\},X)$$
then has the right lifting property against $\partial\Gb_n\to \Gb_n$ and $\Gb_n\to (\Gb_n)_t$. Theorem \ref{theo:f_weak_equivalence_ssi_f_G_equivalence}
then implies that it has the right lifting property against all cofibrations of marked simplicial sets. The morphism
$$\Hom(\Delta[1]_c,X)\to \Hom(\Delta[1]_c,Y)\times_{\Hom(\{1\},Y)}\Hom(\{1\},X)$$
has the right lifting property against $\partial\Gb_n\to \Gb_n$ and $\Gb_n\to (\Gb_n)_c$. Theorem \ref{theo:f_weak_equivalence_ssi_f_G_equivalence}
then implies that it has the right lifting property against all cofibrations of marked simplicial sets.

The other assertions are proved similarly.
\end{proof}

\begin{cor}
\label{prop:action_of_duality_on_fib}
Let $X, Y$ be two marked $\infty$-categories and $p:X\to Y$  a blind fibration having the right lifting property against cartesian trivializations.
\begin{enumerate}
\item $f$ is a  left fibration if and only if $f^{\circ}$ is a  right fibration.
\item $f$ is a  co-right fibration if and only if $f^{co}$ is a  right fibration.
\item $f$ is a  co-left fibration if and only if $f^{op}$ is a  right fibration.
\end{enumerate}
\end{cor}
\begin{proof}
This is a direct consequence of theorem  \ref{theo:naive_fib_are_fib} and proposition \ref{prop:action_of_duality_on_naiv_fib}.
\end{proof}

\begin{definition}
For $A$ a marked $\infty$-category,  we define several categories of fibrant objects:
\begin{enumerate}
\item $RFib(A)$ is the category of fibrant objects of the right model structure on $\bmSset_{/X}$.
\item $LFib(A)$ is the category of fibrant objects of the left model structure on $\bmSset_{/X}$.
\item $cRFib(A)$ is the category of fibrant objects of the co-right model structure on $\bmSset_{/X}$.
\item $cLFib(A)$ is the category of fibrant objects of the co-left model structure on $\bmSset_{/X}$.
\end{enumerate}
\end{definition}

\begin{lemma}	
Let $i,j$ be two left adjoints, $N_i,N_i$ their associated right adjoints, and $\psi:i\to j$ an invertible natural transformation that is a pointwise acyclic cofibration. Then, the induced natural transformation $N_j\to N_i$ is a trivial fibration on all fibrant objects.
\end{lemma}
\begin{proof}
Let $i:K\to L$ be a cofibration.
We have a diagram:
\[\begin{tikzcd}
	{i(K)} & {j(K)} \\
	{i(L)} & {i(L)\cup j(K)} & {j(l)}
	\arrow[from=1-1, to=2-1]
	\arrow["\sim", from=1-1, to=1-2]
	\arrow["\sim", from=2-1, to=2-2]
	\arrow[from=1-2, to=2-2]
	\arrow["{ \sim}"', curve={height=18pt}, from=2-1, to=2-3]
	\arrow[from=1-2, to=2-3]
	\arrow[from=2-2, to=2-3]
\end{tikzcd}\]
showing that $j(K)\cup i(L)\to j(L)$ is a trivial cofibration. Now, let $B$ be a fibrant object. We can see that the two following lifting problems are equivalent:
\[\begin{tikzcd}
	K & {N_j(B)} & {i(L)\cup j(K)} & B \\
	L & {N_i(B)} & {j(l).}
	\arrow[""{name=0, anchor=center, inner sep=0}, from=1-3, to=2-3]
	\arrow[from=1-1, to=1-2]
	\arrow[from=2-1, to=2-2]
	\arrow[from=1-1, to=2-1]
	\arrow[""{name=1, anchor=center, inner sep=0}, from=1-2, to=2-2]
	\arrow[from=1-3, to=1-4]
	\arrow["\exists"', dotted, from=2-1, to=1-2]
	\arrow["\exists"', dotted, from=2-3, to=1-4]
	\arrow[shorten <=9pt, shorten >=9pt, Rightarrow, 2tail reversed, from=1, to=0]
\end{tikzcd}\]
\end{proof}

\begin{prop}
There is an equivalence of categories of fibrant objects:
\[\begin{tikzcd}
	{RFib(A)} & {cRFib(A^{co})} \\
	{cLFib(A^{op})} & {LFib(A^{\circ}).}
	\arrow["op"', from=1-1, to=2-1]
	\arrow["co", from=1-1, to=1-2]
	\arrow["op", from=1-2, to=2-2]
	\arrow["co"', from=2-1, to=2-2]
	\arrow["\circ"{description}, from=1-1, to=2-2]
\end{tikzcd}\]
\end{prop}
\begin{proof}
We will show that $(\uvar)^{co}:RFib(A) \to cRFib(A)$ is an equivalence of categories of fibrant objects. The other are proven similarly. 
In the proof of theorem \ref{theo:co_is_an_duality}, we construct the following zigzag of weakly invertible natural transformations.
$$(\uvar)_{coco}\to i_{str} \leftarrow id.$$
Remark that the right one is a pointwise acyclic cofibration.
This induces a zigzag of weakly invertible natural transformations between right adjoints, where, according to last lemma, the right one is a pointwise  trivial fibration:
$$(\uvar)_{coco}\leftarrow N_{i_{str}}\to id.$$
Evaluating on $A$ this induces a spine: 
$$A^{coco}\xleftarrow{f}N_{i_{str}}(A) \xrightarrow{g} A$$
where $g$ is a trivial fibration and $f$ a weak equivalence. 
We then define the composite functor:
$$ \psi:RFib(A^{coco}) \xrightarrow{f^*} RFib(N_{i_{str}}(A))\xrightarrow{g_!} RFib(A)$$
 Remark that we have a diagram:
\[\begin{tikzcd}
	{X^{coco}} & {f^*X^{coco}} & {N_{i_{str}}(X)} & X \\
	& {A^{coco}} & {N_{i_{str}}(A)} & A.
	\arrow[""{name=0, anchor=center, inner sep=0}, from=1-1, to=2-2]
	\arrow["f", from=2-3, to=2-2]
	\arrow["\sim"', from=1-2, to=1-1]
	\arrow[from=1-2, to=2-3]
	\arrow[from=1-3, to=1-2]
	\arrow["\sim"', curve={height=18pt}, from=1-3, to=1-1]
	\arrow[from=1-3, to=2-3]
	\arrow["g"', from=2-3, to=2-4]
	\arrow["\sim", from=1-3, to=1-4]
	\arrow[from=1-4, to=2-4]
	\arrow["\lrcorner"{anchor=center, pos=0.125, rotate=-90}, draw=none, from=1-2, to=0]
\end{tikzcd}\]
By two out of three, the morphism $N_{i_{str}}(X)\to f^*X^{coco}$ is a weak equivalence, and so  the composite 
$$RFib(A) \xrightarrow{co} cRFib(A^{co}) \xrightarrow{co} RFib(A^{coco})\xrightarrow{\psi} RFib(A)$$ is weakly equivalent to the identity. 
The functor $RFib(A) \xrightarrow{co} cRFib(A^{co})$ admits a right inverse and the functor $cRFib(A^{co}) \xrightarrow{co} RFib(A^{coco})$ admits a left inverse. We show similarly that for every $B$, the functor $cRFib(B) \xrightarrow{co} RFib(B^{co})$ admits a right inverse. If we apply this for $B:=A^{co}$, we deduce that $cRFib(A^{co}) \xrightarrow{co} RFib(A^{coco})$ also admits a right inverse, and so is an equivalence. This implies that  $RFib(A) \xrightarrow{co} cRFib(A^{co})$ is an equivalence.
\end{proof}

\appendix

\section{Cisinski-model structures on stratified presheaves}
\label{section:appendix1}

\subsection{Stratified presheaves} 

\begin{definition}
An \textit{Eilenberg-Zilber category} is a quadruple $(A, A_+ , A_- , d)$
where $A$ is a small category, while $A_+$ and $A_-$ are subcategories of $A$, and
$d : Ob(A) \to \mathbb{N}$ is a function with values in the set of non-negative integers,
such that the following properties are verified:
\begin{enumerate}
\item any isomorphism of $A$ is in both $A_+$ and $A_-$ . Moreover, for any isomorphic
objects $a$ and $b$ in $A$, we have $d(a) = d(b)$;
\item if $a \to b$ is a morphism in $A_+$ (resp. in $A_-$ ) that is not an identity, then  $d(a) < d(b )$ (resp. $d(a) > d(b)$);
\item any morphism $u : a \to b$ in $A$ has a unique factorization of the form
$u = ip$, with $p : a \to c$ in $A_-$ and $i : c \to b$ in $A_+$ ;
\item if a morphism $\pi : a \to b$ belongs to $A_-$ there exists a morphism $\sigma : b \to
a$ in $A$ such that $\pi\sigma = 1_b$; moreover, for any two morphisms in $A_-$ of the
form $\pi_1,\pi_2 : a \to b$, if $\pi_1$ and $\pi_2$ have the same set of sections, then they
are equal.
\end{enumerate}
\end{definition}

\begin{construction}
Let $(A, A_+ , A_- , d)$ be a   Eilenberg-Zilber category. 
Let $tA$ be the category whose object are $ob(A)\cup\{a_t,a\in ob(A_{>0})\}$ and morphisms given by: 
$$\begin{array}{rcl}
\Hom_{tA}(a,b)&:=& \Hom_A(a,b),\\
\Hom_{tA}(a,b_t)&:=& \Hom_A(a,b),\\
\Hom_{tA}(a_t,b)&:=& \Hom_A(a,b)\cap A_- \diagdown \{id_a\}, \\
\Hom_{tA}(a_t,b_t)&:=&\Hom_A(a,b)\cap A_-.\\
\end{array}$$

We also define the subcategory $tA_+$ as the smaller one that includes $A_+$ and all morphisms $a\to a_t$, and $tA_-$  whose morphisms are of shape $a_t\to b$ or $a_t\to b_t.$
Eventually, we define $d_{t}(a):= 2d(a)$ and $d_{t}(a_t) := 2d(a)+1$.

The  Eilenberg-Zilber category $(tA, tA_+ , tA_- , d_{t})$ is the \textit{stratification of $(A, A_+ , A_- , d)$}. 
Let $B$ be a subset of $ob(A_{>0})$. We define the \textit{$B$-saturation} of $(A, A_+ , A_- , d)$ to be the sub  Eilenberg-Zilber category whose objects are $$ob(A)\cup \{a_t,a\in B\},$$ denoted $(t_BA, t_BA_+ , t_BA_- , d_{t})$ .
\end{construction}

\begin{construction}
Let $(A, A_+ , A_- , d)$ be a   Eilenberg-Zilber category. 
Let $ctA$ be the category whose object are $ob(A)\cup\{a_c,a\in ob(A_{>0})\}\cup,\{a_t,a\in ob(A_{>0})\}$ and morphisms given by: 
$$\begin{array}{rcl}
\Hom_{ctA}(a,b)&:=& \Hom_A(a,b),\\
\Hom_{ctA}(a,b_c)&:=& \Hom_A(a,b),\\
\Hom_{btA}(a,b_t)&:=& \Hom_A(a,b),\\
\\
\Hom_{ctA}(a_c,b)&:=& \Hom_A(a,b)\cap A_- \diagdown \{id_a\},\\
\Hom_{ctA}(a_c,b_c)&:=&\Hom_A(a,b)\cap A_-,\\
\Hom_{ctA}(a_c,b_t)&:=&\Hom_A(a,b)\cap A_-,\\
\\	
\Hom_{ctA}(a_t,b)&:=& \Hom_A(a,b)\cap A_- \diagdown \{id_a\},\\
\Hom_{ctA}(a_t,b_c)&:=&\Hom_A(a,b)\cap A_- \diagdown \{id_a\},\\
\Hom_{ctA}(a_t,b_t)&:=&\Hom_A(a,b)\cap A_-.\\
\end{array}$$

We also define the subcategory $ctA_+$ as the smaller one that includes $A_+$ and all morphisms $a\to a_c$ and $a_t\to a_c$. The subcategory $ctA_-$  is the one whose morphisms are of shape $a_c\to \uvar$ and $a_t\to \uvar$.
Eventually, we define $d_{ct}(a):= 3d(a)$, $d_{ct}(a_c) := 3d(a)+1$ and $d_{ct}(a_t) := 3d(a)+2$.

The  Eilenberg-Zilber category $(ctA, ctA_+ , ctA_- , d_{ct})$ is the \textit{ bistratification of $(A, A_+ , A_- , d)$}.
Let $B, C$ be two subsets of $ob(A_{>0})$. We define the \textit{$(B,C)$-saturation} of $(A, A_+ , A_- , d)$ to be the sub  Eilenberg-Zilber category whose objects are $$ob(A)\cup \{a_c,a\in B\}\cup \{a_t, a\in C\},$$  noted $(c_Bt_CA, c_Bt_CA_+ , c_Bt_CA_- , d_{ct})$.
\end{construction}

\begin{remark}
The bistratification of $A$  is the $B$-saturation of $A_t$ where $B=\{a_t,a\in A\}.$
\end{remark}

\begin{definition}
\label{defi:stratified_presheaf}
Let $A$ be a Eilenberg-Zilber category and $B\subset ob(A_{>0})$. The category of \textit{$B$-stratified presheaves on $A$}, denoted $\relativtPsh{A}{B}$, is the full subcategory of $\Psh{t_BA}$ composed of presheaves $X$ such that $X(a_t)\to X(a)$ is a monomorphism.
\end{definition}

An object of $\relativtPsh{A}{B}$ is then just a couple $(X,tX)$ where $X$ is a presheaf on $A$ and $tX :=\{tX_a\}_{a\in B}$ is a family of sets such that, for $a\in B$, $tX_a$ is a subset of  $X_a$ that includes all degenerate simplices. 
There is a functor 
$$\begin{array}{rrcl}
\pi: &\Psh{tA}&\to & \relativtPsh{A}{B}\\
&X&\mapsto & (X_{|A},tX)
\end{array}$$
where $(tX)_a$ is the image of the morphism $X(a_t)\to X(a)$. This functor is the left adjoint to the forgetful functor $\iota: \relativtPsh{A}\to \Psh{tA}$.

\begin{example}
Let $A$ be an   Eilenberg-Zilber category and $X\in \relativtPsh{A}{B}$. The category $(tA)_{/X}$ is a Eilenberg-Zilber category and is the $tX$-saturation of $A_{/X}$. We then have an equivalence of categories: 
$$\relativtPsh{A}{B}_{/X} \cong \relativtPsh{A/X}{tX}.$$
\end{example} 

\begin{definition}
\textit{A saturation set $S$ for $\relativtPsh{A}{B}$} is a set of  morphisms of $\relativtPsh{A}{B}$ that are sequence of pushouts along morphisms of shape $a\to a_t$. A stratified presheaf $X$ is \textit{$S$-saturated} if it has the right lifting property against morphisms in $S$.
The full subcategory of $S$-saturated stratified presheaves is denoted $\relativtPsh{A}{B}_S$.
\end{definition}

Let $S$ be a saturation set. There is a functor
$$\begin{array}{rrcl}
R:\relativtPsh{A}{B}&\to & \relativtPsh{A}{B}_S\\
(X,tX)&\to &(X,\overline{tX})
\end{array}$$
where $\overline{tX}$ is the smaller set making $(X,\overline{tX})$ an object of $\relativtPsh{A}{B}_S$, and such that $tX\subset \overline{tX}$. This functor is the left adjoint to the obvious forgetful functor $i:\relativtPsh{A}{B}_S\to \relativtPsh{A}{B}_S$.

Furthermore the category $\relativtPsh{A}{B}_S$ is complete and cocomplete.
If $F:I\to \relativtPsh{A}{B}_S$ is any diagram, it's colimit is given by the formula:
$$\Colim F := (\Colim X_i, \overline{\Colim tX_i})$$ 
and it's limit by
$$\Lim F := (\Lim X_i, \Lim tX_i)$$ 
where $F(i)=:(X_i,tX_i)$.

\begin{example}
If $(X,tX)$ is a $S$-saturated presheaf, we have an equivalence of categories:
$$(\relativtPsh{A}{B}_S)_{/X}\cong(\relativtPsh{A_{/X}}{tX}_{S/X}).$$ 
\end{example}

All these results have obvious analogues for bistratified presheaves.
\begin{definition}
Let $A$ be a Eilenberg-Zilber category and $B\subset C\subset ob(A_{<0})$. The category of $(A,B)$-bistratified presheaves on $A$, denoted $\relativctPsh{A}{B}{C}$ is the full subcategory of $\Psh{c_Bt_CA}$ composed of presheaf $X$ such that $X(a_t)\to X(a_c)$ and $X(a_c)\to X(a)$ are monomorphism.

\textit{A saturation set $S$ for $\relativctPsh{A}{B}{C}$} is  a set of morphisms of $\relativctPsh{A}{B}{C}$ that are sequences of pushouts along morphisms shape $a\to a_c$ and $a_c\to a_t$. A bistratified presheaves $X$ is \textit{$S$-saturated} if it has the right lifting property against morphism in $S$.
The full subcategory of $S$-saturated stratified presheaves is denoted $\relativctPsh{A}{B}{C}_S$.
These categories are complete and cocomplete.
\end{definition}

\subsection{A criterium to transfer a model structure.}

\begin{definition}
A \textit{cellular model} for a category $C$ is a set of monomorphisms $Cel(C)$ such that every monomorphism is a sequence of pushouts along morphisms in $Cel(C)$.
\end{definition}

\begin{remark}
Presheaves on a Eilenberg-Zilber category  $A$ always admit cellular models: the cellular model for $A$ is 
$$Cel(A):= \{\partial a\to a,~a\in A\},$$
the cellular model of $tA$ is given by: 
$$Cel(tA):= Cel(A)\cup \{a\to a_t,~a\in A_{>0}\}$$ and the cellular model of  $ctA$ is :
$$Cel(ctA):= Cel(A)\cup \{a\to a_c,~a\in A_{>0}\}\cup \{a_c\to a_t,~a\in A_{>0}\}.$$
\end{remark}

We fix an complete and cocomplete category  $C$ admitting a cellular model, and $T$ a set of monomorphisms.

\begin{definition}
An \textit{interval object} is an object $I$ together with morphisms: 
\[\begin{tikzcd}
	1\coprod1 & I & 1
	\arrow["\sigma"', from=1-2, to=1-3]
	\arrow["{\partial_0\amalg\partial_1}"', from=1-1, to=1-2]
	\arrow["{id\coprod id}", curve={height=-12pt}, from=1-1, to=1-3]
\end{tikzcd}\]
such that $\partial_0\coprod\partial_1$ is a monomorphism. For $\epsilon\in\{0,1\}$, the sub-object corresponding to the image of $\partial_\epsilon$ is noted $\{\epsilon\}$.
\end{definition}

\begin{definition}
Let $f,g;A\to X$ be two morphisms. An \textit{$I$-homotopy} between $f$ and $g$ is a lifting in the following diagram:
\[\begin{tikzcd}
	{A\coprod A} & X \\
	{A\times I}
	\arrow[dotted, from=2-1, to=1-2]
	\arrow[from=1-1, to=2-1]
	\arrow["{f\amalg g}", from=1-1, to=1-2]
\end{tikzcd}\]
Let $\sim$ be the smaller transitive, symmetric and reflexive relation on $\Hom(A,X)$ such that $f\sim g$ when there exists a homotopy between $f$ and $g$.
The set of equivalence classes is denoted $[A,X]$.
\end{definition}

The Leibniz  product is defined as usual. 
If $f:K\to L$ and $g:X\to Y$ are two morphisms, then $f\hat{\times}g$ is the canonical arrow:
$$f\hat{\times}g ~:= K\times Y\amalg_{K\times X}L\times X\to L\times Y.$$
\begin{construction}
\label{cons:Lambda}
Let $T$ be a class of monomorphisms. We define 
$$\Gamma T:= T~ \hat{\times}~ Cel(D)$$
\end{construction}

\begin{definition}
\label{defi:cof}
We define several classes of maps of $C$:
\begin{itemize}[leftmargin=* ,parsep=0cm,itemsep=0cm,topsep=0cm]
\item A \textit{cofibration} is a monomorphism.
\item A \textit{trivial fibration} is a morphism having the right lifting property against cofibrations.
\item An \textit{anodyne extension} is map obtained as a sequence of pushouts along  morphisms in $$T \cup \Gamma \{\{\epsilon\}\to I, \epsilon\in \{0,1\}\}.$$
\item A \textit{naive fibration} is a morphism having the right lifting property against anodyne extensions.
\item An object $X$ is \textit{fibrant} if $X\to 1$ is a naive fibration.
\item A morphism $A\to B$ is a \textit{weak equivalence} if for all fibrant objects $X$, the induced morphism is a bijection: 
$$f^*:[B,X]\to [A,X].$$ 
\item An \textit{acyclic cofibration} is both cofibration and a weak equivalence.
\item A \textit{fibration} is a morphism having the right lifting property against acyclic cofibration.
\end{itemize}
\end{definition}

\begin{definition}
\label{defi:local_model_structure}
If it exists,  a \textit{$(T,I)$-local model structure on $C$} is a monoidal model structure on $C$ such that
\begin{enumerate}
\item  cofibrations, fibrations and weak equivalences are defined as in \ref{defi:cof},
\item A object $X$ is fibrant whenever $X\to 1$ is a naive fibration and fibrations between fibrant objects correspond to naive fibrations.
\end{enumerate}
\end{definition}

\begin{definition}
\label{defi:localization_of_model_structure}
Let $Q$ be any set of morphism. A \textit{left Boosfield localization} of $(T,I)$-local model structure is a $(Q\cup T,I)$-local model structure.
\end{definition}

Suppose now that we have a reflexive full subcategory $C$ of $D$ that includes an interval object $I$. We then have an adjoint pair:
\[\begin{tikzcd}
	D & C
	\arrow[""{name=0, anchor=center, inner sep=0}, "G", curve={height=-6pt}, from=1-2, to=1-1]
	\arrow[""{name=1, anchor=center, inner sep=0}, "F", curve={height=-6pt}, from=1-1, to=1-2]
	\arrow["\dashv"{anchor=center, rotate=-90}, draw=none, from=1, to=0]
\end{tikzcd}\]

The rest of the section is devoted to the proof of the following theorem:

\begin{theorem}
\label{theo:transfered_model_structure}
Let $T$ be a set of monomorphisms in $C$ and $S$ a set of monomorphisms in $D$  such that $F(f)=id$ for all $f\in S$ . Suppose that there exists a $(\Gamma(GT\cup S),GI)$-local model structure on $D$ and that:
\begin{enumerate}
\item $C$ admits a cellular model $Cel(C)$, and $G(Cel(C))$ is a cellular model for $D$.
\item $C$ is closed by finite limits and $F$ preserves them.
\item $G$ preserves cofibrations.
\item for every fibration $f:X\to G(Y)$, $F f:F(X)\to Y$ is a fibration.
\item $\eta:id\to GF$ is a pointwise weak equivalence.
\end{enumerate}
Then there exists a $(\Gamma T,I)$-local structure on $C$. The adjunction is then a Quillen equivalence.
\end{theorem}

From now on, we suppose that the pair $(F,G)$ satisfies all the hypotheses of the theorem.

\begin{lemma}
\label{lem:G_detect_trivial_fibration}
The functor $G$ preserves  and detects trivial fibrations.  
\end{lemma}
\begin{proof}
Trivial fibrations are detected by morphisms in the image of $G$, which is a full inclusion.
\end{proof}

\begin{lemma}
\label{lem:G_preserves_naive_fibration}
The functor $G$ preserves  naive fibrations.  
\end{lemma}
\begin{proof}
Let $f$ be a naive fibration. By adjunction, $G(f)$ has the right lifting property against morphisms of $\Gamma(GT\cup S)$ and $\Gamma \{\{\epsilon\}\to I\}$ if and only if $f$ has the right lifting property against morphisms of $F(\Gamma(GT\cup S))$ and $F\Gamma \{\{\epsilon\}\to G I\}$. By hypothesis,  $F$ preserves finite limits, and so Leibniz  products. These two sets are then respectively equal to $\Gamma (T\cup S)$ and $\Gamma \{\{\epsilon\}\to I\}$, which concludes the proof. 
\end{proof}

\begin{lemma}
\label{lem:G_preserve_weak_equivalence}
A morphism $f:A\to B$ in $C$ is a weak equivalence if and only if $G(f):G(A)\to G(B)$ is a weak equivalence in $D$.
\end{lemma}
\begin{proof}
The functor $G$ being full, it induces isomorphism:
$$[A,X]\cong[GA,GX].$$
Let $X$ be any fibrant object in $C$. The forth condition of \ref{theo:transfered_model_structure} implies that $F$  preserves fibrant objects.
 One then has a commutative diagram:
\[\begin{tikzcd}
	{[G(B),X]} & {[G(B),GF(X)]} & {[B,F(X)]} \\
	{[G(A),X]} & {[G(A),GF(X)]} & {[A,F(X)]}
	\arrow["{(\eta_X)_!}", from=1-1, to=1-2]
	\arrow["{(\eta_X)_!}"', from=2-1, to=2-2]
	\arrow[from=1-1, to=2-1]
	\arrow[from=1-2, to=2-2]
	\arrow["\sim"', from=1-3, to=1-2]
	\arrow["\sim", from=2-3, to=2-2]
	\arrow[from=1-3, to=2-3]
\end{tikzcd}\]
where horizontal applications are bijections. The functor $F$ being surjective, this implies the result.
\end{proof}

\begin{lemma}
\label{lem:F_preserves_weak_equivalence}
The functor $F$ preserves weak equivalences.
\end{lemma}
\begin{proof}
By assumption, we have an isomorphism $F(X\times G(I)) = F(X)\times I$. We then have for all object $A$, a bijection :
$$[F(A),X]\cong[ A,G(X)].$$
Let $f:A\to B$ be a weak equivalence in $D$, and $X$ a fibrant object of $C$. We then have a diagram:
\[\begin{tikzcd}
	{[F(B),X]} & {[B,G(X)]} \\
	{[F(A),X]} & {[ A,G(X)]}
	\arrow[from=1-1, to=2-1]
	\arrow["\sim", from=1-1, to=1-2]
	\arrow["\sim"', from=2-1, to=2-2]
	\arrow[from=1-2, to=2-2]
\end{tikzcd}\]
where all horizontal morphisms are bijections. Lemma \ref{lem:G_preserves_naive_fibration} implies that $G(X)$ is fibrant. The right hand morphism is then a bijection, and so is the left one, which implies that $F(f)$ is a weak equivalence.
\end{proof}

\begin{proof}[Proof of theorem \ref{theo:transfered_model_structure}]
Let $f:A\to B$ be any morphism in $C$. We factorize $G(f)$ in an acyclic cofibration $j$ followed by a fibration $p$. There is a factorization in an acyclic cofibration and fibration as follows. The couple $(F(j),F(p))$ is a factorization of $f$,  $F(j)$ is an acyclic cofibration according to \ref{lem:F_preserves_weak_equivalence}, and the forth condition of \ref{theo:transfered_model_structure} implies that $F(p)$ is a fibration.
The factorization in a cofibration followed by an trivial fibration exists thanks to the small object argument. 

Let $f$ be both a fibration and a weak equivalence in $C$. According to \ref{lem:G_preserve_weak_equivalence}  $G(f)$ is a weak equivalence.
Furthermore, according to \ref{lem:F_preserves_weak_equivalence},
$F$ preserves acyclic cofibrations. The maps $G(f)$ is then a trivial fibration, and so is $f$ according to \ref{lem:G_detect_trivial_fibration}.

The fact that we get a monoidal model structure comes from the fact that $F$ commutes with finite products.

Eventually the adjonction is a Quillen adjonction because $F$ preserves cofibrations and weak equivalences. This is a Quillen equivalence because $id\to GF$ is a weak equivalence and $GF = id$.
\end{proof}

\subsection{Model structures on stratified presheaves}

For the rest of this section, we fix an Eilenberg-Zilber category $A$, $B$ a subset of $ob(A_{>0})$, $I$ an interval object in $\relativtPsh{A}{B}$ and  $T$ a class of monomorphisms of $\relativtPsh{A}{B}$.
We then have an adjonction 
\[\begin{tikzcd}
	{\Psh{t_BA}} && {\relativtPsh{A}{B}}
	\arrow[""{name=0, anchor=center, inner sep=0}, "\pi", curve={height=-6pt}, from=1-1, to=1-3]
	\arrow[""{name=1, anchor=center, inner sep=0}, "\iota", curve={height=-6pt}, from=1-3, to=1-1]
	\arrow["\dashv"{anchor=center, rotate=-90}, draw=none, from=0, to=1]
\end{tikzcd}\]

The data of $\iota(I)$ and of the class of $(\iota(I),\iota(\Gamma T))$-anodyne extensions, is a homotopical structure on $\Psh{t_BA}$ in the sense of \cite{cisinski2019higher}.
Theorem \cite[Theorem 2.4.19]{cisinski2019higher} implies that the existence of a $(I,\Gamma T )$-local models structure on $\Psh{t_BA}$.

We will use  the theorem \ref{theo:transfered_model_structure}, apply to the previous adjunction, 
to define one on $\relativtPsh{A}{B}$.

\begin{lemma}
\label{lem:pi_preserve_trivial_fibration}
The functor $\pi$ preserves trivial fibrations.
\end{lemma}
\begin{proof}
Straightforward.
\end{proof}

\begin{lemma}
\label{lem:X_to_pi(X)_is_a_trivial_fibration}
The morphism $X\xrightarrow{\pi} \iota\pi(X)$ is a trivial fibration. 
\end{lemma}
\begin{proof}
This morphism obviously has the right lifting property against $\partial a\to a$ and $a\to a_t$, which is a cellular model of $tA$.
\end{proof}

\begin{lemma}
\label{lem:iota_preserve_fibrant object}
If $X$ is fibrant, so is $\iota \pi X$.
\end{lemma}
\begin{proof}
Let $X$ be a fibrant object, and $i:K\to L$ an anodyne extension, and $K\to \iota \pi X$ any morphism. We consider the diagram: 
$$\begin{tikzcd}
	K & X \\
	L & {\iota\pi(X).}
	\arrow["{\eta_X}", from=1-2, to=2-2]
	\arrow[from=1-1, to=2-1]
	\arrow[from=1-1, to=2-2]
	\arrow["h"{description}, dotted, from=1-1, to=1-2]
	\arrow[dotted, from=2-1, to=1-2]
\end{tikzcd}$$
There exists a lifting $h:K\to X$ because $\eta_X$ is a trivial fibration, and there exists a lifting $l:L\to X$ because $X$ is fibrant. This shows that $\iota \pi X$ is fibrant.
\end{proof}

\begin{lemma}
\label{lem:iota_preserve_weak_equivalence}
The functor $\iota$ preserves weak equivalences.
\end{lemma}
\begin{proof}
The functor $\iota$ being full, it induces isomorphism:
$$[A,X]\cong[\iota (A),\iota (X)].$$
Let $X$ be any fibrant object in $\Psh{t_BA}$, and $i:A\to B$ a weak equivalence.
The last lemma implies that $\iota\pi(X)$ is fibrant.
 One then has a commutative diagram:
\[\begin{tikzcd}
	{[\iota(B),X]} & {[\iota(B),\iota\pi(X)]} & {[B,\pi(X)]} \\
	{[\iota(A),X]} & {[\iota(A),\iota\pi(X)]} & {[A,\pi(X)].}
	\arrow[from=1-1, to=2-1]
	\arrow[from=1-2, to=2-2]
	\arrow["\sim", from=2-3, to=2-2]
	\arrow["\sim"', from=1-3, to=1-2]
	\arrow[from=1-3, to=2-3]
	\arrow["\eta", from=1-1, to=1-2]
	\arrow["{\eta_X}"', from=2-1, to=2-2]
\end{tikzcd}\]
where horizontal applications are bijections. This implies that the left hand morphism is an isomorphism, and the morphism $\iota(A)\to \iota(B)$ is then a weak equivalence.
\end{proof}

\begin{lemma}
\label{lem:pi_preserve_fibration}
If $f:X\to \iota Y$ is a fibration  in $\Psh{t_BA}$, the induced morphism $\pi(f):\pi(X)\to Y$ is a fibration in $\relativtPsh{A}{B}$.
\end{lemma}
\begin{proof}
Let $i:K\to L$ be a trivial cofibration  in $\relativtPsh{A}{B}$. The functor $\iota$ preserves cofibrations, the lemma  \ref{lem:iota_preserve_weak_equivalence} then implies that  $\iota(i)$ is an acyclic cofibration. Because $\iota$ is full, it is enough to show that $\iota(\pi(f))$ has the right lifting property against $\iota(i)$. 
Consider the following diagram:
\[\begin{tikzcd}
	&& X \\
	{\iota (K)} && {\iota\pi (X)} \\
	{\iota(L)} && {\iota(Y).}
	\arrow[from=1-3, to=2-3]
	\arrow["{~}"{description}, from=2-1, to=2-3]
	\arrow[from=3-1, to=3-3]
	\arrow[from=2-1, to=3-1]
	\arrow[from=2-3, to=3-3]
	\arrow["h", dotted, from=2-1, to=1-3]
	\arrow["l", dotted, from=3-1, to=1-3]
\end{tikzcd}\]
One can find a lifting $h:\iota(K)\to L$ because $X\to \iota\pi(X)$ is a trivial fibration. The composite of the two right hand morphisms is a fibration, and there then exists a lift $l:\iota(L)\to X$. This induces the desired lift.
\end{proof}

\begin{theorem}
\label{theo:local_model_structure_on_stratified_presheave}
There exists a $(I,\Gamma T)$-local model structure on $\relativtPsh{A}{B}$.
\end{theorem}
\begin{proof}
We apply theorem \ref{theo:transfered_model_structure}. Condition $(1)$, $(2)$ and $(3)$ are obvious. The condition $(4)$ is lemma \ref{lem:pi_preserve_fibration}, the fifth is lemma \ref{lem:X_to_pi(X)_is_a_trivial_fibration}. 
\end{proof}

One can show the analogue result for bistratified presheaves:

\begin{theorem}
\label{theo:local_model_structure_on_bistratified_presheave}
Let $I$ be an interval object on $\ctPsh{A}$, and $T$ a set of monomorphisms. 
There exists an $(I,\Gamma T)$-local model structure on $\ctPsh{A}$.
\end{theorem}

Let $S$ be saturation set for $\relativtPsh{A}{B}$. Theorem \ref{theo:local_model_structure_on_stratified_presheave} implies the existence of a $(I,\Gamma(T\cup S))$-local structure on $\relativtPsh{A}{B}$. We will use this structure and the adjonction: 
\[\begin{tikzcd}
	{\relativtPsh{A}{B}} && {\relativtPsh{A}{B}_{\Gamma S}}
	\arrow[""{name=0, anchor=center, inner sep=0}, "S", curve={height=-6pt}, from=1-1, to=1-3]
	\arrow[""{name=1, anchor=center, inner sep=0}, "i", curve={height=-6pt}, from=1-3, to=1-1]
	\arrow["\dashv"{anchor=center, rotate=-90}, draw=none, from=0, to=1]
\end{tikzcd}\]
to create a model structure on $\relativtPsh{A}{B}_{\Gamma S}$.

\begin{lemma}
\label{lem:X_to_iSX_is_a_weak_equivalence}
The morphism $X\to iRX$ is obtained as a sequence of pushout along morphisms of $\Gamma S$, and is then a weak equivalence.
\end{lemma}
\begin{proof}
Straightforward.
\end{proof}

\begin{lemma}
\label{lem:S_commutes_with_finite_product}
The functor $R$ commutes with finite products. 
\end{lemma}
\begin{proof}
Let $X$ and $Y$ be two objects of $\Psh{t_BA}$. We will show that $iR(X\times Y) \cong iR(X)\times iR(Y)$. First, we have a morphism $p:iR(X\times Y)\to iR(X)\times iR(Y)$. For the other way, remark that $X\times Y\to R(X)\times R(Y)$ is also obtained  as a sequence of pushout along morphisms of $\Gamma S$. We then have a lift in the following diagram.
\[\begin{tikzcd}
	{X\times Y} & {R(X\times Y)} \\
	{R(X)\times R(Y).}
	\arrow[from=1-1, to=2-1]
	\arrow[from=1-1, to=1-2]
	\arrow["q"', dotted, from=2-1, to=1-2]
\end{tikzcd}\]
Both $pq$ and $qp$ are identities on the underlying presheaf on $A$, and so are equal to identities. 
\end{proof}

\begin{theorem}
\label{theo:local_model_structure_on_saturated_stratified_presheave}
There exists a $(I,\Gamma T)$-local model structure on $\relativtPsh{A}{B}_{\Gamma S}$. 
\end{theorem}
\begin{proof}
We apply theorems \ref{theo:local_model_structure_on_stratified_presheave}. Conditions $(1)$ and $(3)$ are obvious. Condition $(2)$ is \ref{lem:S_commutes_with_finite_product}, and condition $(5)$ is \ref{lem:X_to_iSX_is_a_weak_equivalence}. 
\end{proof}

We also have the analogue for bistratified presheaves.

\begin{theorem}
\label{theo:local_model_structure_on_saturated_bistratified_presheave}
Let $B\subset C\subset ob(A_{>0})$,  $I$ be an interval object on $\relativctPsh{A}{B}{C}$,  $T$ a set of monomorphisms, and $S$ a saturation set for $\relativctPsh{A}{B}{C}$.
There exists a $(\Gamma T,I)$-local model structure on $\relativctPsh{A}{B}{C}_{\Lambda S}$. 
\end{theorem}

\section{Technical results}

\subsection{Zigzag of acyclic cofibrations} 
Let $C$ be a category endowed with a model structure where all objects are cofibrant.

\begin{definition}
Let $i:A\to B$ and $i':A'\to B'$ be two cofibrations. A \textit{zigzag of acyclic cofibration} between $i$ and $i'$, denoted $i\leftrightsquigarrow i'$ is a zig zag  in the category of arrows such that all the horizontal maps are acyclic cofibrations, and all the vertical maps are cofibrations. 
\end{definition}

\begin{lemma}
Let $i$ and $j$ be two cofibrations, and $f:X\to Y$ a fibration between fibrant objects. Suppose that we have a morphism in the category of arrows $i\to j$ which is a pointwise  trivial cofibration. Then, if $j$ has the left lifting property against $f$, so has $i$.
\end{lemma}
\begin{proof}
We consider a diagram of the following shape:
\[\begin{tikzcd}
	A & {A'} & X \\
	B & {B'} & Y.
	\arrow["i", from=1-1, to=2-1]
	\arrow["\sim", from=1-1, to=1-2]
	\arrow["\sim"', from=2-1, to=2-2]
	\arrow["j"', from=1-2, to=2-2]
	\arrow[from=1-3, to=2-3]
	\arrow[curve={height=-18pt}, from=1-1, to=1-3]
	\arrow[curve={height=18pt}, from=2-1, to=2-3]
	\arrow["{l_0}"', dotted, from=2-2, to=2-3]
	\arrow["{l_1}"{description}, dotted, from=1-2, to=1-3]
	\arrow["{l_2}"{description}, dotted, from=2-2, to=1-3]
\end{tikzcd}\]
We construct one after the other lifting $l_0$, $l_1$ and $l_2$.
\end{proof}

\begin{lemma}
Let $i$ and $j$ be two cofibrations, and $f:X\to Y$ a fibration between fibrant objects. Suppose that we have a morphism in the category of arrows $i\to j$ which is a pointwise  trivial cofibration. Then, if $i$ has the right lifting property against $f$, so has $j$.
\end{lemma}
\begin{proof}
We consider a diagram of the following shape:
\[\begin{tikzcd}
	A & {A'} &&& X \\
	B & {B\coprod_A A'} \\
	&& {B'} && Y.
	\arrow[from=1-1, to=2-1]
	\arrow["\sim"', from=1-1, to=1-2]
	\arrow["\sim", from=2-1, to=2-2]
	\arrow["\sim"{description}, from=2-2, to=3-3]
	\arrow["\sim"', curve={height=6pt}, from=2-1, to=3-3]
	\arrow[curve={height=-6pt}, from=1-2, to=3-3]
	\arrow[from=1-2, to=2-2]
	\arrow[from=1-2, to=1-5]
	\arrow[from=3-3, to=3-5]
	\arrow[from=1-5, to=3-5]
	\arrow["{l_0}"{description}, dotted, from=2-2, to=1-5]
	\arrow["{l_1}"{description}, dotted, from=3-3, to=1-5]
\end{tikzcd}\]
We construct one after the other lifting $l_0$, $l_1$.
\end{proof}

\begin{prop}
\label{prop:lifting_property_zigzag_of_acyclic_cofibration}
Let $f$ be a fibration between fibrant objects, $i$, $j$ two cofibrations such that there exists a zigzag of acyclic cofibrations $i\leftrightsquigarrow j$. Then $f$ has the right lifting property against $i$ if and only if it has the right lifting property against $j$. 
\end{prop} 
\begin{proof}
This is a direct consequence of the last two lemmas.
\end{proof}

\subsection{Lemmas for marked simplicial sets}

We conclude this section by defining some particular anodyne extensions. 

\begin{definition} Let $0<k<n$ be two integers. We define several marked simplicial sets:
\begin{itemize}[leftmargin=* ,parsep=0cm,itemsep=0cm,topsep=0cm]
\item $\Delta^k[n]\oslash K := (K\times\Delta[n],M_{K\oslash \Delta^k[n]})$. The set $M_{K\oslash \Delta^k[n]}$ is the smaller marking that includes all simplices $(v,s^px)$ such that 
$v_{[p-1,p,p+1]}$ is equal to $[k-1,k,k+1]$ or $[k,k,k+1]$.
\item $\Delta^0[n] \oslash K := (K\times\Delta[n],M_{K\oslash \Delta[n]^0})$. The set $M_{K\oslash \Delta^0[n]}$ is the smaller marking that includes all simplices $(v,s^px)$ such that 
$v_{[p,p+1]}$ is equal to $[0,1]$.
\item $\Delta^k[n]\invoslash K  := (K\times\Delta[n],M_{K\invoslash \Delta^k[n]})$. The set $M_{K\invoslash \Delta^k[n]}$ is the smaller marking that includes all simplices $(v,s^{p-1}x)$ such that 
$v_{[p-1,p,p+1]}$ is equal to $[k-1,k,k+1]$ or $[k-1,k,k]$.
\item $\Delta^n[n]\invoslash K := (K\times\Delta[n],M_{K\invoslash \Delta[n]^0})$. The set $M_{K\invoslash \Delta^n[n]}$ is the smaller marking that includes all simplices $(v,s^{p-1}x)$ such that 
$v_{[p-1,p]}$ is equal to $[n-1,n]$.
\end{itemize}
For $k<n$, we define  the regular inclusion
$$\Lambda^k[n]\oslash K \to_e (\Delta^k[n]\oslash K),$$ and for $0<k$, the regular inclusion $$\Lambda^k[n]\oslash K \to_e (\Delta^k[n]\oslash K).$$ 
\end{definition}

\begin{construction}
Let $0\leq k<n$ be two integers, $L$ a simplicial set. 
Let $w:=(v,x)$ be a non degenerate $m$-simplex of $\Delta[n]\times L$ where $v$ includes $k$ but not $k+1$.  Let $a$ and $b$ be two integers such that $v_c = k$ if and only if $a< c \leq b$.
For $p\in [a,b]$, we define $m$-simplices of $\Delta[n]\times L$:
$$\begin{array}{rcll}
v^p_q& = & v_q &\mbox{ if $q\leq p$},\\
v^p_q& = & k+1 &\mbox{ if $p< q \leq  b$},\\
v^p_q& = & v_{q} &\mbox{ if $b< q$},\\
\end{array}$$
and  for $p\in]a,b]$, $(m+1)$-simplices of $\Delta[n]\times L$
$$\begin{array}{rcll}
\bar{v}^p_q& = & v_k &\mbox{ if $q\leq p$},\\
\bar{v}^p_q& = &k+1 &\mbox{ if $p< q \leq  b+1$},\\
\bar{v}^p_q& = & v_{k-1} &\mbox{ if $b+1< k$}.\\
\end{array}$$

We define  $w^p:=(v,x)$ and $\bar{w}^p:=(\bar{v}^p,s^px)$.
We then have
$$\begin{array}{rcll}
d^q \bar{w}^p &= &\overline{d^q w}^{p-1}  &\mbox{ if $k\leq p$},\\
d^p \bar{w}^p& = & w^{p-1} &\mbox{ if $a< p$},\\
d^{p+1} \bar{w}^p& = & w^{p} &\mbox{ if $a< p$},\\
d^q \bar{w}^p &= &\overline{d^{q} w}^p  &\mbox{ if $k+1<p$}.\\
\end{array}$$ 
The fact that $w$ is non degenerate implies that each $w_p$ and $\overline{w}^p$ are non degenerate.
\end{construction}

\begin{lemma}
\label{lemma:all_simplexe_are_of_the_desired_shape}

Let $0\leq k<n$ be two integers, $K\to L$ a cofibration between simplicial sets.  We consider the cofibration:
 $$\Lambda_k[n]\times L\cup \Delta[n]\times K\to \Delta[n]\times L.$$
Let $w:=(v,x)$ be a non degenerate simplex of $L\times \Delta[n]$ which is not present in the domain. Then one and only one of the following statements is true.
\begin{enumerate}
\item There exists a unique  $v'$ and a unique $p$ such that $(v',x)$ is non degenerate and $w=((v')^p,x)$.
\item There exists a unique $w'$ such that $\overline{(w')}^{p} = w$.
\end{enumerate}
\end{lemma}
\begin{proof}
We remark that $w:=((v')^p,x')$ if and only if $v' = v[k+1/k]$.

\textbf{Case $k\notin v$.} This implies that $w$ cannot fulfill the second condition and that $(v',x)$ is non degenerate. Finally, one has $w=((v')^{q-1},x)$ where $q$ is the minimal integer such that $v_q=k+1$.
\textbf{Case $k\in v$ and $x$ non degenerate in $p$,} where $p$ is the maximal integer such that $v_p=k$. The simplex $w$ cannot fulfill the second condition and  $w=((v')^{p},x)$.
\textbf{Case $k\in v$ and $x$ degenerate in $p$.} We then have $v_{p+1}=k+1$, but $v'$ is then degenerate on $p$, so is $(v',x)$. The simplex $w$ cannot fulfill the first condition. But $v' = \overline{v}^p$, and $w =( \overline{v},s^pd^px)^p$.
\end{proof}

\begin{lemma}
\label{lemma:oslash}
Let $0\leq k<n$ be two integers, $K\to L$ a cofibration between simplicial sets. 
The cofibration 
$$\Delta^k[n]\oslash K \cup \Lambda^k[n]\oslash L\to \Delta^k[n]\oslash L$$
is an anodyne extension.
\end{lemma}
\begin{proof}
We proceed by induction. 
We define $A_0$ as the domain of the cofibration. Suppose construct  a regular sub marked simplicial set $A_{m-1}$, that includes all simplices of shape $\bar{w}^p$ where $w:=(v,x)$ is a simplex of dimension $<m$, such that $v$ includes $k$ and not $k+1$. \newline
Let $w:=(v,x)$ be a non degenerate $m$-simplex, non present in $A_{m-1}$, such that $v$ includes $k$ and not $k+1$. Let $a$ and $b$ be two integers such that $v_c = k$ if and only if $a< c \leq b$. We define a sequence of regular inclusions:
$$A_{m-1}:=B_b \subset B_{b+1}\subset...\subset B_{a+1}=:A_{m-1}^{w}$$
such that $B_{p} := B_{p+1}\cup \bar{w}^p$. Let $d:\Delta[n]\to \Delta[m]$ be an inclusion that reaches $\{p-1,p,p+1\}\cap[n]$. The simplex $d^*\bar{w}^p$ is then either in the domain of the cofibration, or of shape $\bar{w'}^q$ for $w'$ a simplex of dimension strictly inferior to $m$. In both case, it is thin. 
The inclusion $B_{p+1}\to B_p$ then fits in a pushout diagram:
\[\begin{tikzcd}
	{\Lambda^p[n]} & {B_{p+1}} \\
	{\Delta^p[n]} & {B_p.}
	\arrow[from=1-1, to=1-2]
	\arrow[from=1-1, to=2-1]
	\arrow[from=2-1, to=2-2]
	\arrow[from=1-2, to=2-2]
	\arrow["\lrcorner"{anchor=center, pos=0.125, rotate=180}, draw=none, from=2-2, to=1-1]
\end{tikzcd}\]

We then define $A_m$ as the reunion of all $A_m^w$ where $w$ is a $m$-simplex satisfyingly the previous condition, and $A_{\infty} := \Colim_{A_m}$. The lemma \ref{lemma:all_simplexe_are_of_the_desired_shape} implies that $A_\infty$  is equal to $\Delta^k[n]\oslash L$.
\end{proof}

\begin{lemma}
\label{lemma:oslash_saturation_one_way}
Let $(X,tX)$ be a marked simplicial set, $0<k<n$ two integers and a morphism $f:\Delta^k[n]\oslash L \to (X,M)$. Let $w:=(v,x)$ be an object of $\Delta^k[n]\oslash L$ that  includes $k$ and $k-1$,
$$f(v[k+1/k],x)\in tX \Rightarrow f(v,x)\in tX.$$
\end{lemma}
\begin{proof}
Let $w:=(v,x)$ be a simplex fulfilling the desired condition.
Let $a$ and $b$ be two integers such that $v_c = k$ if and only if $a< c \leq b$.
Lemma \ref{lemma:all_simplexe_are_of_the_desired_shape} implies that $w$ is either of the form $\overline{w}^p$ or of the form $w^p$ with $a<p$. In the fist case, $f(v,x)$ is thin. For the second case, we show by decreasing induction on $p$ that $w^p$ is thin. The initialization, i.e the case $p=b$ corresponds to the hypotheses. Now, let's remark that if $a+1<q$, a careful analysis off all the cases shows that $d^lf(\overline{w}^q)$ is thin for all $l\neq q$. This implies that $f(w^{p}) := d^{p+1} f(\overline{w}^p)$ is thin. This proves the result.
\end{proof}

\begin{lemma}
\label{lemma:oslash_saturation_one_way_variation}
Let $(X,M)$ be a marked simplicial set, $0<k<n$ two integers and a morphism $f:\Delta^k[n]\oslash L \to (X,M)$. Suppose furthermore that for all simplices $(v,s^{p}x)$ such that $(k-1)\notin v$ and $v_{[p-1,p,p+1]}=[k,k,k+1]$, 
$f(v,s^{p}x)$ is in $M$. Then if
$w:=(v,x)$ is an object of $\Delta^k[n]\oslash L$ that  includes $k-1$,
$$f(v[k+1/k],x)\in tX \Rightarrow f(v,x)\in tX.$$
\end{lemma}
\begin{proof}
Let $w:=(v,x)$ be a simplex fulfilling the desired conditions.
If $v$ includes $k$, the last lemma proves the assertion. We then suppose that $v$ doesn't include $k$. 
Lemma \ref{lemma:all_simplexe_are_of_the_desired_shape} implies that 
there exists a unique  $v'$ and a unique $p$ such that $(v',x)$ is non degenerate and $w=((v')^p,x)$. The last lemma implies that  $f((v')^{p+1},x)$ is thin. 
A careful analysis off all the cases shows that $d^lf(\overline{w}^{p})$ is thin for all $l\neq p$. This implies that $f(w^{p}) := d^{p+1} f(\overline{w}^p)$ is thin. 
\end{proof}

Similarly, we can prove the  next three lemmas:

\begin{lemma}
\label{lemma:antioslash}
Let $0< k\leq n$ be two integers, $K\to L$ a cofibration between simplicial sets. 
The cofibration 
$$\Delta^k[n]\invoslash K \cup \Lambda^k[n]\invoslash L\to \Delta^k[n]\invoslash L$$
is an anodyne extension.
\end{lemma}

\begin{lemma}
\label{lemma:antioslash_saturation_one_way}
Let $(X,tX)$ be a marked simplicial set, $0<k<n$ two integers and a morphism $f:\Delta^k[n]\oslash L \to (X,M)$. Let $w:=(v,x)$ be an object of $\Delta^k[n]\oslash L$ that  includes $k$ and $k+1$,
$$f(v[k-1/k],x)\in tX \Rightarrow f(v,x)\in tX.$$
\end{lemma}

\begin{lemma}
\label{lemma:antioslash_saturation_one_way_variation}
Let $(X,M)$ be a marked simplicial set, $0<k<n$ two integers and a morphism $f:\Delta^k[n]\oslash L \to (X,M)$. Suppose furthermore that for all simplices $(v,s^{p}x)$ such that $(k-1)\notin v$ and $v_{[p-1,p,p+1]}=[k,k,k+1]$, 
$f(v,s^{p}x)$ is in $tX$. Then if
$w:=(v,x)$ is an object of $\Delta^k[n]\oslash L$ that  includes $k+1$,
$$f(v[k-1/k],x)\in tX \Rightarrow f(v,x)\in tX.$$
\end{lemma}

\begin{lemma}
\label{lemma:oslash_saturation_two way}
Let $(X,tX)$ be a marked simplicial set, $0<k<k+1<n$ two integers and $f:(\Delta[n]\times L,\overline{\emptyset}) \to (X,tX)$ a morphism. Suppose furthermore that $f$ lifts to $\Delta^k[n]\oslash L$ and $\Delta^{k+1}[n]\invoslash L$. Then, if $v$ is a simplex that includes $k-1$ and $k+2$
$$f(v[k+1/k],x)\in tX \Leftrightarrow f(v,x)\in tX  \Leftrightarrow f(v[k/k+1],x)\in tX.$$
\end{lemma}
\begin{proof}
This is a direct consequence of \ref{lemma:oslash_saturation_one_way_variation} and \ref{lemma:antioslash_saturation_one_way_variation}.
\end{proof}

\begin{lemma}
\label{lemma:oslash_saturation_two way_case_extremum}
Let $(X,tX)$ be a marked simplicial set, $0<n$ an integer and $f:(\Delta[n]\times L,\overline{\emptyset}) \to (X,tX)$ a morphism. If that $f$ lifts to $\Delta^0[n]\oslash L$ Then, 
$$f(v[1/0],x)\in tX \Leftrightarrow f(v,x)\in tX  \Leftrightarrow f(v[0/1],x)\in tX.$$

If $f$ lifts to $\Delta^n[n]\invoslash L$. Then,
$$f(v[n/n-1],x)\in tX \Leftrightarrow f(v,x)\in tX  \Leftrightarrow f(v[n/n-1],x)\in tX.$$
\end{lemma}
\begin{proof}
We will prove only the first assertion, the second one being symmetric. 
Let $w:=(v,x)$ be any simplex such that $v$ includes $0$ but not $1$. Let $a$ be the maximal integer such that $v_a=0$. According to lemma  \ref{lemma:all_simplexe_are_of_the_desired_shape}, it is enough to show that  for all $p<a$, $w^p$ is thin if and only if $w^{p+1}$ is thin. The simplex $\overline{w}^{p+1}$ is thin, and hypotheses imply that for all $d:\Delta[k]\to \Delta[n]$ that reaches $p$ and $p+1$, $d^*w$ is thin. This implies that  $w^p$ is thin if and only if $w^{p+1}$ is thin, which concludes the proof.
\end{proof}

\subsection{Lemmas for bimarked simplicial sets}

\begin{lemma}
Let $n,m$ be two non null integers. The morphism
$$\Delta[n]_c\times \partial\Delta[m] \cup \Delta[n]\times\Delta[m] \to \Delta[n]_c\times \Delta[m]$$
is an identity in the category of bimarked simplicial sets.
\end{lemma}
\begin{proof}
If $m\geq n$, this is an identity at the level of bistratified sets. We then suppose $m<n$.
Let $M$ be the set of cartesian simplices of $\Delta[n]_c\times \partial\Delta[m] \cup \Delta[n]$.  Let $v$ be a degeneracy of the non trivial $m$-simplex of $\Delta[m]$. We then have to show that $(i_n,v)$ is in $M$. Let $p$ be the maximal integer such that $v_p=0$.
Let $w:=(s^{p+1}i_n,s^pv)$. The simplex $d^pw$  is thin, and so in $M$, and the simplex $d^{p+2}w$ is in $M$. The cartesian thinness extension implies that $(i_n,v)$ is in $M$.
\end{proof}

\begin{lemma}
Let $n,m$ be two non null integers. 
The morphism
$$\Delta[n]_c\times \Delta[m] \cup \Delta[n]\times\Delta[m]_c \to \Delta[n]_c\times \Delta[m]_c$$
is an identity in the category of bimarked simplicial sets.
\end{lemma}
\begin{proof}
If $m\neq n$, this is an identity at the level of bistratified sets. We then suppose $m=n$.
Let $M$ be the set of cartesian simplices of $\Delta[n]_c\times \Delta[n] \cup \Delta[n]\times\Delta[n]_c$.  We then have to show that $(i_n,i_n)$ is in $M$. Let $w:=(s^{0}i_n,s^1i_n)$. Simplices $d^0w$  and $d^{2}w$ are in $M$. The cartesian thinness extension implies that $(i_n,i_n)$ is in $M$.
\end{proof}

\begin{lemma}
Let $n,m$ be two non null integers. The morphism
$$\Delta[n]_c\times \Delta[m] \cup \Delta[n]\times\Delta[m]_t \to \Delta[n]_c\times \Delta[m]_t$$
is an identity in the category of bimarked simplicial sets.
\end{lemma}
\begin{proof}
If $m\neq n$, this is an identity at the level of bistratified sets. We then suppose $m=n$.
Let $M$ be the set of cartesian simplices of $\Delta[n]_c\times \Delta[n] \cup \Delta[n]\times\Delta[n]_t$.  We then have to show that $(i_n,i_n)$ is in $M$. Let $w:=(s^{0}i_n,s^1i_n)$. The simplex $d^0w$ is in $M$, and the simplex  $d^{2}w$ is thin. The cartesian thinness extension implies that $(i_n,i_n)$ is in $M$.
\end{proof}

\begin{prop}
\label{prop:behavious_of_times_with_cofibration}
Let $n$ be an non null integer. Let $K\to L$ be a cofibration which is surjective  on $0$-simplices. The map 
$$\Delta[n]_c\times K\cup \Delta[n]\times L\to \Delta[n]_c\times L$$ 
is an identity in the category of bimarked simplicial sets.
\end{prop}
\begin{proof}
This is a direct consequence of the three last lemmas.
\end{proof}

\nocite{*}
\bibliography{Dualities_in_the_complicial_model}{}
\bibliographystyle{plain}

\end{document}